\documentclass{amsart}
\usepackage{graphicx} 
\usepackage{amsmath}
\usepackage{amssymb}
\usepackage{array}
\usepackage[english]{babel}

\usepackage[tracking=true]{microtype}

\usepackage[shortcuts]{extdash}
\usepackage{yhmath,amsmath}
\usepackage{mathtools}
\usepackage{caption}
\usepackage{subcaption}
\usepackage{tikz-cd} 
\usepackage[utf8]{inputenc}
\usepackage{enumitem}
\usepackage[hang,flushmargin]{footmisc}

\usepackage{biblatex}

\usepackage[dvipsnames]{xcolor}
\usepackage{hyperref}
\usepackage[nameinlink]{cleveref}
\hypersetup{
colorlinks,
linkcolor=Blue,
citecolor=Blue,
urlcolor=Blue
}

\newtheorem{theorem}{Theorem}[section]
\newtheorem{thm*}{Theorem}
\newtheorem{lemma}[theorem]{Lemma}
\newtheorem{definition}[theorem]{Definition}
\newtheorem{corollary}[theorem]{Corollary}
\newtheorem{proposition}[theorem]{Proposition}
\newtheorem{example}[theorem]{Example}
\newtheorem{remark}[theorem]{Remark}
\let\oldexample\example
\renewcommand{\example}{\oldexample\normalfont}
\let\oldremark\remark
\renewcommand{\remark}{\oldremark\normalfont}

\newtheorem{subdefinition}{Definition}[theorem]

\def\s{\ensuremath\sigma}
\def\endomorphisms{\ensuremath\sigma_1,\dots,\sigma_n}
\def\basicset{\ensuremath\{\endomorphisms\}}
\def\generalelemt{\ensuremath\sigma_1^{k_1}\cdots\sigma_n^{k_n}}
\def\tauequals{\ensuremath \tau=\generalelemt}
\def\sk{\ensuremath [\sigma]_k}
\def\X{\ensuremath \mathcal{X}}
\def\Y{\ensuremath \mathcal{Y}}

\def\D{\ensuremath \mathcal{D}}
\def\K{\ensuremath \mathcal{K}}
\def\G{\ensuremath \mathcal{G}}
\def\H{\ensuremath \mathcal{H}}
\def\I{\ensuremath \mathbb{I}}
\def\V{\ensuremath \mathbb{V}}

\def\ord{\ensuremath \operatorname{ord}}
\def\id{\ensuremath \operatorname{id}}
\def\Hom{\ensuremath \operatorname{Hom}}
\def\mult{\ensuremath \operatorname{mult}}
\def\unit{\ensuremath \operatorname{unit}}
\def\inv{\ensuremath \operatorname{inv}}
\def\hgt{\ensuremath \operatorname{hgt}}

\usetikzlibrary{positioning,fit,calc} 
\tikzset{block/.style={draw, thick, text width=3cm, minimum height=1.5cm, align=center},   
line/.style={-latex} 
}  
\usetikzlibrary{arrows}

\usepackage{csquotes}

\usepackage[hmargin=3cm, vmargin=3.2cm, twoside]{geometry}
\begin{document}

\title{Dimension Polynomials for Affine\linebreak\ Partial Difference Algebraic Groups}

\author{Orla McGrath}
\address{University of Leeds}
\email{O.P.McGrath@leeds.ac.uk}


\date{April 14, 2026}

\thanks{This work was financially supported by EPSRC via the Doctoral Training Partnership Studentship, reference 2940851.}

\begin{abstract}
We develop the theory of difference algebraic groups in the case where we have finitely many pairwise commuting difference operators. We show that the defining ideal of a difference algebraic group is finitely generated as a difference ideal, and this result allows us to prove the existence of a dimension polynomial for any partial difference algebraic group.
\end{abstract}
\maketitle
\section*{Introduction}

While an affine algebraic group is a subgroup of the general linear group that is defined by polynomial equations, an affine difference algebraic group is a subgroup of the general linear group that is defined by difference polynomial equations. 

Firstly, for the benefit of the reader with no background in difference algebra, we will state a special case of our main results in the language of algebraic symbolic dynamics. For an affine group scheme $\G$ of finite type (over some field), we can consider the $n$\=/dimensional group shift $\G^{\mathbb{N}^n}$ equipped with the shift maps 
$\s_j\colon \G^{\mathbb{N}^n}\rightarrow \G^{\mathbb{N}^n}$ given by $\s_j((g_\alpha)_{\alpha\in\mathbb{N}^n})=(g_{\alpha+e_j})_{\alpha\in\mathbb{N}^n}$, where $e_j$ is the $j$-th standard basis vector for $1\leq j\leq n$.
An affine difference algebraic group is then analogous to a closed subgroup scheme $G$ of $\G^{\mathbb{N}^n}$ which is closed under each of the shift maps. For each $i\in\mathbb{N}$, let $G[i]$ be the scheme theoretic image of $G$ in $\G^{\mathbb{N}^n_{\leq i}}$, where $\mathbb{N}^n_{\leq i}=\{(\alpha_1,\dots,\alpha_n)\in \mathbb{N}^n\ |\ \alpha_1+\dots+\alpha_n\leq i\}$. In this setup, the main results in this paper are as follows. There exists $m\geq 1$ and some closed subgroup $\H\leq \G^{\mathbb{N}^n_{\leq m}}$ such that $G$ is fully determined by $\H$ and the fact that $G$ is closed under each of the shift maps. That is, \begin{align*}
    G=\{(g_\alpha)_{\alpha\in\mathbb{N}^n}\ |\ (g_\alpha)_{\alpha\in \beta+\mathbb{N}^n_{\leq m}}\in \H\ \text{for all}\ \beta\in \mathbb{N}^n\}.
\end{align*}
Further, there exists a numerical polynomial $\phi(t)\in\mathbb{Q}[t]$ of degree less than or equal to $n$ such that for large enough $i\in\mathbb{N}$, $\phi(i)=\dim(G[i])$.

Now we will give a brief summary of the concepts and results in this paper in the language of difference algebra.
A difference ring ($\s$\=/ring) is a commutative ring along with a set $\s=\basicset$ of pairwise commuting endomorphisms. When $n=1$, this is an ordinary $\s$\=/ring, and when $n>1$, this is a partial $\s$\=/ring. Many standard concepts in commutative algebra have natural difference algebraic analogues. For example, a $\s$\=/field is a $\s$\=/ring where the underlying ring is a field, and a $\s$\=/ideal $I$ of a $\s$\=/ring $R$ is an ideal $I$ of $R$ that is closed under each of the endomorphisms in $\s$. Given a $\s$\=/ring $R$ and some subset $F\subseteq R$, we write $[F]$ for the smallest $\s$\=/ideal of $R$ containing $F$. 
Any composition of the basic endomorphisms $\basicset$ can be written $\s_1^{k_1}\cdots\s_n^{k_n}$ for $k_1,\dots,k_n\in \mathbb{N}$, and we call such a composition a transform of the  $\s$\=/ring $R$. We assign a natural order to the transforms of a $\s$\=/ring, where $\ord(\s_1^{k_1}\cdots\s_n^{k_n})=k_1+\cdots+k_n$.
A $\s$\=/polynomial over a $\s$\=/field $k$ in the $\s$\=/indeterminates $\{y_1,\dots,y_s\}$ is a polynomial over $k$ in the indeterminates $\{\s_1^{k_1}\cdots\s_n^{k_n}(y_i)\ |\ 1\leq i\leq s,\ k_1,\dots,k_n\in \mathbb{N}\}$. We get an induced order on the $\s$\=/polynomials over $k$, where the order of an indeterminate is the order of the corresponding transform, and the order of a $\s$\=/polynomial is the highest order of the indeterminates it contains.

Then, an affine difference algebraic group ($\s$\=/algebraic group) over a $\s$\=/field $k$ is a group defined by $\s$\=/polynomials over $k$. For example, let $k$ be a $\s$\=/field where $\s=\{\s_1,\s_2\}$. Then for any $k$\=/$\s$\=/algebra $R$,  \begin{align}\label{ex: intro example symbolic dynam}
    G(R)=\{r\in R\ |\ r+\s_1(r)+\s_2(r)=0\}
\end{align}
is a group under addition, and hence the functor $G$ from $k$\=/$\s$\=/$\operatorname{Alg}$ to $\operatorname{Sets}$, $R\rightsquigarrow G(R)$ is a $\s$\=/algebraic group over $k$, defined by the $\s$\=/polynomial $y+\s_1(y)+\s_2(y)$. This is a certain analogue to the motivating example in dynamical systems from the introduction of \cite[page x]{schmidt1995dynamical}, which is due to Ledrappier \cite{ledrappier1978unchamp}.

Similarly to the dual correspondence between affine algebraic groups over a field $k$ and finitely generated $k$\=/Hopf algebras, there is a dual correspondence between $\s$\=/algebraic groups $G$ over a $\s$\=/field $k$ and finitely $\s$\=/generated $k$\=/$\s$\=/Hopf algebras $k\{G\}$. Then $\s$\=/closed subgroups of $G$ correspond to $\s$\=/Hopf ideals of $k\{G\}$.

We prove that any $\s$\=/algebraic group can be defined as a $\s$\=/closed subgroup of a general linear $\s$\=/algebraic group by some $\s$\=/Hopf ideal $\I(G)$ of $\s$\=/polynomials, and generalise the result from the ordinary case \cite[Theorem 4.1, page 532]{wibmer2022finiteness} that any such defining $\s$\=/Hopf ideal is finitely $\s$\=/generated. Due to the dual correspondence with $\s$\=/Hopf structure and $\s$\=/group structure, this $\s$\=/finiteness of the defining $\s$\=/ideals extends to any $\s$\=/Hopf ideal in a finitely $\s$\=/generated $k$\=/$\s$\=/Hopf algebra.

\begin{thm*}[Corollary \ref{cor: hopf ideal fin gen}]\label[theorem]{the: ideal intro}
    Any $\s$\=/Hopf ideal of a finitely $\s$\=/generated $k$\=/$\s$\=/Hopf algebra is finitely $\s$\=/generated.
\end{thm*}

 One application of \Cref{the: ideal intro} is that any $k$\=/$\s$\=/Hopf subalgebra of a finitely $\s$\=/generated $k$\=/$\s$\=/Hopf algebra is finitely $\s$\=/generated, and this allows us to prove the existence of the quotient of a $\s$\=/algebraic group by any normal $\s$\=/closed subgroup.

In fact, we prove something stronger than \Cref{the: ideal intro}. The ideals $\I(G[i])$ of all defining $\s$\=/polynomials in $\I(G)$ of order up to and including $i\in\mathbb{N}$ form an ascending chain of Hopf ideals. We prove that these chains of Hopf ideals eventually stabilise with respect to the shift maps, in the sense that \begin{align}\label{eq: ideal gen prop intro}
    \I(G[i+1])=(\I(G[i]),\s_1(\I(G[i])),\dots,\s_n(\I(G[i])))
\end{align}
for large enough $i\in\mathbb{N}$. Each $\I(G[i])$ is the defining ideal for an algebraic group $G[i]$, and we call the sequence of closed subgroups $(G[i])_{i\in\mathbb{N}}$ the Zariski closures of the $\s$\=/algebraic group $G$ (with respect to the embedding into the general linear $\s$\=/algebraic group). 
 The fact that the ideals $\I(G[i])$ eventually stabilise with respect to the endomorphisms (\ref{eq: ideal gen prop intro}) can be used to prove that the dimensions of the Zariski closures also stabilise in some nice way, that is, we can find a dimension polynomial. 

\begin{thm*}[Corollary \ref{cor: partial dim poly zariskis}]\label[theorem]{the: dim poly intro}
    Let $G$ be a $\s$\=/algebraic group, and let $(G[i])_{i\in\mathbb{N}}$ be the Zariski closures of $G$ (with respect to some given embedding). Then there exists a numerical polynomial $\phi(t)\in\mathbb{Q}[t]$ of degree less than or equal to $n$ such that for large enough $i\in\mathbb{N}$, $\phi(i)=\dim(G[i])$.
\end{thm*}

While the dimension polynomial for $G$ depends on the embedding we are working with, the degree and leading coefficient of a dimension polynomial for a $\s$\=/algebraic group $G$ are independent of the chosen embedding, and hence are invariants of the $\s$\=/algebraic group $G$. This allows us to define the $\s$\=/dimension, $\s$\=/type and typical $\s$\=/dimension of a $\s$\=/algebraic group (\Cref{def: diff type diff dim etc}).

A significant portion of this paper is dedicated to using induction to extend the result (\ref{eq: ideal gen prop intro}) from the ordinary case \cite[Theorem 4.1, page 532]{wibmer2022finiteness} to the partial case. Since for each $i\geq 1$, $\I(G[i-1])\subseteq\I(G[i])$, there is a (quotient) morphism of algebraic groups $\pi_i\colon G[i]\rightarrow{G[i-1]}$.
In both the ordinary and partial cases, the proof of (\ref{eq: ideal gen prop intro}) is conducted by studying the kernels $(H_{i})_{i\in\mathbb{N}}$ of these $\pi_i$ morphisms. Intuitively, $H_i$ stores the information in $G[i]$ of order exactly $i$. Each $H_{i}$ can be thought of as a closed subgroup of a product of $t_i$ general linear groups, where $t_i$ is the number of transforms of order exactly $i$. In the ordinary case, $t_i=1$ for every $i\in\mathbb{N}$ as the only transform of order $i$ is $\s^i$. However, in the partial case, the number of transforms of order $i$ grows with $i$, in fact $t_i=\binom{i+n-1}{n-1}$. In the ordinary case these kernels become isomorphic (up to some base change), but the main difficulty in the paper is that this is not necessarily true in the partial case. Instead, they form a difference structure with $n-1$ endomorphisms, allowing us to employ induction. However, the difference structure that they form is more general than the structure of Zariski closures, motivating the notion of a generalised $\s$\=/algebraic group.

Note that these results (\Cref{the: ideal intro,the: dim poly intro}) do not hold for general $\s$\=/varieties, one needs the group, and hence the corresponding Hopf structure. For example, consider the $\s$\=/ideal $[xy,x\s(y),x\s^2(y),\dots]$ of the ordinary $\s$\=/polynomial ring in two $\s$\=/indeterminates (which is a finitely $\s$\=/generated $k$\=/$\s$\=/algebra). This is not finitely $\s$\=/generated, and hence defines a $\s$\=/variety $X$ over $k$, where for any $k$\=/$\s$\=/algebra $R$, \begin{align*}
    X(R)=\{(r,s)\in R^2\ |\  r\s^i(s)=0 \text{ for all } i\in\mathbb{N}\},
\end{align*}
 which cannot be defined by finitely many $\s$\=/polynomials. 
As previously stated, the proof of (\ref{eq: ideal gen prop intro}) relies on the kernels of morphisms of algebraic groups, a construction which is not valid for morphisms of general varieties, hence this result fails without the group structure.

That (\ref{eq: ideal gen prop intro}) holds for large $i\in\mathbb{N}$ tells us that any $\s$\=/algebraic group $G$ can be determined (not necessarily uniquely) by an algebraic group. More precisely, a $\s$\=/closed subgroup $G$ of an algebraic group $\G$ is determined by its $m$-th order Zariski closure $G[m]$, where $m\in\mathbb{N}$ is such that (\ref{eq: ideal gen prop intro}) holds for all $i\geq m$. 
This provides a way to describe all additive and multiplicative $\s$\=/algebraic groups. For example, if $G$ is an additive (resp. multiplicative) $\s$\=/algebraic group, the algebraic group $G[m]$ is additive (multiplicative). One can therefore apply known classification results, for example, closed subgroups of $\mathbb{G}_{a}^s$ correspond to finite systems of linear equations, while closed subgroups of $\mathbb{G}_{m}^s$ correspond to sublattices of $\mathbb{Z}^s$.


The study of solutions of both differential and difference polynomials were introduced by Ritt in the 1930s (\cite{ritt1932differential} and \cite{doob1933systems}, \cite{ritt1939ideal} respectively). Picard-Vessiot theory, which is the Galois theory of differential field extensions \cite{kolchin1973differential}, provides a link between the study of differential algebra and the study of algebraic groups, and hence motivated further study of differential algebraic groups and their properties \cite{cassidy1972differential}, \cite{kolchin1985differential}, \cite{buium1992differential}. 
Difference algebraic groups are the difference analogue to differential algebraic groups, but have thus far not been studied as deeply as their differential counterpart. 
Both Cohn \cite{cohn1965difference} and Levin \cite{levin2008difference} provide comprehensive sources for the study of difference algebra.

A classical example of a solution to an ordinary difference equation is the gamma function \cite[Chapter 2, page 19]{beals2010special}, defined $\Gamma(z)=\int_{0}^\infty t^{z-1}e^{-t}dt$ for any complex number $z$ with strictly positive real part. This function has a range of wide applications across mathematical analysis, probability theory, and, due to the fact it extends the factorial function, combinatorics. Integrating by parts, one can see that $\Gamma(z+1)=z\Gamma(z)$, and hence, considering $\mathbb{C}(z)$ as an ordinary $\s$\=/field with $\s(f(z))=f(z+1)$, the gamma function is a natural solution to the difference equation $\s(f(z))=zf(z)$.

The ideas of the gamma function were generalised to higher dimensions by Barnes \cite{barnes1901thetheory}, for example via the Barnes double gamma function, $G(z;\tau)$, which is defined by an infinite product involving the gamma function for $\tau\in\mathbb{C}\backslash (-\infty,0]$ and $z\in\mathbb{C}$ \cite{alexanian2023onthebarnes}. Interestingly, the Barnes double gamma function is the solution to some partial difference equations. Considering $\mathbb{C}(z,\tau)$ as a partial $\s$\=/field with shift maps $\s_1(f(z,\tau))=f(z+1,\tau)$ and $\s_2(f(z,\tau))=f(z+\tau,\tau)$, the Barnes double gamma function is the solution to some difference equations involving the gamma function \begin{align*}
    \s_1(f(z,t))=\Gamma\left(\frac{z}{\tau}\right)f(z,\tau),\quad \s_2(f(z,\tau))=A(z,\tau)\Gamma(z)f(z,\tau)
\end{align*}
for some function $A$ of $z$ and $\tau$. This gives a natural example of where partial difference equations arise.

A classical and intuitive example of a partial difference field is the field of rational functions in $n$ variables, where each basic endomorphism $\s_j\in \s$ shifts the $j$-th variable (see \Cref{ex: difference field}). These kinds of functions arise naturally across a wide range of mathematics, as if we have functions depending on independent variables, these shift maps allow us to study the effect of changing just one variable and keeping the others stable.

The development of difference algebra led to the study of Galois theory in the difference setting, \cite{singer1997galois}, \cite{hardouin2008differential}, \cite{bachmayr2022algebraic}, where difference algebraic groups arise naturally as Galois groups \cite{divizio2014difference}, \cite{ovchinnikov2015sigma}.
This in turn has motivated further study into the area of ordinary difference algebraic groups \cite{wibmer2021almost}, \cite{wibmer2024etale}, \cite{bachmayr2022torsors}. 

Another place where difference algebraic groups arise is in the intersection of algebraic geometry and symbolic dynamics. That is, they are analogous to closed, shift-invariant subgroups of the full $n$\=/dimensional one-shift, where the alphabet is an algebraic group $\G$. Standard references in dynamical systems include \cite{kitchens1998symbolic} and \cite{schmidt1995dynamical}. Note that in symbolic dynamics, one would usually study the two-shift, which would be analogous to the case of inversive difference algebraic groups (the basic endomorphisms are invertible). We have already seen, in the beginning of this introduction, how the results in this paper can be stated in this language of symbolic dynamics. 
An analogous result to (\Cref{the: ideal intro}) has been found in in \cite{phung2022dynamical}, that is, Phung proves the descending chain condition for algebraic group subshifts of an algebraic group shift $\G^G$, where now the shift action is by a polycyclic-by-finite group $G$. While the results in this paper regard the case where the action is by $\mathbb{N}^n$, Phung's action is via a polycyclic-by-finite group, and hence neither of these results contain the other. Another area of dynamical systems relating to difference algebraic geometry is algebraic dynamics, the study of an algebraic variety with an endomorphism. Algebraic dynamical systems are studied in \cite{medvedev2014invariant}, and are connected to the model theory of difference fields in \cite{chatzidakis2008difference} and \cite{chatzidakis2008iidifference}, in particular studying how results in model theory relate to descent.

Difference algebra, and in particular difference algebraic groups, have also become of interest in model theory. ACFA (algebraically closed fields with automorphism) serves as the model theoretic companion to ordinary difference fields \cite{chatzidakis1999model}, and groups definable in ACFA (see \cite{chatzidakis1997groups}, \cite{kowalski2002anote}) have applications in number theory, such as in the proof of the Manin-Mumford conjecture \cite{hrushovski2001manin}, \cite{kowalski2007algebraic}.
Some study of fields with multiple commuting endomorphisms from the model theoretic direction has been started in \cite{hyttinen2023AEC}, however, it is a well-known result of Hrushovski that these partial difference fields do not have a model companion. Some work has been done in the case where the difference endomorphisms are not necessarily commuting, see \cite{moosa2014model}, however here, things complicate quickly and there is no descending chain condition for groups definable in this theory. Further, since many natural examples of difference algebra arise from the shift maps acting on parameters (see \Cref{ex: difference field}), it is intuitive to study commuting endomorphisms.
Due to these challenges, it has not yet been natural to study partial difference algebraic groups from the model-theoretic point of view. 

Partial difference algebraic groups are introduced and studied in the context of Galois theory in \cite{antieu2014galois}, in the case where the parameters are periodic (the $\s_j$'s are assumed to generated a fixed finite group). Linear groups over fields with many operators are studied using the Tannakian approach in \cite{kamensky2013tannakian}, and this approach is also taken in \cite{ovchinnikov2017tannakian} to study the notion of a linear difference algebraic group, where the action is now by a semigroup rather than commuting automorphisms. 
This paper will introduce partial affine difference algebraic groups following the approach of Wibmer in the ordinary case \cite{wibmer2022finiteness}. These results will allow for further study into partial affine difference algebraic groups.
We will put the main results of this paper into context with current work in difference algebra.

A particular open problem in difference algebra is finding classes of $\s$\=/ideals that satisfy the ascending chain condition. It is not the case that any $\s$\=/ideal of a finitely $\s$\=/generated $k$\=/$\s$\=/algebra is finitely $\s$\=/generated. Recall the example of the $\s$\=/ideal $[xy,x\s(y),x\s^2(y),\dots]$ of the $\s$\=/polynomial ring in two $\s$\=/indeterminates over an ordinary $\s$\=/field $k$. 
The current basis theorem for difference ideals \cite{levin2008difference} states that the ascending chain condition holds for perfect $\s$\=/ideals of a finitely $\s$\=/generated $k$\=/$\s$\=/algebra, where a $\s$\=/ideal $I$ of $R$ is perfect if for $a\in R$, transforms $\tau_1,\dots,\tau_s$ of $R$, $k_1,\dots,k_s\in\mathbb{N}$, the inclusion $\tau_1(a)^{k_1}\cdots\tau_s(a)^{k_s}\in I$ implies that $a\in I$.
It is conjectured \cite{hrushovski2004elem} that the basis theorem can be generalised to radical mixed $\s$\=/ideals of finitely $\s$\=/generated $k$\=/$\s$\=/algebras (a $\s$\=/ideal $I$ of a $\s$\=/ring $R$ is mixed if, given any $\s_j\in \s$, $r,s\in R$, $rs\in I$ implies that $r\s_j(s)\in I$), and in the same paper the conjecture is proven under certain additional assumptions. It is known that not every ascending chain of mixed $\s$\=/ideals of a finitely $\s$\=/generated $k$\=/$\s$\=/algebra is finite \cite{levin2015acc}, that is, the radical condition is necessary. Some work has been done towards this conjecture in the ordinary case \cite{wang2017monomial}, \cite{wang2018binomial}. Further work into aspects of finite difference generation of ordinary difference ideals has been conducted, for example, \cite{chen2017criteria} gives necessary criteria for normal binomial difference ideals (of the univariate difference polynomial ring) to have a finite Gröbner basis. 
Therefore, while work has been done in the ordinary case to extend the knowledge of classes of $\s$\=/ideals that satisfy the ascending chain condition, the partial case has not yet been investigated as thoroughly.
In this paper, we use algebro-geometric methods to generalise the result from \cite[Theorem 4.1]{wibmer2022finiteness}, that $\s$\=/Hopf ideals of finitely $\s$\=/generated $k$\=/$\s$\=/Hopf algebras are finitely $\s$\=/generated.

Dimension polynomials are of central importance in differential and difference algebra, playing an analogous role to Hilbert polynomials in commutative algebra. Existence of dimension polynomials for several differential and difference constructions are proven in the partial case in \cite{kondratieva1999differential} and \cite{levin2008difference}. More work has since been done on dimension polynomials in the ordinary case for other difference constructions, such as for prime difference-differential ideals \cite{levin2021bivariate}, and for ordinary $\s$\=/algebraic groups in \cite{wibmer2022finiteness}. 
The existence of the dimension polynomial (\Cref{the: dim poly intro}) and the invariance of its degree and leading coefficient proven in this paper is analogous to classical results \cite[Sections 3 and 4]{levin2008difference},  regarding dimension polynomials for filtrations and gradings of difference modules, and for difference field extensions. In fact, \Cref{the: dim poly intro} is a generalisation of the result for filtrations of filtered modules, as one can consider a difference module to be an additive $\s$\=/algebraic group (\Cref{the: relating modules and groups}).  An application of the invariants introduced in \Cref{def: diff type diff dim etc} is that they yield a necessary criterion for isomorphism, that is, two isomorphic difference algebraic groups cannot have differing invariants. The invariants will be useful in furthering the theory of partial difference algebraic groups, for example, analogous invariants have been used to find Jordan\=/H\"older decompositions for both ordinary difference algebraic groups \cite{wibmer2021almost}, and partial differential algebraic groups \cite{cassidy2011ajordan}.

We now outline of the contents of the paper.
\Cref{sec: prelims} gives a brief overview of the necessary background for both difference algebra and algebraic groups. In \Cref{sec: diff varieties}, we introduce partial difference algebraic geometry and difference varieties, before restricting in \Cref{sec: diff alg goups} to studying difference algebraic groups. In particular we see their correspondence to Hopf structures and generalise the definition of Zariski closures from \cite{wibmer2022finiteness} to the partial case. Next, in \Cref{sec: gen dif alg grp def} we define generalised difference algebraic groups, a more general notion of the Zariski closures, which will allow us to use induction throughout our proofs. In \Cref{sec: kernels}, we study the kernels of certain projection maps of generalised difference algebraic groups with $n$ endomorphisms, and we see that these form generalised difference algebraic groups with $n-1$ endomorphisms. This allows us in \Cref{sec: ideal gen prop} to prove a strong ideal generation property on the defining ideals for a generalised difference algebraic group, of which \Cref{the: ideal intro} is a corollary. Another consequence of \Cref{the: ideal intro} is the existence of quotients of difference algebraic groups, which is proven in \Cref{sec: quotients}. In \Cref{sec: dim polys}, we prove the existence of a dimension polynomial for generalised difference algebraic groups, \Cref{the: dim poly intro}. Finally in \Cref{sec: relating polys to levin} we compare our dimension polynomials with classical results from \cite{levin2008difference} and see that they coincide in certain cases.

We will assume that all rings are unital and commutative, and take the convention that $\mathbb{N}=\{0,1,2,\dots\}$.


\section{Preliminaries}\label{sec: prelims}

\subsection{Difference Algebra}

Here we introduce the basic concepts of partial difference algebra, as defined in the standard reference \cite[Section 2]{levin2008difference}, however here we will not require the assumption that the basic endomorphisms are injective.

A \textbf{difference ring} ($\s$\=/ring) is a ring $R$ together with a finite set $\s=\basicset$ of pairwise commuting endomorphisms of $R$. We call $\s$ the \textbf{basic set} of $R$, and its elements the \textbf{basic endomorphisms}, or \textbf{shifts} of $R$. If the basic endomorphisms are automorphisms of $R$, we call $R$ an \textbf{inversive} $\s$\=/ring. A \textbf{difference field} (\s\=/field) is a difference ring where the underlying ring is a field. 

For any concepts in difference algebra where $\s=\basicset$, if $n=1$, this is called the \textbf{ordinary} case, and if $n>1$, this is the \textbf{partial} case. In the ordinary case, we will write $\s\colon  R\rightarrow R$ for the basic endomorphism of a $\s$\=/ring $R$.

\begin{example}\label[example]{ex: difference field}
   Consider the field $R=\mathbb{C}(x_1,\dots,x_n)$ of rational functions in $n$ variables over $\mathbb{C}$. Then $R$ is a $\s$\=/field with basic set $\{\s_1,\dots,\s_n\}$, where for each $1\leq j\leq n$, $\s_j\colon R\rightarrow R$ is the ring endomorphism \begin{align*}
            \s_j(f(x_1,\dots,x_n))=f(x_1,\dots,x_{j-1},x_j+1,x_{j+1},\dots,x_n).
        \end{align*}
      In fact, we can generalise this example in the following way. Given any $n$ endomorphisms $\phi_j\colon \mathbb{C}(x)\rightarrow \mathbb{C}(x)$ for $1\leq j\leq n$ (that is, each $(\mathbb{C}(x),\phi_j)$ is an ordinary difference field), one can consider the partial $\s$\=/field $\mathbb{C}(x_1,\dots,x_n)$, where $\s=\basicset$ and for each $1\leq j\leq n$, $\s_j\colon \mathbb{C}(x_1,\dots,x_n)\rightarrow \mathbb{C}(x_1,\dots,x_n)$ is the endomorphism $\s_j(f(x_1,\dots,x_n))\mapsto f(x_1,\dots,x_{j-1},\phi_j(x_j),x_{j+1},\dots,x_n)$. 
\end{example}

\begin{example}\label[example]{ex: hypergeometric functions}
    A classical example of where shifts of the kind in (\Cref{ex: difference field}) can be found is in the study of hypergeometric functions \cite[Chapter 8]{beals2010special}. These are functions $F={_2}F_1(a,b,c;x)$ with parameters $a,b,c$ and variable $z$, which arise naturally as the solutions to the hypergeometric equation. Two functions are called contiguous if two of their corresponding parameters are equal, and the third differs by $\pm 1$. For example, ${_2}F_1(a,b,c;x)$ is contiguous to ${_2}F_1(a+1,b,c;x)$. This property can be written clearly in the language of difference algebra. Considering $\mathbb{C}(a,b,c,x)$ as a partial $\s$\=/field with $\s_1(f(a,b,c;x))=f(a+1,b,c;x)$, $\s_2(f(a,b,c;x))=f(a,b+1,c;x)$ and $\s_3(f(a,b,c;x))=f(a,b,c+1;x)$, two hypergeometric functions are said to be contiguous if one is a shift of the other.
\end{example}

We will slightly abuse notation by saying that two or more difference rings have the same basic set $\s$, when we really mean that there is some mapping from $\s$ to the basic set of each difference ring. If $R$ and $S$ are both $\s$\=/rings, a morphism of rings $\varphi\colon  R\rightarrow S$ is called a \textbf{morphism of difference rings} (morphism of $\s$\=/rings) if $\varphi$ commutes with $\s_j$ for each $1\leq j\leq n$.
For a $\s$\=/ring $R$, a subring $S$ of the underlying ring $R$ is a \textbf{difference subring} ($\s$\=/subring) of $R$ if $\s_j(S)\subseteq S$ for each $1\leq j\leq n$, and an ideal $I$ of the underlying ring $R$ is called a \textbf{difference ideal} ($\s$\=/ideal) of $R$ if $\s_j(I)\subseteq I$ for each $1\leq j\leq n$. 

Notice that a $\s$\=/subring or $\s$\=/ideal of $R$ is not just closed under each basic endomorphism, but also any compositions of them. Since the basic endomorphisms are pairwise commuting, any composition of them can be written $\generalelemt$ for $k_1,\dots,k_n\in \mathbb{N}$. We call these endomorphisms \textbf{transforms} of $R$, and denote the set of them by the free commutative monoid 
            $T_\s=\{\generalelemt\ | \ k_1,\dots,k_n\in \mathbb{N}\}$.  We assign a natural order to the transforms of a $\s$\=/ring, that is, given $\tau\in T_\s$, so $\tauequals$ for some $k_1,\dots,k_n\in\mathbb{N}$, the \textbf{order} of $\tau$ is 
        $\ord(\tau)=k_1+\cdots+k_n$. The subsets \begin{align}\label{eq: Ts[i] and Ts(i)}
    T_\s[i]=\{\tau\in T_\s\ | \ \ord(\tau)\leq i\}\subseteq T_\s\quad \text{and}\quad T_\s(i)=\{\tau\in T_\s\ | \ \ord(\tau)= i\}\subseteq T_\s[i]
\end{align} for each $i\in\mathbb{N}$ will be crucial for proving finiteness properties of partial difference algebraic groups.

    Let $R$ be a $\s$\=/ring, and let $F\subseteq R$. While $(F)$ is the ideal of $R$ generated by $F$, we write $[F]$ for the smallest $\s$\=/ideal of $R$ containing $F$, and call it \textbf{$\s$\=/ideal of $R$ $\s$\=/generated by $F$}. We can see that $[F]=(\{\tau(f)\ |\ \tau\in T_\s,\ f\in F\})$, that is, $[F]$ is generated as an ideal by the elements of $F$, and every transform of every element of $F$. If a $\s$\=/ideal $I\subseteq R$ is such that $I=[F]$ for some finite  $F\subseteq R$, then we say $I$ is \textbf{finitely $\s$\=/generated} as a $\s$\=/ideal of $R$.

Let $R$ and $k$ be $\s$\=/rings, and suppose that $R$ is a $k$-algebra via the morphism $\psi_R\colon k\rightarrow R$. If $\psi_R$ is a morphism of $\s$\=/rings,
then $R$ is called a \textbf{difference $k$\=/algebra} ($k$\=/$\s$\=/algebra). A \textbf{morphism of $k$\=/$\s$\=/algebras} is a morphism of $k$\=/algebras that is also a morphism of $\s$\=/rings. An \textbf{isomorphism of $k$\=/$\s$\=/algebras} is an isomorphism of $k$\=/algebras that is also a morphism of $k$\=/$\s$\=/algebras. We write $R\cong S$ to denote that $k$\=/$\s$\=/algebras $R$ and $S$ are isomorphic.
A \textbf{$k$\=/$\s$\=/subalgebra} of a $k$\=/$\s$\=/algebra $R$ is a $k$\=/subalgebra of the underlying $k$\=/algebra which is also a $\s$\=/subring of $R$.

Let $k$ be a $\s$\=/ring and let $R$ be a $k$\=/$\s$\=/algebra. For a subset $F\subseteq R$, $k[F]$ is the $k$\=/subalgebra of $R$ generated by $F$. We denote by $k\{F\}$ the \textbf{$k$\=/$\s$\=/subalgebra of $R$ $\s$\=/generated by $F$}, that is, $k\{F\}$ is the smallest $k$\=/$\s$\=/subalgebra of $R$ containing $F$.
Similarly to the case for $\s$\=/ideals, we in fact see that $k\{F\}=k[\{\tau(f)\ |\ \tau\in T_\s,\ f\in F\}]$. 
If a $k$\=/$\s$\=/algebra $R$ is such that $R=k\{F\}$ for some finite  $F\subseteq R$, we say $R$ is a \textbf{finitely $\s$\=/generated} $k$\=/$\s$\=/algebra.

Where algebraic geometry is the study of the solutions of polynomials, difference algebraic geometry is the study of solutions of $\s$\=/polynomials.
    The \textbf{ring of $\s$\=/polynomials} over a $\s$\=/ring $k$ in the $\s$\=/indeterminates $y_1,\dots,y_s$ is the polynomial ring \begin{align*}
    k\{y_1,\dots,y_s\}=k\bigl[\{\tau (y_i) \ |\  \tau\in T_\s, 1\leq i\leq s\}\bigr]
\end{align*}
whose indeterminates are not just $y_1,\dots,y_s$, but also every transform of $y_1,\dots,y_s$.

\begin{example}\label[example]{ex: poly ring k sigma alg}
    The ring of $\s$\=/polynomials is a natural example of a $k$\=/$\s$\=/algebra, where for each $1\leq j\leq n$, $\s_j\colon  k\{y_1,\dots,y_s\}\rightarrow k\{y_1,\dots,y_s\}$ is extended from $\s_j\colon k\rightarrow k$ such that $\s_j(\tau (y_i))=(\s_j\tau) (y_i)$ for any $1\leq i\leq s$, and any $\tau\in T_\s$. 
\end{example}

If $R$ is a $k$\=/$\s$\=/algebra, we evaluate a $\s$\=/polynomial $f\in k\{y_1,\dots,y_s\}$ at a point $x=(x_1,\dots,x_s)\in R^s$ by substituting $\tau(x_i)$ for $\tau (y_i)$ for each $\tau\in T_\s$, $1\leq i\leq s$. A point $x$ in $R^s$ is called a \textbf{solution} to a $\s$\=/polynomial $f\in k\{y_1,\dots,y_s\}$ if $f(x)=0$.

\begin{example}
    The hypergeometric functions (see \Cref{ex: hypergeometric functions}) are naturally solutions to some $\s$\=/polynomials. It is a well-known result of Gauss that one can find an equation (linear over $\mathbb{C}(a,b,c,x)$), relating any hypergeometric function with any two of its contiguous functions. One can express any of these linear relations as a difference equation, for example, (8.5.4 in \cite{beals2010special}), \begin{align*}
    (a+(b-c)x)F=a(1-x)\s_1(F)-\frac{(c-a)(c-b)x}{c}\s_3(F).
\end{align*}
That is, contiguous hypergeometric functions are solutions to some linear $\s$\=/polynomials over $\mathbb{C}(a,b,c,x)$.
\end{example}

Finally, for a $\s$\=/ring $k$ and $k$\=/$\s$\=/algebras $R$ and $S$, the tensor product $R\otimes_k S$ of the underlying $k$\=/algebras is in fact a $k$\=/$\s$\=/algebra, where for each $1\leq j\leq n$, $\s_j\colon R\otimes_k S\rightarrow R\otimes_k S$ is the unique ring endomorphism such that $\s_j(r\otimes s)=\s_j(r)\otimes \s_j(s)$ for each $r\in R$, $s\in S$.
If $R$ and $S$ are both finitely $\s$\=/generated $k$\=/$\s$\=/algebras, then $R\otimes_k S$ is finitely $\s$\=/generated as a $k$\=/$\s$\=/algebra.

\subsection{Algebraic Groups}

We now review the required knowledge of affine schemes of finite type over a field and (affine) algebraic groups over a field. The facts in this section can either be found or derived from \cite{waterhouse2012introduction} and \cite{milne2017algebraic}, and we follow this functorial approach to studying affine schemes. 

Given a field $k$, an \textbf{affine scheme of finite type} over $k$ is a functor from the category of $k$\=/algebras to the category of sets which is isomorphic to the functor $\Hom(A,-)$ for some finitely generated $k$\=/algebra $A$. We write $k[\X]=A$, and call $A$ the coordinate ring or the representing algebra for $\X$.

An \textbf{algebraic group} $\G$ over $k$ is a functor from the category of $k$\=/algebras to the category of groups which is an affine scheme of finite type over $k$ when considered as a functor into the category of sets. Equivalently, an algebraic group $\G$ over $k$ is an affine scheme of finite type over $k$ such that $k[\G]$ has the structure of a $k$\=/Hopf algebra.

\begin{example}\label[example]{example: algebraic groups}
Let $k$ be a field and let $s\geq 1$ be some natural number. Throughout this paper we will consider the following classical examples of algebraic groups.
    \begin{itemize}
            \item The additive group $\mathbb{G}_a^s$, which takes a $k$\=/algebra $R$ and returns $\mathbb{G}_a^s(R)=R^s$, is an algebraic group, represented by $k[\mathbb{G}_a^s]=k[y_1,\dots,y_s]$.
        \item The multiplicative group $\mathbb{G}_m^s$, which takes a $k$\=/algebra $R$ and returns $\mathbb{G}_m^s(R)=(R^\times)^s$, is an algebraic group, represented by $k[\mathbb{G}_m^s]=k[y_1,\dots,y_s,y_1^{-1},\dots,y_s^{-1}]$.
        \item The general linear group $\operatorname{GL}_s$, which takes a $k$\=/algebra $R$ and returns $\operatorname{GL}_s(R)$, the set of $s\times s$ invertible matrices with entries in $R$, is an algebraic group. Putting $X=(x_{i,j})_{1\leq i,j\leq s}$, $k[\operatorname{GL}_s]=k[X,1/\det(X)]$ is the representing algebra for $\operatorname{GL}_s$.
    \end{itemize}
\end{example}

Given an algebraic group $\G$ over $k$, a \textbf{closed subgroup} $\H$ of $\G$ (denoted $\H\leq\G$) is a closed subscheme of $\G$ such that $\H(R)$ is a subgroup of $\G(R)$ for each $k$\=/algebra $R$. 
If $\H$ is a closed subgroup of $\G$, there exists some Hopf ideal $I$ of $k[\G]$ such that $k[\H]=k[\G]/I$, and conversely, for each Hopf ideal $I$ of $k[\G]$, there is a closed subgroup $\H$ of $\G$ such that $k[\H]=k[\G]/I$. We write $\I(\H)$ for $I$ and call $\I(\H)$ the \textbf{defining ideal} of $\H$ in $\G$. 
For closed subgroups $\H_1$ and $\H_2$ of $\G$, $\H_1\leq \H_2$ if and only if $\I(\H_2)\subseteq \I(\H_1)$. Intuitively, this means that $\H_1$ is defined by more polynomial equations than $\H_2$.

Morphisms of algebraic groups and their properties will be an important tool for the proofs of some of our main results. A morphism $\phi\colon\G\rightarrow \H$ of algebraic groups over $k$ is a natural transformation of functors such that $\phi_R\colon\G(R)\rightarrow \H(R)$ is a group homomorphism for every $k$\=/algebra $R$. Then $\phi\colon\G\rightarrow \H$ corresponds to a dual morphism $\phi^*\colon k[\H]\rightarrow k[\G]$ of the representing $k$\=/Hopf algebras.

We can define the image and kernel of a morphism $\phi\colon\G\rightarrow \H$ of algebraic groups. The \textbf{image} $\phi(\G)$ is the closed subgroup of $\H$ defined by the Hopf ideal $\ker(\phi^*)$ of $k[\H]$. The \textbf{kernel} $\ker(\phi)$ is the functor from the category of $k$\=/algebras to the category of groups which takes a $k$\=/algebra $R$ to $\ker(\phi_R)$, and is a closed subgroup of $\G$.

A morphism $\phi\colon\G\rightarrow \H$ of algebraic groups over $k$ is called a \textbf{closed embedding} if it satisfies any one of the  following equivalent conditions: 
    \begin{itemize}
    \item the induced morphism $\phi\colon \G\rightarrow \phi(\G)$ is an isomorphism,
    \item for every $k$\=/algebra $R$, $\phi_R\colon\G(R)\rightarrow \H(R)$ is injective,
    \item the dual morphism $\phi^*\colon k[\H]\rightarrow k[\G]$ of $k$\=/algebras is surjective.
\end{itemize}

A morphism $\phi\colon\G\rightarrow \H$ of algebraic groups over $k$ is called a \textbf{quotient map} if $\phi(\G)=\H$, or equivalently, if the dual morphism $\phi^*\colon k[\H]\rightarrow k[\G]$ of $k$\=/algebras is injective.

Finally, we introduce how we can consider an algebraic group $\G$ over a field $k$ to be an algebraic group over some other field $k'$. 
    Suppose that we have fields $k$, $k'$, and a morphism $k\rightarrow k'$ of rings. Any $k'$\=/algebra $R$ is then also a $k$\=/algebra via $k\rightarrow k'\rightarrow R$, and morphisms of $k'$\=/algebras can be considered to be morphisms of $k$\=/algebras. Therefore, a functor $\X$ from the category of $k$\=/algebras to the category of sets can be considered as a functor $\X'$ from the category of $k'$\=/algebras to the category of sets. 
    We call $\X'$ the functor obtained from $\X$ by \textbf{base change} via $k\rightarrow k'$. 
A morphism $\phi\colon \X\rightarrow \Y$ of functors from $k$\=/algebras to sets can be considered to be a morphism $\phi'\colon \X'\rightarrow \Y'$ of functors from $k'$\=/algebras to sets by base change. If $\phi$ is an isomorphism of functors, then $\phi'$ is an isomorphism of functors. 

In this paper, we will be concerned with base change of an algebraic group of a certain kind. We will describe this in more detail here.

\begin{definition}\label[definition]{def: group base change}
    Let $\G$ be an algebraic group over a field $k$ and let $\tau\colon  k\rightarrow k$ be an endomorphism. Denote by ${^\tau}\G$ the functor over $k$ obtained from $\G$ by base change via $\tau\colon  k\rightarrow k$, and notice that ${^\tau}\G$ is also an algebraic group over $k$. For a morphism $\phi\colon \G\rightarrow \H$ of algebraic groups, denote by ${^\tau}\phi\colon {^\tau}\G\rightarrow{^\tau\H}$ the morphism of functors obtained from $\phi$ by base change via $\tau\colon k\rightarrow k$, and notice that ${^\tau}\phi$ is also a morphism of algebraic groups. 
\end{definition}

Suppose that we have a $k$\=/algebra $R$. Denote by $R'$, $R$ considered as a $k$\=/algebra via some endomorphism $\tau\colon k\rightarrow k$. 
Suppose that $\G$ is an algebraic group defined by some polynomials $f_1,\dots,f_n\in k[y_1,\dots,y_s]$, that is, $\G$ is represented by $k[y_1,\dots,y_s]/(f_1,\dots,f_n)$. The elements of $\G(R)$ are the solutions of the polynomials $f_1,\dots,f_n$ in $R^s$. Then the elements of ${^\tau}\G(R)=\G(R')$ are the solutions to $f_1,\dots,f_n$ in $(R')^s$. These are the solutions in $R^s$ to the polynomials $g_1,\dots,g_n$, where $g_i$ is obtained from $f_i$ by applying $\tau$ to all of the coefficients. Equivalently, these are the solutions in $R^s$ to the polynomials $\tau(f_1),\dots,\tau(f_n)\in k\bigl[\tau (y_1),\dots,\tau (y_s)\bigr]$, where we have relabelled our indeterminates. We use this relabelling notation for clarity when we simultaneously consider algebraic groups obtained by different base changes.

\begin{remark}\label[remark]{rem: representing alg for base change}
    Let $\G$ be an algebraic group over a field $k$ and let $\tau\colon  k\rightarrow k$ be an endomorphism. If $\G$ is represented by $k[y_1,\dots,y_s]/(f_1,\dots,f_n)$, then ${^\tau}\G$ is represented by $k\bigl[\tau (y_1),\dots,\tau (y_s)\bigr]\big/\bigl(\tau(f_1),\dots,\tau(f_n)\bigr)$.
\end{remark}

\begin{remark}\label[remark]{rem: base change comp}
    Clearly from \Cref{rem: representing alg for base change} and the discussion above, if $\tau_1, \tau_2\colon k\rightarrow k$ are commuting endomorphisms of a field $k$ and $\G$ is an algebraic group over $k$, then ${^{\tau_1}}({^{\tau_2}}\G)={^{\tau_1\tau_2}}\G={^{\tau_2}}({^{\tau_1}}\G)$. 
\end{remark}

\begin{example}\label[example]{ex: tau GLs}
    For $s\geq 1$, $\operatorname{GL}_s$ is an algebraic group over $k$ where for each $k$\=/algebra $R$, if $X=(x_{i,j})_{1\leq i,j\leq s}$, \begin{align*}
        \operatorname{GL}_s(R)=\{(X,y)\in R^{s\times s}\times R\ |\ \det(X)y=1\}.
    \end{align*}
    If $\tau\colon k\rightarrow k$ is an endomorphism, $\tau(\det(X))=\det(\tau(X))$, where $\tau(X)=(\tau (x_{i,j}))_{1\leq i,j\leq s}$. Therefore, \begin{align*}
        {^\tau}\operatorname{GL}_s(R)=\bigl\{(\tau(X),\tau (y))\in R^{s\times s}\times R\ |\ \det(\tau(X))\tau (y)=1\bigr\}
    \end{align*}
    for each $k$\=/algebra $R$.
    So the points in ${^\tau}\operatorname{GL}_s(R)$ are the same as the points in $\operatorname{GL}_s(R)$, but the way that we are considering $R$ as a $k$\=/algebra is different.
\end{example}


\section{Difference Varieties}\label{sec: diff varieties}

We will now introduce some concepts in difference algebraic geometry, in particular, we introduce a partial difference variety over a difference field, and see how given an affine scheme of finite type over a difference field, we can assign difference structure to find a difference variety.
The work in this section is adapted from the ordinary case in \cite[Section 1.2]{wibmer2022finiteness}.
\textbf{From now on we will assume that $k$ is a $\s$\=/field}. Let $k$\=/$\operatorname{Alg}$ denote the category of $k$\=/algebras, $k$\=/$\s$\=/$\operatorname{Alg}$ denote the category of $k$\=/$\s$\=/algebras, $\operatorname{Sets}$ denote the category of sets, and let Groups denote the category of groups.

While an affine scheme of finite type is defined by some polynomials, a $\s$\=/variety is defined by some $\s$\=/polynomials. Let us formalise what this means. 
For some $s\geq 1$ and some set $F\subseteq k\{y_1,\dots,y_s\}$, define a functor $\V_F$ from $k$\=/$\s$\=/$\operatorname{Alg}$ to $\operatorname{Sets}$ such that \begin{align*}
    \mathbb{V}_F(R)=\{x\in R^s\ |\  f(x)=0 \ \text{for all}\ f\in F\}
\end{align*}
for any $k$-$\s$-algebra $R$, and 
    $\V_F(\varphi)\colon \V_F(R)\rightarrow \V_F(S),\ (x_1,\dots,x_s)\mapsto (\varphi(x_1),\dots,\varphi(x_s))$
for any morphism $\varphi\colon R\rightarrow S$ of $k$\=/$\s$\=/algebras. 
It is straightforward to check that $\V_F$ is in fact a well-defined functor.

\begin{definition}
    A \textbf{difference variety} ($\s$\=/variety) over $k$ is a functor $X$ from $k$\=/$\s$\=/$\operatorname{Alg}$ to $\operatorname{Sets}$ such that $X=\V_F$  for some $s\geq 1$ and some $F\subseteq k\{y_1,\dots,y_s\}$. A \textbf{morphism of $\s$\=/varieties} over $k$ is a natural transformation of functors, and an isomorphism of $\s$\=/varieties is a natural isomorphism of functors. 
\end{definition}

Given a $\s$\=/variety $Y=\mathbb{V}_F$ over $k$ for some $s\geq 1$ and $F\subseteq k\{y_1,\dots,y_s\}$, a subfunctor $X\subseteq Y$ is called a \textbf{$\s$\=/closed $\s$\=/subvariety} of $Y$ if $X=\mathbb{V}_G$ for some superset $F\subseteq G\subseteq k\{y_1,\dots,y_s\}$. 

It will be useful to study all of the $\s$\=/polynomials which vanish at every point in a $\s$\=/variety. 
\begin{lemma}\label[lemma]{lem: I(X) ideal eq to [F]}
    Let $X=\mathbb{V}_F$ be a $\s$\=/variety over $k$, where $F\subseteq k\{y_1,\dots,y_s\}$ for some $s\geq 1$. Then \begin{align*}
        \mathbb{I}(X)=\{f\in k\{y_1,\dots,y_s\}\ |\  f(x)=0 \ \text{for all $k$\=/$\s$\=/algebras $R$ and all $x\in X(R)$}\}
    \end{align*}
    is a $\s$\=/ideal of $k\{y_1,\dots,y_s\}$. In fact, $\I(X)=[F]$.
    \end{lemma} 
    \begin{proof}
    It is straightforward to check that $\I(X)$ is a $\s$\=/ideal of $k\{y_1,\dots,y_s\}$. 
        We know that $F\subseteq \I(X)$ by definition, and since $\I(X)$ is a $\s$\=/ideal, $[F]\subseteq\I(X)$. 
    Conversely, let $g\in \I(X)$. Consider the $k$\=/$\s$\=/algebra
        $R=k\{y_1,\dots,y_s\}/[F]$ and notice that any $f\in F$ vanishes at $(y_1,\dots,y_s)+[F]$, so $(y_1,\dots,y_s)+[F]\in X(R)$. Since $g\in\I(X)$, $g$ vanishes at any point in $X(R)$, in particular, $g(y_1,\dots,y_s)+[F]=0+[F]$, and
        hence $g\in [F]$. Therefore $\mathbb{I}(X)\subseteq [F]$, and $\I(X)=[F]$ as required.
    \end{proof}

    \begin{definition}
    Given a $\s$\=/variety $X$ over $k$, the $k$\=/$\s$\=/algebra
        $k\{X\}=k\{y_1,\dots,y_s\}/\mathbb{I}(X)$
    is called the \textbf{coordinate ring} of $X$. 
\end{definition}

Notice that for any $\s$\=/variety $X$ over $k$, its coordinate ring is a finitely $\s$\=/generated $k$\=/$\s$\=/algebra. We will see that this induces a dual correspondence between $\s$\=/varieties over $k$ and finitely $\s$\=/generated $k$\=/$\s$\=/algebras.

\subsection{Difference Varieties and Representing Difference Algebras}

We will now adapt the ideas of representing algebras in \cite[Section 1.2, page 4]{waterhouse2012introduction} to the partial difference case. From now on, we \textbf{call any functor isomorphic to a $\s$\=/variety, a $\s$\=/variety}. 

Recall that for a $k$\=/$\s$\=/algebra $A$, $\operatorname{Hom}_{k\text{\=/}\s\text{\=/}\operatorname{alg}}(A,-)$ is the functor from $k$\=/$\s$\=/$\operatorname{Alg}$ to $\operatorname{Sets}$ mapping a $k$\=/$\s$\=/algebra $R$ to the set of $k$\=/$\s$\=/algebra morphisms $A\rightarrow R$. We will write $\operatorname{Hom}(A,-)$ for $\operatorname{Hom}_{k\text{\=/}\s\text{\=/}\operatorname{alg}}(A,-)$, when it is clear we are considering morphisms of $k$\=/$\s$\=/algebras.

\begin{definition}
    A functor $X$ from $k$\=/$\s$\=/$\operatorname{Alg}$ to $\operatorname{Sets}$ is \textbf{representable} if there exists some $k$\=/$\s$\=/algebra $A$ such that the functors $X$ and $\operatorname{Hom}_{k\text{\=/}\s\text{\=/}\operatorname{alg}}(A,-)$ are isomorphic. In this case, we say that $X$ is represented by the $k$\=/$\s$\=/algebra $A$.
\end{definition}

Notice that if a functor $X$ from $k$\=/$\s$\=/$\operatorname{Alg}$ to $\operatorname{Sets}$ is represented by $k$\=/$\s$\=/algebra $A$, and $A$ is isomorphic to some  $k$\=/$\s$\=/algebra $B$, then $X$ is also represented by $B$.  Two representable functors are isomorphic if and only if their representing algebras are isomorphic.

\begin{proposition}\label[proposition]{prop: prelims varieties and algebras corresp}
    A functor from $k$\=/$\s$\=/$\operatorname{Alg}$ to $\operatorname{Sets}$ is a $\s$\=/variety over $k$ if and only if it is represented by a finitely $\s$\=/generated $k$\=/$\s$\=/algebra.
\end{proposition}
\begin{proof}
    Firstly, suppose that $X=\V_F$ for some $s\geq 1$ and some set $F\subseteq k\{y_1,\dots,y_s\}$, and the coordinate ring $k\{X\}=k\{y_1,\dots,y_s\}/\I(X)$ of $X$. For a $k$\=/$\s$\=/algebra $R$, there is a one-to-one correspondence between elements $(r_1,\dots,r_s)\in X(R)$ and $k$\=/$\s$\=/algebra morphisms $\varphi\colon k\{X\}\rightarrow R$ such that $\varphi(y_i+\I(X))=r_i$ for each $1\leq i\leq s$. This correspondence in fact defines a natural isomorphism between functors $X$ and $\Hom(k\{X\},-)$, that is, $X$ is represented by a finitely $\s$\=/generated $k$\=/$\s$\=/algebra.

Now, say that $X$ is a $\s$\=/variety, that is, $X$ is isomorphic to $\V_F$ for some $s\geq 1$ and some $F\subseteq k\{y_1,\dots,y_s\}$. We have seen that $\V_F$ is represented by the finitely $\s$\=/generated $k$\=/$\s$\=/algebra $k\{\V_F\}=k\{y_1,\dots,y_s\}/[F]$, and as $X$ and $\V_F$ are isomorphic as functors, the same is true for $X$ as required. 

Conversely, suppose that $X$ is a functor from $k$\=/$\s$\=/$\operatorname{Alg}$ to $\operatorname{Sets}$  that is represented by some finitely $\s$\=/generated $k$\=/$\s$\=/algebra $A$. Then $A\cong k\{y_1,\dots,y_s\}/I$ for some $s\geq 1$ and some $\s$\=/ideal $I\subseteq k\{y_1,\dots,y_s\}$. The $k$\=/$\s$\=/algebra $A$ is then a representing algebra for the $\s$\=/variety $\V_I$, so $X$ and $\mathbb{V}_I$ are represented by the same $k$\=/$\s$\=/algebra, hence are isomorphic. That is, $X$ is a $\s$\=/variety as required.   
\end{proof}

From now on, for a $\s$\=/variety $X$ over $k$, write $k\{X\}$ for a $k$\=/$\s$\=/algebra representing $X$.

\begin{proposition}\label[proposition]{prop: morphisms corresp dual}
    Let $X$ and $Y$ be $\s$\=/varieties over $k$. There is a one-to-one correspondence between morphisms $\phi\colon X\rightarrow Y$ of $\s$\=/varieties over $k$ and morphisms $\phi^*\colon k\{Y\}\rightarrow k\{X\}$ of $k$\=/$\s$\=/algebras. We say that the corresponding morphisms $\phi$ and $\phi^*$ are \textbf{dual} to one another.
\end{proposition}
\begin{proof}
    This is the difference analogue of the Yoneda Lemma \cite[Section 1.3, page 6]{waterhouse2012introduction}. Let $A=k\{X\}$ and $B=k\{Y\}$. A morphism $\phi\colon X\rightarrow Y$ of $\s$\=/varieties over $k$ can be considered to be a morphism $\phi\colon\Hom(A,-)\rightarrow\Hom(B,-)$. Then $\phi^*=\phi_A(\id_A)\colon B\rightarrow A$ is the dual morphism of $k$\=/$\s$\=/algebras.
    Conversely, for a morphism $\phi^*\colon B\rightarrow A$ of $k$\=/$\s$\=/algebras, the dual morphism $\phi\colon\Hom(A,-)\rightarrow \Hom(B,-)$ of $\s$\=/varieties over $k$ is such that for a $k$\=/$\s$\=/algebra $R$, $\phi_R(\psi)=\psi\circ\phi^*$.
    Notice that these constructions are inverses to each other, so we have a one-to-one correspondence as required.
\end{proof}

\begin{theorem}\label[theorem]{the: equiv of cate}
    The category of $\s$\=/varieties over $k$ is dually equivalent to the category of finitely $\s$\=/generated $k$\=/$\s$\=/algebras.
\end{theorem}
\begin{proof}
One can check that that given morphisms $\phi\colon  X\rightarrow Y$ and $\psi\colon  Y\rightarrow Z$ of $\s$\=/varieties corresponding to morphisms $\phi^*\colon k\{Y\}\rightarrow k\{X\}$ and $\psi^*\colon  k\{Z\}\rightarrow k\{Y\}$ of $k$\=/$\s$\=/algebras respectively, then $(\psi\circ\phi)^*=\phi^*\circ \psi^*$.
    This theorem then follows from \Cref{prop: prelims varieties and algebras corresp,prop: morphisms corresp dual}.
\end{proof}

Another property of the category of $\s$\=/varieties is that it has products. Given $\s$\=/varieties $X$ and $Y$  over $k$, the product functor $X\times Y$ is also a $\s$\=/variety over $k$, represented by the $k$\=/$\s$\=/algebra $k\{X\}\otimes_k k\{Y\}$. This is the difference analogue of a fact stated in \cite[Section 1.4, page 7]{waterhouse2012introduction}.


\subsection{Difference Subvarieties and Difference Ideals}

Recall that a $\s$\=/closed $\s$\=/subvariety of $\V_F$ is of the form $\V_G$ for some superset $F\subseteq G\subseteq k\{y_1,\dots,y_s\}$. Now, if $Y$ is a general $\s$\=/variety over $k$, that is, if $Y$ is isomorphic to the functor $\V_F$ for some $F\subseteq k\{y_1,\dots,y_s\}$, we call a subfunctor $X\subseteq Y$ a $\s$\=/closed $\s$\=/subvariety of $Y$ if $X$ is isomorphic to $ \V_G$ for some $F\subseteq G\subseteq k\{y_1,\dots,y_s\}$. 
 
\begin{proposition}\label[proposition]{lem: subvarieties and ideals}
    Let $Y$ be a $\s$-variety over $k$. 
    Given a $\s$\=/ideal $I\subseteq k\{Y\}$, there exists a $\s$-closed $\s$-subvariety of $Y$ that is represented by $k\{Y\}/I$, and for any $\s$-closed $\s$-subvariety $X$ of $Y$, there exists some $\s$-ideal $I\subseteq k\{Y\}$ such that $X$ is represented by $k\{Y\}/I$.
\end{proposition}
\begin{proof}
As $Y$ is a $\s$\=/variety over $k$, there exists some $F\subseteq k\{y_1,\dots,y_s\}$ such that $Y$ and $\mathbb{V}_F$ are isomorphic as functors, and hence $k\{Y\}=k\{y_1,\dots,y_s\}/[F]$. Then a $\s$\=/ideal $I\subseteq k\{Y\}$ corresponds to a $\s$\=/ideal $J$ of $k\{y_1,\dots,y_s\}$ containing $[F]$. By definition, $\mathbb{V}_J$ is a $\s$\=/closed $\s$\=/subvariety of $\V_F$.
Given two $\s$\=/polynomials $f,g\in k\{y_1,\dots,y_s\}$ such that $f+[F]=g+[F]$, for any $k$\=/$\s$\=/algebra $R$ and any $x\in Y(R)$, $f(x)=g(x)$. Therefore, it is well-defined to evaluate elements of $k\{Y\}$ at a point $x\in Y(R)$ for any $k$\=/$\s$\=/algebra $R$, by choosing any representative of the coset and evaluating this $\s$\=/polynomial at $x$.
This means that we can consider the functor $X$ from  $k$\=/$\s$\=/$\operatorname{Alg}$ to $\operatorname{Sets}$ such that for any $k$\=/$\s$\=/algebra $R$, \begin{align*}
    X(R)=\{x\in Y(R)\ |\ f(x)=0 \ \text{for any}\ f\in I\}.
\end{align*}
Clearly $X$ is a subfunctor of $Y$ represented by the $k$\=/$\s$\=/algebra $k\{Y\}/I\cong k\{y_1,\dots,y_s\}/J$. That is, $X$ is isomorphic to $\V_J$ and hence $X$ is a $\s$\=/closed $\s$\=/subvariety of $Y$ represented by $k\{Y\}/I$ as required.

Conversely, let $X$ be a $\s$\=/closed $\s$\=/subvariety of $Y$. So $X$ and $Y$ are isomorphic to some $\V_G$ and $\V_F$ respectively, where $F\subseteq G\subseteq k\{y_1,\dots,y_s\}$. In particular, $[F]\subseteq [G]\subseteq k\{y_1,\dots,y_s\}$. We can then consider the image $I$ of $[G]$ in $k\{Y\}=k\{y_1,\dots,y_s\}/[F]$. We know that $I$ is a $\s$\=/ideal of $k\{Y\}$, and that $X$ is represented by $k\{X\}= k\{y_1,\dots,y_s\}/[G]\cong k\{Y\}/I$ as required.
\end{proof}

Given a $\s$\=/closed $\s$\=/subvariety $X$ of $Y$, from the previous proposition, there exists some ideal $I\subseteq k\{Y\}$ such that $X$ is represented by $k\{Y\}/I$.
We can consider $X$ to be defined as a $\s$\=/closed $\s$\=/subvariety of $Y$ by the $\s$\=/ideal $I\subseteq k\{Y\}$, and call the ideal $I$ of $k\{Y\}$ the \textbf{defining ideal of $X$ in $k\{Y\}$}. Where it is clear from the context, we will write $\I(X)$ to mean $I$, rather than the definition given in \Cref{lem: I(X) ideal eq to [F]}.

This dual construction allows us to show that a morphism of $\sigma$\=/varieties has a unique image, hence allowing us to define the difference analogue to an embedding.

\begin{lemma}\label[lemma]{lem: image of morphism}
    Let $\phi\colon  X\rightarrow Y$ be a morphism of $\sigma$\=/varieties over $k$. There exists a unique $\sigma$\=/closed $\sigma$\=/subvariety $\phi(X)$ of $Y$ with the following properties:\begin{itemize}
        \item[$(i)$] The morphism $\phi$ factors through $\phi(X)$.
        \item[$(ii)$] If $Z\subseteq Y$ is a $\sigma$\=/closed $\sigma$\=/variety such that $\phi$ factors through $Z$, then $\phi(X)\subseteq Z$.
    \end{itemize} 
    We call the $\sigma$\=/closed $\sigma$\=/subvariety $\phi(X)$ of $Y$ the \textbf{image} of $\phi$.
\end{lemma}
\begin{proof}
Let $\phi^*\colon k\{Y\}\rightarrow k\{X\}$ be the dual morphism to $\phi$, which exists by \Cref{prop: morphisms corresp dual}. Then $I=\ker(\phi^*)$ is a $\s$\=/ideal of $k\{Y\}$, and hence, by \Cref{lem: subvarieties and ideals}, defines a $\s$\=/closed $\s$\=/subvariety of $k\{Y\}$, which we denote by $\phi(X)$. We know that $I$ is the unique $\s$\=/ideal of $k\{Y\}$ such that $\phi^*$ factors through $k\{Y\}/I$, and if $J\subseteq k\{Y\}$ is a $\s$\=/ideal such that $\phi^*$ factors through $k\{Y\}/J$, then $J\subseteq I$. Therefore, $\phi(X)$ is the unique $\s$\=/closed $\s$\=/subvariety with the required dual properties.
\end{proof}

\begin{definition}
    A morphism $\phi\colon  X\rightarrow Y$ of $\sigma$\=/varieties over $k$ is called a \textbf{$\sigma$\=/closed embedding} if, given $\phi(X)$ as constructed in \Cref{lem: image of morphism}, $\phi$ induces an isomorphism $\phi\colon X\rightarrow\phi(X)$. 
\end{definition}

We will now prove an adaptation to a well-known result in algebraic geometry, characterising the $\s$\=/closed embeddings of $\s$\=/varieties.

\begin{proposition}\label[proposition]{lem: embedding iff surjective}
    Let $X$ and $Y$ be $\s$\=/varieties over $k$.
    A morphism $\phi\colon  X\rightarrow Y$ of $\sigma$\=/varieties over $k$ is a $\sigma$\=/closed embedding if and only if the dual morphism $\phi^*\colon  k\{Y\}\rightarrow k\{X\}$ of $k$\=/$\s$\=/algebras is surjective.
\end{proposition}
\begin{proof}
    Let $I=\ker(\phi^*)$, and recall that $\phi(X)$ is the $\s$\=/closed $\s$\=/subvariety of $Y$ represented by $k\{Y\}/I$ as described in \Cref{lem: subvarieties and ideals}. Notice that the dual map to $X\rightarrow \phi(X)$ is $k\{Y\}/I\rightarrow k\{X\}$.
    Since $k\{Y\}/I\rightarrow k\{X\}$ is an isomorphism if and only if $\phi^*\colon  k\{Y\}\rightarrow k\{X\}$ is surjective, $X\rightarrow \phi(X)$ is an isomorphism if and only if $\phi^*\colon  k\{Y\}\rightarrow k\{X\}$ is surjective as required.
\end{proof}

\subsection{Building Difference Structure}\label{sec: building difference structure}

This section adapts constructions from \cite[Section 1.3]{wibmer2022finiteness} to the partial case. We will show how we can build a difference structure on a $k$\=/algebra to get a $k$\=/$\s$\=/algebra, and see that this corresponds to building difference structure on an affine scheme to get a $\s$\=/variety.

For a $k$\=/$\s$\=/algebra $R$, let $R^\#$ denote the $k$\=/algebra obtained from $R$ by forgetting the ring endomorphisms $\s=\basicset$. Given an affine scheme of finite type $\X$ over $k$, define a functor $[\s]_k\mathcal{X}$ from $k$\=/$\s$\=/$\operatorname{Alg}$ to $\operatorname{Sets}$ such that for a $k$\=/$\s$\=/algebra $R$, $\sk \X(R)=\mathcal{X}(R^\#)$. 
We will show that $[\s]_k\mathcal{X}$ is a $\s$\=/variety over $k$ by constructing $k\{[\s]_k\mathcal{X}\}$ and showing that this is a finitely $\s$\=/generated $k$\=/$\s$\=/algebra. Then the fact that $[\s]_k\mathcal{X}$ is a $\s$\=/variety over $k$ will follow from \Cref{prop: prelims varieties and algebras corresp}.
In order to do this, we will consider how $k$\=/algebras can `become' $k$\=/$\s$\=/algebras, by constructing a left adjoint to the $(-)^\#$ functor.

\begin{definition}\label[definition]{def: tau A and A[i]}
    Let $k$ be a $\s$\=/field and let $A$ be a $k$\=/algebra. For every $\tau\in T_\s$, let $^\tau A=A\otimes_k k$, where the tensor product is formed using $\tau\colon  k\rightarrow k$ on the right hand side. For each $i\in\mathbb{N}$, let  \begin{align*}
    A[i]&=\bigotimes\limits_{\tau\in T_\s[i]} {^\tau A}
\end{align*}
where the tensor products are taken over $k$.
\end{definition}

We consider $^\tau A$ to be a $k$\=/algebra via the morphism of rings $k\rightarrow{^\tau}A$, $\lambda\mapsto 1_A\otimes \lambda$. We will see that $A[i]$ is a $k$\=/algebra, but firstly we introduce some notation. For any $i\in\mathbb{N}$, $1\leq j\leq n$ there are natural inclusions $A[i]\subseteq A[i+1]$ and ${^{\s_j}}A[i]\subseteq A[i+1]$, by filling in any `gaps' in the tensor product with $1_A\otimes 1$.
Further, for $\tau\in T_\s[i]$, there is a natural inclusion $\theta^{\tau}\colon{^\tau}A\rightarrow A[i]$ such that \begin{align*}
    \theta^\tau(r\otimes \lambda)=(1_A\otimes 1)\otimes\cdots\otimes(1_A\otimes 1)\otimes\underbrace{(r\otimes \lambda)}_{\text{in ${^\tau A}$ place}}\otimes (1_A\otimes 1)\otimes\cdots\otimes(1_A\otimes 1)\in \bigotimes_{\tau\in T_\s[i]}{^\tau}A= A[i]
\end{align*}
for any elementary element $r\otimes\lambda\in {^\tau A}$, extended additively for any element of ${^\tau}A$. The target space of $\theta^\tau$ could be any $A[i]$ where $\ord(\tau)\leq i$, by considering $\theta^\tau(r\otimes \lambda)\in A[\ord(\tau)]$ and then using the inclusion $A[\ord(\tau)]\subseteq A[i]$. 
For $i\in\mathbb{N}$, an elementary element of $A[i]$ is of the form  \begin{align*}
    \bigotimes_{\tau\in T_\s[i]}(r_\tau\otimes\lambda_\tau)
    =\prod_{\tau\in T_\s[i]}\theta^\tau(r_\tau\otimes\lambda_\tau)
\end{align*}
where $r_\tau\otimes\lambda_\tau\in {^\tau} A$ for each $\tau\in T_\s[i]$.
For any $\lambda \in k$, $\tau\in T_\s[i]$, \begin{align}\label{eq: A[i] k algebra equiv}
    (1_A\otimes \lambda)\otimes(1_A\otimes 1)\otimes\cdots\otimes(1_A\otimes 1)=(1_A\otimes 1)\otimes\cdots\otimes\underbrace{(1_A\otimes \lambda)}_{\text{in ${^\tau A}$ place}}\otimes\cdots\otimes(1_A\otimes 1)\in A[i],
\end{align}
and hence it is well-defined to consider $A[i]$ to be a $k$\=/algebra via $\lambda\mapsto \theta^\tau(1_A\otimes \lambda)$ for any $\tau\in T_\s[i]$. Notice that this makes the inclusion $\theta^\tau\colon {^\tau}A\rightarrow A[i]$ a morphism of $k$\=/algebras.

\begin{example}\label[example]{ex: tau A and A[i]}
    Consider $A=k[y_1,\dots,y_s]$, where $k$ is a $\s$\=/field. We will describe ${^\tau A}=k[y_1,\dots,y_s]\otimes_k k$ for some $\tau\in T_\s$. Recall that for $\lambda\in k$, $\lambda\otimes 1=1\otimes \tau(\lambda)\in {^\tau}A$, and ${^\tau A}$ is considered a $k$\=/algebra via $\lambda\mapsto 1\otimes\lambda$. Notice that when $\tau=\id$, we have $k$\=/algebra isomorphism 
        ${^{\id}}A\rightarrow k[y_1,\dots,y_s]$, where $f\otimes\lambda\mapsto \lambda f$ for elementary elements in ${^{\id}}A$. This same map would not be well defined on a general ${^{\tau}}A$, as for $\lambda \in k$, $\lambda\otimes 1=1\otimes \tau(\lambda)$ would not have a unique image. Instead, we relabel our polynomial ring $k\bigl[\tau (y_1),\dots,\tau (y_s)\bigr]$, which is a $k$\=/algebra in the usual sense for a polynomial ring, and notice that \begin{align*}
        {^\tau}A\rightarrow k\bigl[\tau (y_1),\dots,\tau (y_s)\bigr],\ f\otimes \lambda\mapsto \lambda\tau(f)\ \text{for elementary elements in ${^{\tau}}A$}
    \end{align*} is a well-defined $k$\=/algebra isomorphism when extended additively to all of ${^\tau A}$. 
    Therefore, for any $\tau\in T_\s$, we put ${^\tau A}= k\bigl[\tau (y_1),\dots,\tau (y_s)\bigr]$. Then for $i\in\mathbb{N}$, $A[i]=k\bigl[\{\tau (y_j)\ |\ \tau\in T_\s[i],1\leq j\leq s\}\bigr]$, and $A[i]$ is considered a $k$\=/algebra in the usual sense for a polynomial ring over $k$.
    
    More generally, if $F\subseteq k[y_1,\dots,y_s]$ and $A=k[y_1,\dots,y_s]/(F)$, for any $\tau\in T_\s$, $i\in\mathbb{N}$, \begin{align*}
        {^\tau}A= k\bigl[\tau (y_1),\dots,\tau (y_s)\bigr]/\bigl(\tau(F)\bigr)\ \text{and}\ A[i]=k\bigl[\{\tau (y_j)\ |\ \tau\in T_\s[i], 1\leq j\leq s\}\bigr]\big/\bigl(\{\tau(F)\ |\ \tau\in T_\s[i]\}\bigr), 
    \end{align*} where $\tau(F)=\{\tau(f)\ |\ f\in F\}$. 
\end{example}

\begin{remark}\label[remark]{rem: tau constructions agree}
Let $\G$ be an algebraic group over a $\s$\=/field $k$, and let $\tau\in T_\s$. Then there exists $s\geq 1$, and $f_1,\dots,f_m\in k[y_1,\dots,y_s]$ such that $\G$ is represented by $A=k[y_1,\dots,y_s]/(f_1,\dots,f_m)$. By \Cref{ex: tau A and A[i]}, ${^\tau}A=k\bigl[\tau (y_1),\dots,\tau (y_s)\bigr]\big/\bigl(\tau(f_1),\dots,\tau(f_m)\bigr)$, and by \Cref{rem: representing alg for base change}, ${^\tau}A$ represents ${^\tau}\G$ (\Cref{def: group base change}). This is a standard fact for base change of algebraic groups \cite[Section 1.6, page 11]{waterhouse2012introduction}.
\end{remark}

\begin{lemma}\label[lemma]{prop: [s]kA ks algebra}
    Given a $k$\=/algebra $A$, the union 
    $[\s]_kA=\bigcup_{i\in \mathbb{N}}A[i]$
is naturally a $k$\=/$\s$\=/algebra.
\end{lemma}
\begin{proof}
As it is the union of $k$\=/algebras, $\sk A$ is a $k$\=/algebra via 
$k\rightarrow \sk A$, $\lambda\mapsto 1_A\otimes \lambda\in A[0]\subseteq \sk A$ (equivalently, $\lambda\mapsto \theta^\tau(1_A\otimes \lambda)\in A[i]\subseteq \sk A$ for any $i\in\mathbb{N}$, $\tau\in T_\s[i]$, by (\ref{eq: A[i] k algebra equiv})). 
For each $1\leq j \leq n$, define $\s_j\colon\sk A\rightarrow \sk A$ as the unique ring morphism such that for each $i\in\mathbb{N}$ \begin{align}\label{eq: endomorphisms act on sk A}
    \s_j\left(\prod_{\tau\in T_\s[i]}\theta^\tau(r_\tau\otimes\lambda_\tau)\right)
    =\prod_{\tau\in T_\s[i]}\theta^{\s_j\tau}\bigl(r_\tau\otimes\s_j(\lambda_\tau)\bigr)
    \in {^{\s_j}}A[i]\subseteq A[i+1].
\end{align}
for an elementary element in $ A[i]$.
 So applying $\s_j$ both applies $\s_j$ to the field elements and shifts the order of the tensor products, filling in `gaps' with $1_A\otimes 1$.
The fact that the endomorphisms $\s_j\colon \sk A\rightarrow\sk A$ are pairwise commuting follows as each $\s_j\colon  k\rightarrow k$ is pairwise commuting.
That is, $\sk A$ is a $\s$\=/ring. 

It is immediate from the way we have defined each $\s_j\colon\sk A\rightarrow \sk A$ that the $k$\=/algebra structure map for $\sk A$ is a morphism of $\s$\=/rings. Therefore, $[\s]_kA$ is a $k$\=/$\s$\=/algebra as required.
\end{proof}

Later, we will use the fact that $\sk A\otimes_k \sk B=\sk(A\otimes_k B)$ for $k$\=/algebras $A$ and $B$. This is an immediate consequence of the way we have constructed $\sk A$ and $\sk B$.

\begin{example}\label[example]{example: sk A}
For a $\s$\=/field $k$, if $A=k[y_1,\dots,y_s]$, using \Cref{ex: tau A and A[i]}, \begin{align*}
    \sk A=\bigcup_{i\in \mathbb{N}}k\bigl[\{\tau (y_j)\ |\ \tau\in T_\s[i], 1\leq j\leq s\}\bigr]=k\{y_1,\dots,y_s\},
\end{align*}
which is a $k$\=/$\s$\=/algebra in the natural way for the ring of $\s$\=/polynomials as described in \Cref{ex: poly ring k sigma alg}. More generally, if $A=k[y_1,\dots,y_s]/(F)$ for some $F\subseteq k[y_1,\dots,y_s]$, \begin{align*}
    \sk A=\bigcup_{i\in\mathbb{N}}\Bigl(k\bigl[\{\tau (y_j)\ |\ \tau\in T_\s[i], 1\leq j\leq s\}\bigr]\big/\bigl(\{\tau(F)\ |\ \tau\in T_\s[i]\}\bigr)\Bigr)=k\{y_1,\dots,y_s\}/[F].
\end{align*}
\end{example}

As $A=A[0]$, there is a natural inclusion $\iota_A\colon A\hookrightarrow \sk A$ into the union $[\s]_kA=\bigcup_{i\in \mathbb{N}}A[i]$. 
We will now show that there is a universal property for this inclusion.

\begin{lemma}\label[lemma]{lem: universal prop}
    For each $k$\=/algebra $A$ there exists a $k$\=/$\s$\=/algebra $[\s]_kA$ and a morphism of $k$\=/algebras $\iota_A\colon  A\hookrightarrow [\s]_kA$  such that for every $k$\=/$\s$\=/algebra $R$ and every morphism $\psi\colon  A\rightarrow R$ of $k$\=/algebras, there exists a unique morphism $\varphi\colon [\s]_kA\rightarrow R$ of $k$\=/$\s$\=/algebras making the diagram in \Cref{fig: univ prop diag} commute.\begin{figure}[!htbp]
\centering
\begin{tikzpicture}  
\node (x)   {$A$};
  \node (z) at ([xshift=3cm]$(x)$) {$[\s]_kA$};  
   \node (s) at ([yshift=-1.2cm]$(x)!0.5!(z)$) {$R$}; 
  \draw[right hook->] (x)-- (z) node[midway,above, color=black] {$\iota_A$};  
  \draw[line] (x)-- (s) node[midway,left, color=black] {$\psi\ $};
  \draw[dashed, ->] (z)--(s) node[midway,right, color=black] {$\ \varphi$};
\end{tikzpicture} 
\caption{Universal Property for $A\hookrightarrow \sk A$}
\label{fig: univ prop diag}
\end{figure}
\end{lemma}
\begin{proof}
We have already defined the $k$\=/$\s$\=/algebra $\sk A$ and the inclusion map $A=A[0]\hookrightarrow \sk A$. Now suppose that $R$ is a $k$\=/$\s$\=/algebra and that $\psi\colon A\rightarrow R$ is a $k$\=/algebra morphism.
    We will determine $\varphi$ subject to the conditions necessary, and as we will not make any choices, this will be the unique morphism of $k$\=/$\s$\=/algebras making the diagram commute. 
    
        Firstly, given $r\in A$, we must have $\varphi(r\otimes 1)=\varphi(\iota_A(r))=\psi(r)$ in order to make the diagram  commute. 
        Notice that given any $i\in\mathbb{N}$, $\tau\in T_\s[i]$, for an elementary element $r\otimes\lambda \in {^\tau A}$, \begin{align*}
            \theta^\tau(r\otimes \lambda)=\lambda\theta^\tau(r\otimes 1) =\lambda\tau(r\otimes 1)\in A[i],
        \end{align*}  where the first equality holds as $\theta^\tau$ is a $k$\=/algebra morphism, and the second is due to the definition (\ref{eq: endomorphisms act on sk A}) of the shift maps on $A[i]$.
    Then, as $\varphi$ must be a morphism of $k$\=/$\s$\=/algebras, for any $i\in\mathbb{N}$, $\tau\in T_\s[i]$,
    \begin{align*}
        \varphi\bigl(\theta^\tau(r\otimes\lambda)\bigr)
    =\varphi\bigl(\lambda\tau(r\otimes 1)\bigr) 
        =\lambda\tau\bigl(\varphi(r\otimes 1)\bigr)
        =\lambda\tau\bigl(\psi(r)\bigr) 
    \end{align*}
for any elementary element $r\otimes\lambda\in {^\tau}A$.
Further, as $\varphi$ must be a ring morphism, for any $i\in\mathbb{N}$ we have
\begin{align*}
    \varphi\left(\prod_{\tau\in T_\s[i]}\theta^\tau(r_\tau\otimes\lambda_\tau)\right)=\prod_{\tau\in T_\s[i]}\varphi\bigl(\theta^\tau(r_\tau\otimes\lambda_\tau)\bigr)=\prod_{\tau\in T_\s[i]} \lambda_{\tau}\tau\bigl(\psi(r_\tau)\bigr)
\end{align*}
for any elementary element of $A[i]$. So we have found how $\varphi$ acts on an elementary element of $A[i]$ and this extends uniquely to $\sk A$ in such a way that $\varphi\colon \sk A\rightarrow R$ is a morphism of $k$\=/$\s$\=/algebras. 
\end{proof}

Given $k$\=/algebras $A$ and $B$ and a morphism $f\colon A\rightarrow B$ of $k$\=/algebras, since $\iota_B\circ f\colon A\rightarrow\sk B$ is a morphism of $k$\=/algebras, we can define the morphism $[\s]_k f\colon\sk A\rightarrow \sk B$ to be the unique morphism of $k$\=/$\s$\=/algebras such that $\iota_B\circ f=\iota_A\circ \sk f$ as per \Cref{lem: universal prop}.
It is straightforward to check that this works in a way that makes $\sk$ a functor from $k$\=/$\operatorname{Alg}$ to $k$\=/$\s$\=/$\operatorname{Alg}$.

\begin{corollary}\label[corollary]{cor: Hom functor stuff}
    For a $k$\=/algebra $A$ and a $k$\=/$\s$\=/algebra $R$, as sets, \begin{align*}
        \operatorname{Hom}_{k\text{\=/}\s\text{\=/}\operatorname{Alg}}(\sk A,R)\cong \operatorname{Hom}_{k\text{\=/}\operatorname{Alg}}(A,R^\#).
    \end{align*}
    That is, the $\sk$ functor from $k$\=/$\operatorname{Alg}$ to $k$\=/$\s$\=/$\operatorname{Alg}$ is left\=/adjoint to the $(\=/)^\#$ functor from $k$\=/$\s$\=/$\operatorname{Alg}$ to $k$\=/$\operatorname{Alg}$. 
\end{corollary}

We can now see that this $\sk $ construction on $k$\=/algebras corresponds to the $\sk$ construction on affine schemes, that $\sk k[\X]=k\{\sk\X\}$.

\begin{proposition}\label[proposition]{prop: X affine scheme skX s variety}
    Let $\X$ be an affine scheme of finite type over a $\s$\=/field $k$. Then the functor $\sk \X$ from $k$\=/$\s$\=/$\operatorname{Alg}$ to $\operatorname{Sets}$ defined by $\sk\X(R)=\X(R^\#)$ is a $\s$\=/variety over $k$.
\end{proposition}
\begin{proof}
    We know that $\X$ is represented by some finitely generated $k$-algebra $A$, and hence as sets,   \begin{align*}
        \operatorname{Hom}_{k\text{\=/}\s\text{\=/}\operatorname{Alg}}(\sk A,R)\cong \operatorname{Hom}_{k\text{\=/}\operatorname{Alg}}(A,R^\#)\cong \X(R^\#)=\sk \X(R)
    \end{align*} for any $k$\=/$\s$\=/algebra $R$ by \Cref{cor: Hom functor stuff}. Therefore the functor $\sk \X$ is represented by $\sk A$. Since $A$ is finitely generated as a $k$\=/algebra, $\sk A$ is finitely $\s$\=/generated as a $k$\=/$\s$\=/algebra. This implies that $\sk \X$ is a functor from $k$-$\s$\=/$\operatorname{Alg}$ to $\operatorname{Sets}$ which is represented by a finitely $\s$\=/generated $k$\=/$\s$\=/algebra, hence is a $\s$\=/variety over $k$ as required.
\end{proof}

Now we will often write $\X$ instead of $[\s]_k\X$. We will write $k\{\X\}$ instead of $k\{[\s]_k\X\}$ and by a $\s$\=/closed $\s$\=/subvariety of $\X$, we mean a $\s$\=/closed $\s$\=/subvariety of $[\s]_k\X$.

\section{Difference Algebraic Groups}\label{sec: diff alg goups}

\subsection{Difference Algebraic Groups and Difference Hopf Algebras}

We will now introduce the difference analogue to algebraic groups and consider their dual equivalence to difference Hopf algebras. We will see the $\s$\=/algebraic group and $\s$\=/Hopf algebra analogues of many of the results for $\s$\=/varieties and $\s$\=/algebras from the previous section.
The concepts and results in this section are adapted from \cite[Section 2]{wibmer2022finiteness} to the partial case. 

\begin{definition}
    A \textbf{difference algebraic group} ($\sigma$\=/algebraic group) over $k$ is a group object in the category of $\s$\=/varieties over $k$. 
\end{definition}

Equivalently, a $\s$\=/algebraic group $G$ is a functor from $k$\=/$\s$\=/$\operatorname{Alg}$ to $\operatorname{Groups}$ that is a $\s$\=/variety over $k$ when considered as a functor from $k$\=/$\s$\=/$\operatorname{Alg}$ to $\operatorname{Sets}$, or a functor from $k$\=/$\s$\=/$\operatorname{Alg}$ to $\operatorname{Groups}$ that is represented by a finitely $\s$\=/generated $k$\=/$\s$\=/algebra.

\begin{example}
        Consider the additive $\s$-algebraic group $G$ over a $\s$\=/field $k$, where $\s=\{\s_1,\s_2\}$, such that \begin{align*}
            G(R)=\{g\in R\ |\  \s_1^2\s_2(g)+2\s_2^3(g)=0\}.
        \end{align*}
        for any $k$-$\s$-algebra $R$.
        Then $\I(G)=[\s_1^2\s_2(y)+2\s_2^3(y)]\subseteq k\{y\}$, and $k\{G\}=k\{y\}/\I(G)$. In fact, any set of homogeneous linear $\s$\=/polynomials defines an additive $\s$-algebraic group. In particular, the example (\ref{ex: intro example symbolic dynam}) given in the introduction is an example of an additive $\s$\=/algebraic group which is an analogue to an example in algebraic dynamics.
    \end{example}

    \begin{example}\label[example]{ex:introducing example}
        Consider the multiplicative $\s$-algebraic group $G$ over a $\s$\=/field $k$, where $\s=\{\s_1,\s_2\}$, such that \begin{align*}
            G(R)=\{g\in R^\times\ |\ \s_1^2\s_2(g)\s_2^4(g)=1\}.
        \end{align*}
        for any $k$-$\s$-algebra $R$. Then $\I(G)=[\s_1^2\s_2(y)\s_2^4(y)-1]\subseteq k\{y,y^{-1}\}$, and $k\{G\}=k\{y,y^{-1}\}/\I(G)$. In fact, any set of $\s$\=/polynomials of the form $f-1$ where $f$ is a monomial defines a multiplicative $\s$-algebraic group. 
    \end{example}

We will see (\Cref{prop: alg grp is diff alg gr}) that algebraic groups provide a wealth of examples of $\s$\=/algebraic groups. Intuitively, as an algebraic group is defined by some given polynomials, and these polynomials are implicitly $\s$\=/polynomials, we can consider the $\s$\=/algebraic group defined by them. This is equivalent to the $\sk$ construction on affine schemes giving a $\s$\=/variety (\Cref{prop: X affine scheme skX s variety}).

    A \textbf{morphism} of $\sigma$\=/algebraic groups $\phi\colon G\rightarrow H$ is a morphism of $\sigma$\=/varieties that respects the group structure.

 \begin{definition}
     Let $G$ be a $\sigma$\=/algebraic group over $k$. A \textbf{$\sigma$\=/closed subgroup} $H$ of $G$ is a $\sigma$\=/closed $\sigma$\=/subvariety $H$ of $G$ such that for every $k$\=/$\s$\=/algebra $R$, $H(R)$ is a subgroup of $G(R)$. Then $H$ itself is a $\sigma$\=/algebraic group over $k$. We write $H\leq G$ to express that $H$ is a $\sigma$\=/closed subgroup of $G$.
 \end{definition}

 Recall that the coordinate ring of an algebraic group over $k$ has the structure of a $k$\=/Hopf algebra. See \cite[Chapter IV]{sweedler1969hopf} for an introduction to Hopf algebras, and \cite[Section 1.4]{waterhouse2012introduction} for a description of their duality to algebraic groups. We now introduce the difference analogue to Hopf algebras, which will be the representing algebras for difference algebraic groups.

\begin{definition}
    A \textbf{$k$\=/$\sigma$\=/Hopf algebra} is a $k$\=/Hopf algebra $A$ that is also a $k$\=/$\sigma$\=/algebra such that the Hopf algebra structure maps $\Delta\colon A\rightarrow A\otimes_k A$, $S\colon A\rightarrow A$ and $\epsilon\colon  A\rightarrow k$ are morphisms of $k$\=/$\sigma$\=/algebras.
\end{definition}

Where necessary for clarity, we use notation $\Delta_A\colon A\rightarrow A\otimes_k A$, $S_A\colon A\rightarrow A$ and $\epsilon_A\colon  A\rightarrow k$ to denote the Hopf structure maps of a $k$\=/$\s$\=/Hopf algebra $A$.

For a $k$\=/$\s$\=/Hopf algebra $A$, a \textbf{$k$\=/$\s$\=/Hopf subalgebra} of $A$ is a $k$\=/$\s$\=/subalgebra of $A$ that is also a $k$\=/Hopf subalgebra of $A$, and a \textbf{$\s$\=/Hopf ideal} of $A$ is a $\s$\=/ideal of $A$ that is also a Hopf ideal of $A$. A \textbf{morphism of $k$\=/$\s$\=/Hopf algebras} is a morphism of $k$\=/$\s$\=/algebras that is also a morphism of $k$\=/Hopf algebras.

\begin{proposition}\label[proposition]{prop: one to one corresp groups and Hopf algs}
    A functor from $k$\=/$\s$\=/$\operatorname{Alg}$ to $\operatorname{Sets}$ is a $\s$\=/algebraic group over $k$ if and only if it is represented by a finitely $\s$\=/generated $k$\=/$\s$\=/Hopf algebra.
\end{proposition}
\begin{proof}
    This is the difference analogue to \cite[Section 1.4, page 7]{waterhouse2012introduction}. 
    The correspondence between $\s$\=/varieties and finitely $\s$\=/generated $k$\=/$\s$\=/algebras described in \Cref{prop: prelims varieties and algebras corresp} restricts to a correspondence between $\s$\=/algebraic groups and finitely $\s$\=/generated $k$\=/$\s$\=/Hopf algebras. This is due to the fact that given a $\s$\=/variety $G$ over $k$ with representing $k$\=/$\s$\=/algebra $A$, there exist morphisms $\mult\colon G\times G\rightarrow G$ and $\unit\colon e\rightarrow G$ and $\inv\colon G\rightarrow G$ of functors which commute in such a way that make each $G(R)$ a group if and only if there exist dual morphisms (as described in \Cref{prop: morphisms corresp dual}) $\Delta\colon A\rightarrow A\otimes_k A$, $\epsilon\colon A\rightarrow k$ and $S \colon A\rightarrow A$ of $k$\=/$\s$\=/algebras which commute in such a way that makes $A$ a $k$\=/$\s$\=/Hopf algebra.
\end{proof}

\begin{proposition}\label[proposition]{prop: one to one corresp subgroups and Hopf ideals}
    Let $G$ be a $\s$\=/algebraic group over $k$, represented by a $k$\=/$\s$\=/Hopf algebra $A$. 
    Given a $\s$\=/Hopf ideal $I\subseteq A$, there exists a $\s$-closed subgroup $H$ of $G$ that is represented by $A/I$, and for any $\s$-closed subgroup $H$ of $G$, there exists some $\s$\=/Hopf ideal $I\subseteq A$ such that $H$ is represented by $A/I$.
\end{proposition}
\begin{proof}
    This is a combination of \Cref{lem: subvarieties and ideals,prop: one to one corresp groups and Hopf algs} and the fact that given a $k$\=/$\s$\=/Hopf algebra $A$ and $\s$\=/ideal $I\subseteq A$, the quotient $k$\=/$\s$\=/algebra $A/I$ is a $k$\=/$\s$\=/Hopf algebra if and only if $I$ is a $\s$\=/Hopf ideal of $A$. This fact follows directly from the analogous fact for $k$\=/Hopf algebras and Hopf ideals \cite[Section 4.3, page 87]{sweedler1969hopf}.
\end{proof}

\begin{corollary}\label[corollary]{cor: morphisms groups and hopf dual}
    Let $G$ and $H$ be $\s$\=/algebraic groups over $k$ represented by finitely $\s$\=/generated $k$\=/$\s$\=/Hopf algebras $A$ and $B$ respectively. A morphism $\phi\colon G\rightarrow H$ of $\s$\=/varieties is a morphism of $\s$\=/algebraic groups if and only if the dual morphism $\phi^*\colon B\rightarrow A$ of $k$\=/$\s$\=/algebras is a morphism of $k$\=/$\s$\=/Hopf algebras.
\end{corollary}
\begin{proof}
This follows as the properties required for a map $\phi\colon G\rightarrow H$ to be a morphism of $\s$\=/algebraic groups are precisely the dual properties to those required for the dual map $\phi^*\colon B\rightarrow A$ to be a morphism of $k$\=/$\s$\=/Hopf algebras.
\end{proof}

We see that the dual equivalence of the category of $\s$\=/varieties over $k$ and the category of finitely $\s$\=/generated $k$\=/$\s$\=/algebras as described in \Cref{the: equiv of cate} restricts to a dual equivalence of the category of $\s$\=/algebraic groups over $k$ and the category of finitely $\s$\=/generated $k$\=/$\s$\=/Hopf algebras.

     Let $\phi\colon G\rightarrow H$ be a morphism of $\s$\=/algebraic groups over $k$, and recall from \Cref{lem: image of morphism} that the image $\phi(G)$ of $\phi$ is the $\s$\=/closed $\s$\=/subvariety of $H$ defined by $\ker(\phi^*)$, where $\phi^*$ is the dual map to $\phi$. Since $\phi$ is a morphism of $\s$\=/algebraic groups over $k$, $\phi^*$ is a morphism of $k$\=/$\s$\=/Hopf algebras and hence $\ker(\phi^*)$ is a $\s$\=/Hopf ideal of $k\{H\}$. Then by \Cref{prop: one to one corresp subgroups and Hopf ideals}, $\phi(G)$ is a $\s$\=/closed subgroup of $H$.

\begin{lemma}\label[lemma]{lem: sk A hopf algebra}
    If $A$ is a $k$\=/Hopf algebra, then $\sk A$ is a $k$\=/$\s$\=/Hopf algebra.
\end{lemma}
\begin{proof}
Let $A$ be a $k$\=/algebra.
We know from \Cref{prop: [s]kA ks algebra} that $\sk A$ is a $k$\=/$\s$\=/algebra.
    Since $\sk A$ is a union of tensor products of $k$\=/Hopf algebras, it is in fact a $k$\=/Hopf algebra. The Hopf algebra structure maps on $\sk A$ are the maps uniquely determined by \Cref{lem: universal prop} (the universal property) to commute with the Hopf structure maps on $A$, and hence are morphisms of $k$\=/$\s$\=/algebras.
    That is, the Hopf and difference structures on $\sk A$ are compatible, hence $\sk A$ is a $k$\=/$\s$\=/Hopf algebra as required.
\end{proof}

\begin{lemma}\label[lemma]{lem: univ hopf prop}
    Let $A$ be a $k$\=/Hopf algebra, $B$ a $k$\=/$\sigma$\=/Hopf algebra, and $\psi\colon  A\rightarrow B$ a morphism of $k$\=/Hopf algebras. Then the morphism $\varphi\colon  [\sigma]_kA\rightarrow B$ of $k$\=/$\sigma$\=/algebras induced by the universal property (\Cref{lem: universal prop}) is a morphism of $k$\=/$\s$\=/Hopf algebras.
\end{lemma}
\begin{proof}
This is the partial difference algebra analogue to Lemma 2.15 in \cite[Section 2, page 524]{waterhouse2012introduction}. For example, to see that $\varphi$ commutes with the comultiplications $\Delta_B$ of $B$ and $\Delta_{\sk A}$ of $\sk A$, notice that both $\Delta_B\circ \varphi$ and $(\varphi\otimes\varphi)\circ\Delta_{\sk A}$ are morphisms of $k$\=/$\s$\=/algebras making the diagram \begin{align*}
    \begin{tikzpicture}  
\node (x)   {$A$};
  \node (z) at ([xshift=3cm]$(x)$) {$[\sigma]_kA$};  
   \node (s) at ([yshift=-1cm]$(x)!0.5!(z)$) {$B\otimes B$}; 
  \draw[right hook->] (x)-- (z) node[midway,above, color=black] {$\iota_A$};  
  \draw[line] (x)-- (s) node[pos=0.75, left, color=black] {$\Delta_B\circ\psi\ \ $};
  \draw[dashed, ->] (z)--(s) node[midway,right, color=black] {};
\end{tikzpicture} 
\end{align*}
commute, and hence by uniqueness of the universal property, are equal. The proof that $\varphi$ commutes with the counit follows a similar method.
\end{proof}

\begin{proposition}\label[proposition]{prop: alg grp is diff alg gr}
    Let $\G$ be an algebraic group over $k$. The functor $\sk \G$ from $k$\=/$\s$\=/$\operatorname{Alg}$ to $\operatorname{Groups}$, where for a $k$\=/$\s$\=/algebra $R$, $\sk \G(R)=\G(R^\#)$, is a $\s$\=/algebraic group over $k$.
\end{proposition}
\begin{proof}
    Since $\G$ can be considered as an affine scheme of finite type over $k$, $\sk \G$ is a $\s$\=/variety over $k$ by \Cref{prop: X affine scheme skX s variety}. Therefore, $\sk\G$ is represented by a finitely $\s$\=/generated $k$\=/$\s$\=/algebra by \Cref{prop: prelims varieties and algebras corresp}. That is, $\sk \G$ is a functor from $k$-$\s$\=/$\operatorname{Alg}$ to $\operatorname{Groups}$ represented by a finitely $\s$\=/generated $k$\=/$\s$\=/algebra, therefore $\sk \G$ is a $\s$\=/algebraic group over $k$.
\end{proof}

\begin{example} For the general linear group $\operatorname{GL}_s$ over $k$ for some $s\geq 1$ as defined in \Cref{example: algebraic groups}, we can consider the corresponding $\s$\=/algebraic group $\sk \operatorname{GL}_s$ which takes a $k$\=/$\s$\=/algebra $R$ and returns the set of $s\times s$ invertible matrices with entries in $R$. Since $k[\operatorname{GL}_s]=k[X,1/\det(X)]$ where $X=(x_{i,j})_{1\leq i,j\leq s}$, \Cref{example: sk A} tells us that $k\{\operatorname{GL}_s\}=k\{X,1/\det(X)\}$. This is a $k$\=/$\s$\=/Hopf algebra.
\end{example}

Again we will abuse notation by writing $\G$ instead of $\sk \G$, $k\{\G\}$ instead of $k\{\sk\G\}$ and by a $\s$\=/closed subgroup of $\G$, we mean a $\s$\=/closed subgroup of $\sk \G$.

So any algebraic group provides an example of a $\s$\=/algebraic group. In \Cref{sec: building difference structure}, we will see that these particular $\s$\=/algebraic groups have a rich structure. In fact, we will now see that any $\s$\=/algebraic group can be embedded into a $\s$\=/algebraic group coming from an algebraic group, hence allowing us to induce information about the structure of any $\s$\=/algebraic group.

\begin{proposition}\label[proposition]{prop: embedding into gln}
    Let $G$ be a $\sigma$\=/algebraic group over $k$. Then $G$ is isomorphic to a $\sigma$\=/closed subgroup of $\operatorname{GL}_s$ over $k$ for some $s\geq 1$.
\end{proposition}
\begin{proof}
We know that $k\{G\}$ is a finitely $\s$\=/generated $k$\=/$\s$\=/Hopf algebra. That is, $k\{G\}=k\{F\}$ for some finite subset $F\subseteq k\{G\}$.
By \cite[Section 3.3, page 24]{waterhouse2012introduction}, every finite subset of a $k$\=/Hopf algebra is contained in some finitely generated $k$\=/Hopf subalgebra of it. Therefore, there is a finitely generated $k$\=/Hopf subalgebra $A$ of $k\{G\}$ that contains $F$. 
We know that $A$ represents some algebraic group over $k$, which, by \cite[Section 3.4, page 25]{waterhouse2012introduction}, is isomorphic to a closed subgroup of $\operatorname{GL}_s$ for some $s\in\mathbb{N}$. That is, $k[\operatorname{GL}_s]\rightarrow A$ is a surjective morphism of $k$\=/Hopf algebras.
Therefore we have a morphism $\psi^*\colon  k[\operatorname{GL}_s]\rightarrow A\hookrightarrow k\{G\}$ of $k$\=/Hopf algebras.
    By the universal property (\Cref{lem: universal prop}), there is a unique $k$\=/$\sigma$\=/algebra morphism $\phi^*\colon k\{\operatorname{GL}_n\}\rightarrow k\{G\}$ making the diagram \begin{align*}
        \begin{tikzpicture}  
\node (x)   {$k[\operatorname{GL}_s]$};
  \node (z) at ([xshift=3cm]$(x)$) {$k\{\operatorname{GL}_s\}$}; 
   \node (s) at ([yshift=-1cm]$(x)!0.5!(z)$) {$k\{G\}$}; 
  \draw[right hook->] (x)-- (z) node[midway,above, color=black] {};  
  \draw[line] (x)-- (s) node[pos=0.75,left, color=black] {$\psi^*\ \ $};
  \draw[dashed, ->] (z)--(s) node[pos=0.75,right, color=black] {$\ \phi^*$};
\end{tikzpicture} 
    \end{align*}
    commute. By \Cref{lem: univ hopf prop}, $\phi^*$ is a morphism of $k$\=/$\s$\=/Hopf algebras. 
    We know that $F$ lies in the image of $\psi^*\colon k[\operatorname{GL}_s]\rightarrow A\hookrightarrow k\{G\}$, and hence lies in the image of the $k$\=/$\s$\=/algebra morphism $\phi^*\colon k\{\operatorname{GL}_s\}\rightarrow k\{G\}$. As $F$ generates $k\{G\}$ as a $k$\=/$\sigma$\=/algebra, and the image of $\phi^*$ is a $k$\=/$\s$\=/algebra, we see that $\phi^*$ is surjective. Therefore, the dual map $\phi\colon G\hookrightarrow \operatorname{GL}_n$ is a $\s$\=/closed embedding by \Cref{lem: embedding iff surjective}.
    By \Cref{cor: morphisms groups and hopf dual}, since $\phi^*$ is a morphism of $k$\=/$\s$\=/Hopf algebras, $\phi$ is a morphism of $\s$\=/algebraic groups. Further, since $\phi$ is a $\s$\=/closed embedding,
$G$ is isomorphic to $\phi(G)$, which is a $\s$\=/closed subgroup of $\operatorname{GL}_s$.
\end{proof}

\begin{example}\label[example]{ex:embedding into Gm}
Let $\s=\{\s_1,\s_2\}$ and let $k$ be a $\s$\=/field.
    The $\s$-algebraic group $G$ over $k$ introduced in \Cref{ex:introducing example}, where
            $G(R)=\{g\in R^\times\ |\   \s_1^2\s_2(g)\s_2^4(g)=1\}$
         for any $k$-$\s$-algebra $R$, is a $\s$\=/closed subgroup of $\operatorname{GL}_1=\mathbb{G}_m$.  
    \end{example}

    Therefore, any $\s$\=/algebraic group is defined as a $\s$\=/closed subgroup of some algebraic group $\G$ by a $\s$\=/Hopf ideal. We will see (\Cref{cor: defining ideal fin generated}) that this $\s$\=/Hopf ideal is finitely $\s$\=/generated.

\subsection{A Geometric Method for Building Difference Structure}\label{sec: defining G[i]s}

We will now introduce useful notation for building difference structure on an algebraic group over a $\s$\=/field using base change, and introduce diagrams which will help us visualise many constructions used in this paper.
The constructions in this section are adapted from \cite[Section 3]{wibmer2022finiteness} to the partial case. 

Recall that given a $k$\=/algebra $A$, we constructed the $k$\=/$\s$\=/algebra $\sk A$ via a $k$\=/algebra ${^\tau}A$ for each $\tau\in T_\s$, and a $k$\=/algebra $A[i]$ for each $i\in\mathbb{N}$ (\Cref{def: tau A and A[i]}). Then $\sk A=\cup_{i\in\mathbb{N}}A[i]$ is a $k$\=/$\s$\=/algebra (\Cref{prop: [s]kA ks algebra}). 
If $A$ is the representing $k$\=/algebra for an algebraic group $\G$ over $k$, we saw that $\sk A$ is the representing $k$\=/$\s$\=/algebra for the corresponding $\s$\=/algebraic group $\sk \G$ (see the proof of \Cref{prop: alg grp is diff alg gr}). We will see that we can also conduct a step-by-step construction on the algebraic group side.

Given $\tau\in T_\s$, we noted in \Cref{rem: tau constructions agree} that ${^\tau}A$ is the representing $k$\=/algebra for the algebraic group ${^\tau}\G$ obtained from $\G$ by base change via $\tau\colon  k\rightarrow k$ (\Cref{def: group base change}). That is, $k[{^\tau}\G]={^\tau}k[\G]$. Now, define
\begin{align}\label{eq: G[i]}
    \quad \G[i]=\prod_{\tau\in T_\s[i]}{^\tau}\G
\end{align}
for each $i\in\mathbb{N}$. Notice that since the product of algebraic groups over $k$ is again an algebraic group over $k$, $\G[i]$ is an algebraic group over $k$ for every $i\in\mathbb{N}$. In particular, notice that \begin{align*}
    k\Bigl[\G[i]\Bigr]=k\left[\prod_{\tau\in T_\s[i]}{^\tau}\G \right]=\bigotimes_{\tau\in T_\s[i]}k[{^\tau}\G] =\bigotimes_{\tau\in T_\s[i]}{^\tau}k[\G] &=\Bigl(k[\G]\Bigr)[i],
\end{align*}
where the second equality is due to properties of products of algebraic groups \cite[Section 1.4, page 7]{waterhouse2012introduction}. That is, if $A$ is the representing $k$\=/algebra for $\G$, then for any $i\in\mathbb{N}$, $\G[i]$ is represented by $A[i]$. Then, \begin{align*}
    \bigcup_{i\in\mathbb{N}}k\Bigl[\G[i]\Bigr]=\bigcup_{i\in\mathbb{N}}\Bigl(k[\G]\Bigr)[i]=\sk\Bigl(k[\G]\Bigr)=k\{\G\}
\end{align*}
where we are using the notation $k\{\G\}$ to mean $k\bigl\{[\s]_k\mathcal{G}\bigr\}$. Therefore, we essentially have growing products of algebraic groups $\G[i]$ for $i\in\mathbb{N}$, the union of whose representing algebras is the representing algebra for the corresponding $\s$\=/algebraic group. We will use that fact that any general $\s$\=/algebraic group can be embedded into one of these $\s$\=/algebraic groups obtained from algebraic groups (\Cref{prop: embedding into gln}) to `split' up our $\s$\=/algebraic groups into algebraic groups.

Throughout this document we will use diagrams such as \Cref{fig: g3 when n=2 example} to visualise (in the two endomorphisms case) the constructions we work with. A dot on a particular part of the diagram means this construction has information in this given ${^\tau}\G$ (see \Cref{fig: describing figures}). The diagrams use black dots to show the information of this order is all of ${^\tau}\G$, and red dots to show there is not the full ${^\tau}\G$, but some closed subgroup of it (see \Cref{fig comp: Gi s1 Gi and extension s1 Gi}). Projection maps \begin{align*}
    \prod_{\tau\in A}{^\tau \G}\rightarrow \prod_{\tau\in B}{^\tau}\G \quad \text{for some subsets}\quad B\subseteq A\subseteq T_\s
\end{align*} forget some information, so forget some dots. We will leave lighter dots in the place to show where the forgotten information was (see \Cref{fig: pi[i] and sj[i] two endomorphism case}). The entry at a point will be $1$ if in our given group, the ${^\tau\G}$ component is just $1_{(^\tau\G)}={^\tau}(1_\G)$ (see \Cref{fig ker: G[i] H[i] rho[i]H[i]}). Of course these figures can't be used to prove properties, they are useful to visualise constructions. Finally, we will ignore the true value of $i$ in these diagrams. We will, for example, say a diagram is displaying $\G[i]$, when really it is showing, for example, $\G[3]$. 

\begin{figure}[!htbp]
\centering
\begin{subfigure}[b]{0.45\textwidth}
\centering
\centering
\begin{tikzpicture}
\draw[step=1cm,gray,very thin] (-0.59,-0.59) grid (3.9,3.9);
\draw[thick,->] (0,0) -- (3.9,0);
\node at (4,-0.3) {$\s_1$};
\draw[thick,->] (0,0) -- (0,3.9);
\node at (-0.4,3.9) {$\s_2$};
\node[fill,circle,radius=0.4cm,inner sep=0pt,white] at (0,0) {\footnotesize ${^{\s_1\s_2}}\G$};
\node[draw,circle,radius=0.4cm,inner sep=0pt,black] at (0,0) {\footnotesize \textcolor{white}{${^{\s_1\s_2}}\G$}};
\node at (0,0) { \footnotesize $\G$};
\node[fill,circle,radius=0.4cm,inner sep=0pt,white] at (1,0) {\footnotesize ${^{\s_1\s_2}}\G$};
\node[draw,circle,radius=0.4cm,inner sep=0pt,black] at (1,0) {\footnotesize \textcolor{white}{${^{\s_1\s_2}}\G$}};
\node at (1,0) {\footnotesize\scalebox{.75}[1.0]{$^{\s_1}$}$\G$};
\node[fill,circle,radius=0.4cm,inner sep=0pt,white] at (2,0) {\footnotesize ${^{\s_1\s_2}}\G$};
\node[draw,circle,radius=0.4cm,inner sep=0pt,black] at (2,0) {\footnotesize \textcolor{white}{${^{\s_1\s_2}}\G$}};
\node at (2,0.057) {\footnotesize\scalebox{.75}[1.0]{$^{\s_1^2}$}$\G$};
\node[fill,circle,radius=0.4cm,inner sep=0pt,white] at (3,0) {\footnotesize ${^{\s_1\s_2}}\G$};
\node[draw,circle,radius=0.4cm,inner sep=0pt,black] at (3,0) {\footnotesize \textcolor{white}{${^{\s_1\s_2}}\G$}};
\node at (3,0.057) {\footnotesize\scalebox{.75}[1.0]{$^{\s_1^3}$}$\G$};
\node[fill,circle,radius=0.4cm,inner sep=0pt,white] at (0,1) {\footnotesize ${^{\s_1\s_2}}\G$};
\node[draw,circle,radius=0.4cm,inner sep=0pt,black] at (0,1) {\footnotesize \textcolor{white}{${^{\s_1\s_2}}\G$}};
\node at (0,1) {\footnotesize\scalebox{.75}[1.0]{$^{\s_2}$}$\G$};
\node[fill,circle,radius=0.4cm,inner sep=0pt,white] at (1,1) {\footnotesize ${^{\s_1\s_2}}\G$};
\node[draw,circle,radius=0.4cm,inner sep=0pt,black] at (1,1) {\footnotesize \textcolor{white}{${^{\s_1\s_2}}\G$}};
\node at (1,1) {\footnotesize\scalebox{.75}[1.0]{$^{\s_1\s_2}$}$\G$};
\node[fill,circle,radius=0.4cm,inner sep=0pt,white] at (2,1) {\footnotesize ${^{\s_1\s_2}}\G$};
\node[draw,circle,radius=0.4cm,inner sep=0pt,black] at (2,1) {\footnotesize \textcolor{white}{${^{\s_1\s_2}}\G$}};
\node at (2,1.057) {\footnotesize\scalebox{.75}[1.0]{$^{\s_1^2\s_2}$}$\G$};
\node[fill,circle,radius=0.4cm,inner sep=0pt,white] at (0,2) {\footnotesize ${^{\s_1\s_2}}\G$};
\node[draw,circle,radius=0.4cm,inner sep=0pt,black] at (0,2) {\footnotesize \textcolor{white}{${^{\s_1\s_2}}\G$}};
\node at (0,2.057) {\footnotesize\scalebox{.75}[1.0]{$^{\s_2^2}$}$\G$};
\node[fill,circle,radius=0.4cm,inner sep=0pt,white] at (1,2) {\footnotesize ${^{\s_1\s_2}}\G$};
\node[draw,circle,radius=0.4cm,inner sep=0pt,black] at (1,2) {\footnotesize \textcolor{white}{${^{\s_1\s_2}}\G$}};
\node at (1,2.057) {\footnotesize\scalebox{.75}[1.0]{$^{{\s_1}\s_2^2}$}$\G$};
\node[fill,circle,radius=0.4cm,inner sep=0pt,white] at (0,3) {\footnotesize ${^{\s_1\s_2}}\G$};
\node[draw,circle,radius=0.4cm,inner sep=0pt,black] at (0,3) {\footnotesize \textcolor{white}{${^{\s_1\s_2}}\G$}};
\node at (0,3.057) {\footnotesize\scalebox{.75}[1.0]{$^{\s_2^3}$}$\G$};
\end{tikzpicture}
\caption{Dots Correspond to ${^\tau}\G$ for some $\tau\in T_\s$}
\label{fig: describing figures}
\end{subfigure}
\begin{subfigure}[b]{0.45\textwidth}
\centering
\begin{tikzpicture}
\draw[step=1cm,gray,very thin] (-0.5,-0.5) grid (3.9,3.9);
\node at (4,-0.35) {\textcolor{white}{g}};
\draw[thick,->] (0,0) -- (3.9,0);
\node at (4,-0.3) {$\s_1$};
\draw[thick,->] (0,0) -- (0,3.9);
\node at (-0.4,3.9) {$\s_2$};
\fill (0,0) circle(3pt);
\fill (1,0) circle(3pt) (0,1) circle(3pt);
\fill (2,0) circle(3pt) (1,1) circle(3pt) (0,2) circle(3pt);
\fill [radius=3pt] (3,0) circle[] (2,1) circle[] (1,2) circle[] (0,3) circle[];
\end{tikzpicture}
\caption{Visualising $\G[3]$ when $n=2$}
\label{fig: g3 when n=2 example}
\end{subfigure}
\caption{Explaining Visualisation Diagrams}
\end{figure}

Applying base change on these diagrams essentially looks like moving the dots around, see \Cref{fig: showing base change}.

\begin{figure}[!htbp]
\centering\scalebox{0.9}{\begin{tikzpicture}
\draw[step=0.5cm,lightgray,very thin] (-0.25,-0.25) grid (2.45,2.45);
\draw[darkgray,->] (0,0) -- (2.35,0) node[anchor=north west] {$\s_1$};
\draw[darkgray,->] (0,0) -- (0,2.35) node[anchor=south east] {$\s_2$};
\fill (0,0) circle(2pt);
\fill (0.5,0) circle(2pt) (0,0.5) circle(2pt);
\fill (1,0) circle(2pt) (0.5,0.5) circle(2pt) (0,1) circle(2pt);
\fill [radius=2pt] (1.5,0) circle[] (1,0.5) circle[] (0.5,1) circle[] (0,1.5) circle[];
\node at (1,-0.55) {{$\G[i]$}};
\draw[step=0.5cm,lightgray,very thin] (3.75,-0.25) grid (6.45,2.45);
\draw[darkgray,->] (4,0) -- (6.35,0) node[anchor=north west] {$\s_1$};
\draw[darkgray,->] (4,0) -- (4,2.35) node[anchor=south east] {$\s_2$};
\fill (4.5,0) circle(2pt);
\fill (5,0) circle(2pt) (4.5,0.5) circle(2pt);
\fill [radius=2pt] (5.5,0) circle[] (5,0.5) circle[] (4.5,1) circle[];
\node at (5,-0.55) {{${^{\s_1}}\G[i]$}};
\fill [radius=2pt] (6,0) circle[] (5.5,0.5) circle[] (5,1) circle[] (4.5,1.5) circle[];
\draw[step=0.5cm,lightgray,very thin] (7.75,-0.25) grid (10.45,2.45);
\draw[darkgray,->] (8,0) -- (10.35,0) node[anchor=north west] {$\s_1$};
\draw[darkgray,->] (8,0) -- (8,2.35) node[anchor=south east] {$\s_2$};
\fill  (8.5,0.5) circle(2pt);
\fill [radius=2pt] (9,0.5) circle[] (8.5,1) circle[];
\fill [radius=2pt]  (9.5,0.5) circle[] (9,1) circle[] (8.5,1.5) circle[];
\node at (9,-0.55) {${^{\s_2}}\G[i]$};
\fill [radius=2pt]  (8,0.5) circle[] (8,1) circle[] (8,1.5) circle[] (8,2) circle[];
\end{tikzpicture}}
\caption{Visualising Base Change when $n=2$}
\label{fig: showing base change}
\end{figure}

\subsection{Zariski Closures}\label{sec: zariski closures def}
We will consider a method for `splitting' our $\s$\=/algebraic groups into algebraic groups.
Let $\G$ be an algebraic group over $k$. For each $i\geq 1$ and $1\leq j\leq n$, define the projections\begin{align}\label{eq: def of pi map}
    \pi_i\colon  \G[i]&\rightarrow\G[i-1], \qquad
    (x_\tau)_{\tau\in T_\s[i]}\mapsto (x_\tau)_{\tau\in T_\s[i-1]} &&\text{and}\\\label{eq: def of sigma map} (\s_j)_i\colon  \G[i]&\rightarrow{^{\s_j}}\G[i-1],\quad\hspace{0.4mm}
    (x_\tau)_{\tau\in T_\s[i]}\mapsto (x_{\s_j\tau})_{\tau\in T_\s[i-1]}
\end{align}
and notice that these are morphisms of algebraic groups.
See \Cref{fig: pi[i] and sj[i] two endomorphism case} for a visualisation of $\pi_i$, $(\s_1)_i$ and $(\s_2)_i$ in the two endomorphisms case. 
\begin{figure}[!htbp]
\centering
\scalebox{0.9}{\begin{tikzpicture}
\draw[step=0.5cm,lightgray,very thin] (-0.25,-0.25) grid (2.45,2.45);
\draw[darkgray,->] (0,0) -- (2.35,0) node[anchor=north west] {$\s_1$};
\draw[darkgray,->] (0,0) -- (0,2.35) node[anchor=south east] {$\s_2$};
\fill (0,0) circle(2pt);
\fill (0.5,0) circle(2pt) (0,0.5) circle(2pt);
\fill (1,0) circle(2pt) (0.5,0.5) circle(2pt) (0,1) circle(2pt);
\fill [radius=2pt] (1.5,0) circle[] (1,0.5) circle[] (0.5,1) circle[] (0,1.5) circle[];
\fill [radius=2pt] (2,0) circle[] (1.5,0.5) circle[] (1,1) circle[] (0.5,1.5) circle[] (0,2) circle[];
\node at (1,-0.55) {{$\G[i]$}};
\draw[step=0.5cm,lightgray,very thin] (3.75,-0.25) grid (6.45,2.45);
\draw[darkgray,->] (4,0) -- (6.35,0) node[anchor=north west] {$\s_1$};
\draw[darkgray,->] (4,0) -- (4,2.35) node[anchor=south east] {$\s_2$};
\fill (4,0) circle(2pt);
\fill (4.5,0) circle(2pt) (4,0.5) circle(2pt);
\fill (5,0) circle(2pt) (4.5,0.5) circle(2pt) (4,1) circle(2pt);
\fill [radius=2pt] (5.5,0) circle[] (5,0.5) circle[] (4.5,1) circle[] (4,1.5) circle[];
\node at (5,-0.55) {{$\pi_i(\G[i])$}};
\fill[black!30] [radius=2pt] (6,0) circle[] (5.5,0.5) circle[] (5,1) circle[] (4.5,1.5) circle[] (4,2) circle[];
\draw[step=0.5cm,lightgray,very thin] (7.75,-0.25) grid (10.45,2.45);
\draw[darkgray,->] (8,0) -- (10.35,0) node[anchor=north west] {$\s_1$};
\draw[darkgray,->] (8,0) -- (8,2.35) node[anchor=south east] {$\s_2$};
\fill (8.5,0) circle(2pt);
\fill (9,0) circle(2pt) (8.5,0.5) circle(2pt);
\fill [radius=2pt] (9.5,0) circle[] (9,0.5) circle[] (8.5,1) circle[];
\fill [radius=2pt] (10,0) circle[] (9.5,0.5) circle[] (9,1) circle[] (8.5,1.5) circle[];
\node at (9,-0.55) {{$(\s_1)_i\bigl(\G[i]\bigr)$}};
\fill[black!30] [radius=2pt] (8,0) circle[] (8,0.5) circle[] (8,1) circle[] (8,1.5) circle[] (8,2) circle[];
\draw[step=0.5cm,lightgray,very thin] (11.75,-0.25) grid (14.45,2.45);
\draw[darkgray,->] (12,0) -- (14.35,0) node[anchor=north west] {$\s_1$};
\draw[darkgray,->] (12,0) -- (12,2.35) node[anchor=south east] {$\s_2$};
\fill (12,0.5) circle(2pt);
\fill  (12.5,0.5) circle(2pt) (12,1) circle(2pt);
\fill [radius=2pt] (13,0.5) circle[] (12.5,1) circle[] (12,1.5) circle[];
\fill [radius=2pt] (13.5,0.5) circle[] (13,1) circle[] (12.5,1.5) circle[] (12,2) circle[];
\node at (13,-0.55) {$(\s_2)_i\bigl(\G[i]\bigr)$};
\fill[black!30] [radius=2pt] (12,0) circle[] (12.5,0) circle[] (13,0) circle[] (13.5,0) circle[] (14,0) circle[];
\end{tikzpicture}}
\caption{$\pi_i$, $(\s_1)_i$ and $(\s_2)_i$ acting on $\G[i]$ when $n=2$}
\label{fig: pi[i] and sj[i] two endomorphism case}
\end{figure}

\begin{example}
    Consider the general linear group $\operatorname{GL}_s$ over a $\s$\=/field $k$. We saw in \Cref{ex: tau GLs} that for any $\tau\in T_\s$, ${^\tau}\operatorname{GL}_s$ is just a copy of $\operatorname{GL}_s$. Therefore, $\operatorname{GL}_s[i]=\prod_{\tau\in T_\s[i]}{^\tau}\operatorname{GL}_s=\operatorname{GL}_s^{|T_\s[i]|}$ for every $i\in\mathbb{N}$. So for a $k$\=/algebra $R$, an element of $(A_\tau)_{\tau\in T_\s[i]}\in\operatorname{GL}_s[i](R)$ is a sequence of $|T_\s[i]|$ invertible $s\times s$ matrices with entries in $R$. Then for $i\geq 1$, the map $(\pi_i)_R$ simply drops the entries $A_\tau$ where $\ord(\tau)=i$, and for $1\leq j\leq n$, the map $((\s_j)_i)_R$ drops the entries $A_\tau$ where $\s_j|\tau$. 
\end{example}

We will often consider the projection maps defined in (\ref{eq: def of pi map}) and (\ref{eq: def of sigma map}) restricted to closed subgroups of $\G[i]$.

\begin{remark}\label[remark]{lem: restrictions well defined}
Suppose that $i\geq 1$ and we have closed subgroups $G_i\leq \G[i]$ and $G_{i-1}\leq \G[i-1]$. Then $\pi_i\bigl(G_i\bigr)$ is the closed subgroup of $\G[i-1]$ defined by Hopf ideal $\I(G_i)\cap k[\G[i-1]]$ and $(\s_j)_i\left(G_i\right)$ is the closed subgroup of ${^{\s_j}}\G[i-1]$ defined by Hopf ideal $\I(G_i)\cap \s_j(k[\G[i-1]])$.
That is, $\pi_i\bigl(G_i\bigr)\leq G_{i-1}$ if and only if $\I(G_{i-1})\subseteq \I(G_i)$, and $(\s_j)_i\left(G_i\right)\leq {^{\s_j}}G_{i-1}$ if and only if $\s_j(\I(G_{i-1}))\subseteq \I(G_i)$.
\end{remark}

\begin{definition}\label[definition]{def: zariski closures}
    Let $G$ be a $\s$\=/closed subgroup of an algebraic group $\G$ over $k$, defined by the $\s$\=/Hopf ideal $\I(G)\subseteq k\{\G\}$. For each $i\in\mathbb{N}$, let $G[i]$ be the closed subgroup of $\G[i]$ defined by the Hopf ideal \begin{align*}
        \I(G[i])=\I(G)\cap k\bigl[\G[i]\bigr].
    \end{align*} We call the sequence of closed subgroups $(G[i])_{i\in\mathbb{N}}$ the \textbf{Zariski closures} of $G$ in $\G$.
\end{definition}

That is, the $i$-th order Zariski closure $G[i]$ of $G$ with respect to $\G$ is the closed subgroup of $\G[i]$ defined by all $\s$\=/polynomials in $\I(G)$ of order up to and including $i$. 
Notice that as $\G$ is a $\s$\=/closed subgroup of itself defined by $(0)\subseteq k\{\G\}$, for each $i\in\mathbb{N}$, $\G[i]$ is in fact the $i$-th order Zariski closure of $\G$ in itself.

\begin{example}\label[example]{ex: zariski closures example}
        Let $k$ be a $\s$-field, with $\s=\{\s_1,\s_2\}$. Consider the $\s$\=/closed subgroup $G$ of $\mathbb{G}_m$ over $k$ defined by the $\s$\=/Hopf ideal $\I(G)=[\s_1^2\s_2(y)\s_2^4(y)-1]\subseteq k\{y,y^{-1}\}=k\{\mathbb{G}_m\}$, as introduced in \Cref{ex:introducing example,ex:embedding into Gm}. For any $i\in\mathbb{N}$, 
             $k\bigl[\mathbb{G}_m[i]\bigr]=k\bigl[\{\tau(y),\tau(y)^{-1}\ |\  \tau\in T_\s[i]\}\bigr]$,
         and hence 
         \begin{align*}
             \I(G[i])&=(0)\subseteq k[\mathbb{G}_m[i]] &&\text{for $0\leq i\leq 3$}\\
             \I(G[i])&=\bigl(\{\tau(\s_1^2\s_2(y)\s_2^4(y)-1)\ |\ \tau\in T_\s[i-4]\}\bigr)\subseteq k[\mathbb{G}_m[i]] &&\text{for $i\geq 4$}
         \end{align*}
         are the defining ideals for the Zariski closures $(G[i])_{i\in\mathbb{N}}$ of $G$ in $\mathbb{G}_m$.
    \end{example}

    \begin{remark}\label[remark]{rem: zariski closures have restrictions}
        Let $(G[i])_{i\in\mathbb{N}}$ be the Zariski closures of a $\s$\=/closed subgroup $G$ of $\G$. Naturally, for all $i\geq 1$, $\I(G[i-1])\subseteq \I(G[i])$ and $\s_j\bigl(\I(G[i-1])\bigr)\subseteq \I(G[i])$ for all $1\leq j\leq n$, due to the fact that $\I(G)$ is a $\s$\=/ideal. Therefore, by \Cref{lem: restrictions well defined}, we have induced restrictions \begin{align*}
    \pi_i\colon G[i]\rightarrow G[i-1]\ \text{and}\ (\s_j)_i\colon G[i]\rightarrow{^{\s_j}}G[i-1].
\end{align*}  Further, since for each $i\geq 1$, we have that $\I(G[i])\cap k[\G[i-1]]= \I(G[i-1])$, \Cref{lem: restrictions well defined} tells us that the restriction $\pi_i\colon G[i]\rightarrow G[i-1]$ is in fact a quotient map.
    \end{remark}

Some finiteness properties for the Zariski closures of a $\s$\=/algebraic group $G$ of $\G$ have already been proven in the ordinary case.
\begin{theorem}[Wibmer]\label[theorem]{the: Zariski closures ordinary case}
    Let $\G$ be an algebraic group over an ordinary $\s$\=/field $k$. Let $G$ be a $\s$\=/closed subgroup of $\G$, and let $(G[i])_{i\in\mathbb{N}}$ be the Zariski closures of $G$ with respect to $\G$. Then \begin{itemize}
        \item $\I(G[i+1])=(\I(G[i]),\s(\I(G[i])))$ for large enough $i\in\mathbb{N}$,
        \item there exist constants $d,e\in\mathbb{N}$ such that for large enough $i\in\mathbb{N}$, $\dim(G[i])=d(i+1)+e$.
    \end{itemize}
\end{theorem}
\begin{proof}
    See \cite[Section 3, page 528]{wibmer2022finiteness}, and \cite[Section 4, page 533]{wibmer2022finiteness}.
\end{proof}

We wish to extend these properties to the partial case. In order to do so, we will introduce a more general idea for a `splitting' of a $\s$\=/closed subgroup $G$ of $\G$.

\section{Generalised Difference Algebraic Groups}\label{sec: gen dif alg grp def}

Let $\G$ be an algebraic group over a $\s$\=/field $k$. In the previous section we introduced the Zariski closures $(G[i])_{i\in\mathbb{N}}$ of a $\s$\=/closed subgroup $G$ of $\G$ (\Cref{def: zariski closures}), recall that
    $\I(G)=\bigcup_{i\in\mathbb{N}}\I(G[i])$
and $G[i]\leq \G[i]$ for each $i\in\mathbb{N}$. 
We will see a benefit to introducing a slightly more relaxed concept than the Zariski closures of a $\s$\=/closed subgroup $G$ of $\G$. For example, for each $i\geq 1$, we have the map $\pi_i\colon G[i]\rightarrow G[i-1]$ (\Cref{rem: zariski closures have restrictions}), and we can define the kernel $H_i=\ker(\pi_i|_{G[i]})$ of this restriction. It would be nice if the $H_i$'s formed the Zariski closures of some difference closed subgroup of $\G$, but this doesn't hold as the conditions to be the Zariski closures are too restrictive. For this reason, we introduce a more general definition of a sequence of closed subgroups $(G_i)_{i\in\mathbb{N}}$ where the union of the defining ideals $\I(G_i)$ defines a $\s$\=/closed subgroup $G$ of $\G$. These will prove more flexible to work with than the Zariski closures.

\begin{definition}\label[definition]{def: gen diff alg grp}
    Let $\G$ be an algebraic group over $k$. A sequence $(G_i)_{i\in\mathbb{N}}$ is a generalised $\s$\=/algebraic group with respect to $\G$ if the following hold: \begin{itemize}
    \item For every $i\in\mathbb{N}$, $G_i$ is a closed subgroup of $\G[i]$,
        \item for every $i\geq 1$, we have induced restriction $\pi_i\colon G_i\rightarrow G_{i-1}$ of \cref{eq: def of pi map},
        \item for every $i\geq 1$, we have induced restriction $(\s_j)_i\colon G_i\rightarrow{^{\s_j}}G_{i-1}$ of \cref{eq: def of sigma map} for each $1\leq j\leq n$. 
    \end{itemize}
\end{definition}

By \Cref{rem: zariski closures have restrictions}, the Zariski closures $(G[i])_{i\in\mathbb{N}}$ of a $\s$\=/closed subgroup $G$ of an algebraic group $\G$ are a generalised $\s$\=/algebraic group with respect to $\G$. The key difference between a generalised $\s$\=/algebraic group $(G_i)_{i\in\mathbb{N}}$ with respect to $\G$ and the Zariski closures $(G[i])_{i\in\mathbb{N}}$ of a $\s$\=/closed subgroup $G$ of $\G$ is that for every $i\geq 1$, the restriction $\pi_i\colon  G[i]\rightarrow G[i-1]$ is a quotient map, whereas $\pi_i\colon  G_i\rightarrow G_{i-1}$ is not necessarily a quotient map.
Notice that given a generalised $\s$\=/algebraic group $(G_i)_{i\in\mathbb{N}}$ with respect to $\G$, for every $i\geq 2$, we get the commutative diagram in \Cref{fig: comm diagram partial- for g} for each $1\leq j\leq n$.
    \begin{figure}[!htbp]
\centering
\begin{tikzpicture}
\node(a) {$G_i$};
\node(b)at ([xshift=4cm]$(a)$) {$G_{i-1}$};
\node(d) at ([yshift=-1.5cm]$(a)$) {$^{\sigma_j}(G_{i-1})$};
\node(e) at ([xshift=4cm]$(d)$) {$^{\sigma_j}(G_{i-2})$};
\draw[line] (a)-- (b) node[midway,above, color=black] {$\pi_i$};
\draw[line] (a)-- (d) node[midway,left, color=black] {$(\sigma_j)_i$};
\draw[line] (b)-- (e) node[midway,right, color=black] {$(\sigma_j)_{i-1}$};
\draw[line] (d)-- (e) node[midway,below, color=black] {$^{\s_j}({\pi_{i-1}})$};
\end{tikzpicture} 
\caption{Commutative Diagram for a Generalised $\s$\=/Algebraic Group $(G_i)_{i\in\mathbb{N}}$}
\label{fig: comm diagram partial- for g}
\end{figure}

By \Cref{lem: restrictions well defined}, we have an equivalent ideal-theoretic definition of a generalised $\s$\=/algebraic group.

\begin{subdefinition}\label[definition]{def: gen diff alg grp def ideal theoretic}
    Let $\G$ be an algebraic group over $k$. A sequence $(G_i)_{i\in\mathbb{N}}$ is a generalised $\s$\=/algebraic group with respect to $\G$ if the following hold: \begin{itemize}
    \item For every $i\in\mathbb{N}$, $G_i$ is a closed subgroup of $\G[i]$, defined by Hopf ideal $\I(G_i)$ of $k[\G[i]]$,
        \item for every $i\geq 1$, $\I(G_{i-1})\subseteq \I(G_i)$,
        \item for every $i\geq 1$, $\s_j(\I(G_{i-1}))\subseteq \I(G_i)$ for each $1\leq j\leq n$. 
    \end{itemize}
\end{subdefinition}

\begin{remark}\label[remark]{rem: one inclusion for gen grp}
    If $(G_i)_{i\in\mathbb{N}}$ is a generalised $\s$\=/algebraic group with respect to an algebraic group $\G$ over a $\s$\=/field $k$, for every $i\in\mathbb{N}$, $(\I(G_i),\s_1(\I(G_i)),\dots,\s_n(\I(G_i)))\subseteq \I(G_{i+1})$ as ideals in $k[\G[i+1]]$.
\end{remark}

One of the main results of the paper (\Cref{the: partial ideal generation property}) is that for large enough $i\in\mathbb{N}$, the inclusion in \Cref{rem: one inclusion for gen grp} is in fact an equality.

We see that similarly to the Zariski closures of a $\s$\=/closed subgroup of $\G$, these generalised $\s$\=/algebraic groups are some kind of `splitting' of a $\s$\=/closed subgroup of $\G$.

\begin{lemma}\label[lemma]{lem: union of generalised group forms alg subgroup}
    Let $(G_i)_{i\in\mathbb{N}}$ be a generalised $\s$\=/algebraic group with respect to an algebraic group $\G$ over $k$. The union of defining ideals \begin{align*}
    \I(G)=\bigcup_{i\in\mathbb{N}}\I(G_i)
\end{align*} defines a $\s$\=/closed subgroup $G$ of $\G$.
\end{lemma}
\begin{proof}
We will show that $\I(G)$ is a $\s$\=/Hopf ideal of $k\{\G\}$.
Since for each $i\in\mathbb{N}$, $\I(G_i)$ is a Hopf ideal of $k[\G[i]]$, and the ideals $(\I(G_i))_{i\in\mathbb{N}}$ form an ascending chain, their union $\I(G)$ is a Hopf ideal of $\bigcup_{i\in\mathbb{N}}k[\G[i]]=k\{\G\}$. Further, given $f\in \I(G)$, $f\in\I(G_i)$ for some $i\in\mathbb{N}$, and hence for any $1\leq j\leq n$, 
    $\s_j(f)\in\s_j(\I(G_i))\subseteq\I(G_{i+1})\subseteq \I(G)$. That is, $\I(G)$ is a $\s$\=/Hopf ideal of $k\{\G\}$ and therefore, by \Cref{prop: one to one corresp subgroups and Hopf ideals}, defines a $\s$\=/closed subgroup $G$ of $\G$.
\end{proof}

For a $\s$\=/closed subgroup $G$ of an algebraic group $\G$, we can either form the Zariski closures $(G[i])_{i\in\mathbb{N}}$ of $G$ with respect to $\G$, or a generalised $\s$\=/algebraic group $(G_i)_{i\in\mathbb{N}}$ with respect to $\G$ such that $\bigcup_{i\in\mathbb{N}}\I(G_i)=\I(G)$. These both correspond to ascending chains of ideals for each $i\in\mathbb{N}$ whose union is $\I(G)$, but there is a unique choice for the Zariski closures, and many choices for such a generalised $\s$\=/algebraic group.

\begin{example}\label[example]{ex: trivial group zariski generalised}
    Let $k$ be a $\s$\=/field with $\s=\{\s_1,\s_2\}$, and consider the algebraic group $\mathbb{G}_a$ over $k$. We know that $k[\mathbb{G}_a]=k[y]$, and by \Cref{ex: tau A and A[i],example: sk A}, that $k[\mathbb{G}_a[i]]=k\bigl[\{\tau(y)\ |\ \tau\in T_\s[i]\}\bigr]$ for any $i\in\mathbb{N}$, and $k\{ \mathbb{G}_a\}=k\{y\}$.
    Consider the trivial $\s$-closed subgroup $G$ of $\mathbb{G}_a$ such that for every $k$-$\s$-algebra $R$, $G(R)=0$, which is defined as a $\s$\=/closed subgroup of $\mathbb{G}_a$ by the $\s$\=/Hopf ideal $\I(G)=[y]$. 

    The Zariski closures $(G[i])_{i\in\mathbb{N}}$ of $G$ with respect to $\mathbb{G}_a$ must be defined by all possible $\s$\=/polynomials in $\I(G)$ up to order $i$. That is, for each $i\in\mathbb{N}$, 
            $\I(G[i])=\bigl(\{\tau(y)\ |\  \tau\in T_\s[i]\}\bigr)$
        and $G[i]$ is the trivial subgroup of $\mathbb{G}_a[i]$. Notice that $\bigcup_{i\in\mathbb{N}}\I(G[i])=\I(G)$. 
        
    There are many choices for a generalised $\s$-algebraic group $(G_i)_{i\in\mathbb{N}}$ with respect to $\mathbb{G}_a$ such that $\bigcup_{i\in\mathbb{N}}\I(G_i)=\I(G)$. This is because the defining $\s$\=/polynomials do not need to be in the lowest order defining ideal that they are allowed in, and instead can be introduced `late'. For example, the ideals $\I(G_0)=(0)$ and 
            $\I(G_i)=\bigl(\{\s_1\tau(y), \tau(y)\ |\ \tau\in T_\s[i-1]\}\bigr)\subseteq k[\mathbb{G}_a[i]]$
        for each $i\geq 1$ define a generalised $\s$-algebraic group $(G_i)_{i\in\mathbb{N}}$ with respect to $\mathbb{G}_a$. Here, for each $i\geq 1$, the $\s$\=/polynomial $\s_2^i(y)$ is not in $\I(G_i)$, instead it is introduced in the next order defining ideal, $\I(G_{i+1})$. This is in contrast to the Zariski closures, where we must have $\s_2^i(y)\in \I(G[i])$ for each $i\in\mathbb{N}$. That is, for a $k$\=/algebra $R$, there is no restrictions on the ${^{\s_2^i}}\mathbb{G}_a$ entry of an element in $G_i(R)$, whereas the entry must be $0$ in $G[i](R)$.
        See \Cref{fig: zariskis vs generalised} for a visualisation of how the chosen generalised $\s$\=/algebraic group is more relaxed than the Zariski closures.
\begin{figure}[!htbp]
    \centering
    \scalebox{0.8}{
        \begin{tikzpicture}
\draw[step=0.5cm,lightgray,very thin] (-0.25,-0.25) grid (2.45,2.45);
\draw[darkgray,->] (0,0) -- (2.35,0) node[anchor=north west] {$\s_1$};
\draw[darkgray,->] (0,0) -- (0,2.35) node[anchor=south east] {$\s_2$};
\foreach \x in {0,0.5,1,...,2}
    \foreach \y in {0,0.5,1,...,2}  
      \node at (\x,\y) {\textcolor{black}{$0$}};
\node at (1,-0.55) {{$\textcolor{black}{G}$}};
\draw[step=0.5cm,lightgray,very thin] (3.75,-0.25) grid (6.45,2.45);
\draw[darkgray,->] (4,0) -- (6.35,0) node[anchor=north west] {$\s_1$};
\draw[darkgray,->] (4,0) -- (4,2.35) node[anchor=south east] {$\s_2$};
\node(a) at (4,0) {{\textcolor{black}{0}}};
\node(b) at ([xshift=0.5cm]$(a)$) {{\textcolor{black}{0}}};
\node(c) at ([xshift=0.5cm]$(b)$) {{\textcolor{black}{0}}};
\node(d) at ([yshift=0.5cm]$(a)$) {{\textcolor{black}{0}}};
\node(e) at ([xshift=0.5cm]$(d)$) {{\textcolor{black}{0}}};
\node(f) at ([yshift=0.5cm]$(d)$) {{\textcolor{black}{0}}};
\node(g) at ([yshift=0.5cm]$(f)$) {{\textcolor{black}{0}}};
\node(h) at ([xshift=0.5cm]$(f)$) {{\textcolor{black}{0}}};
\node(i) at ([xshift=0.5cm]$(e)$) {{\textcolor{black}{0}}};
\node(i) at ([xshift=0.5cm]$(c)$) {{\textcolor{black}{0}}};
\node at (5,-0.55) {$\textcolor{black}{G[i]}$};
\draw[step=0.5cm,lightgray,very thin] (7.75,-0.25) grid (10.45,2.45);
\draw[darkgray,->] (8,0) -- (10.35,0) node[anchor=north west] {$\s_1$};
\draw[darkgray,->] (8,0) -- (8,2.35) node[anchor=south east] {$\s_2$};
\node(a) at (8,0) {{\textcolor{black}{0}}};
\node(b) at ([xshift=0.5cm]$(a)$) {{\textcolor{black}{0}}};
\node(c) at ([xshift=0.5cm]$(b)$) {{\textcolor{black}{0}}};
\node(d) at ([yshift=0.5cm]$(a)$) {{\textcolor{black}{0}}};
\node(e) at ([xshift=0.5cm]$(d)$) {{\textcolor{black}{0}}};
\node(f) at ([yshift=0.5cm]$(d)$) {{\textcolor{black}{0}}};
\node(g) at ([xshift=0.5cm]$(c)$) {{\textcolor{black}{0}}};
\node(h) at ([xshift=0.5cm]$(f)$) {{\textcolor{black}{0}}};
\node(i) at ([xshift=0.5cm]$(e)$) {{\textcolor{black}{0}}};
\fill (8,1.5) circle(2pt);
\node at (9,-0.55) {$\textcolor{black}{G_i}$};
\draw[step=0.5cm,lightgray,very thin] (11.75,-0.25) grid (14.45,2.45);
\draw[darkgray,->] (12,0) -- (14.35,0) node[anchor=north west] {$\s_1$};
\draw[darkgray,->] (12,0) -- (12,2.35) node[anchor=south east] {$\s_2$};
\node(a) at (12,0) {{\textcolor{black}{0}}};
\node(b) at ([xshift=0.5cm]$(a)$) {{\textcolor{black}{0}}};
\node(c) at ([xshift=0.5cm]$(b)$) {{\textcolor{black}{0}}};
\node(d) at ([yshift=0.5cm]$(a)$) {{\textcolor{black}{0}}};
\node(e) at ([xshift=0.5cm]$(d)$) {{\textcolor{black}{0}}};
\node(f) at ([yshift=0.5cm]$(d)$) {{\textcolor{black}{0}}};
\node(g) at ([yshift=0.5cm]$(f)$) {{\textcolor{black}{0}}};
\node(h) at ([xshift=0.5cm]$(f)$) {{\textcolor{black}{0}}};
\node(i) at ([xshift=0.5cm]$(e)$) {{\textcolor{black}{0}}};
\node(k) at ([xshift=0.5cm]$(g)$) {{\textcolor{black}{0}}};
\node(j) at ([xshift=0.5cm]$(h)$) {{\textcolor{black}{0}}};
\node(l) at ([xshift=0.5cm]$(i)$) {{\textcolor{black}{0}}};
\node(m) at ([xshift=0.5cm]$(c)$) {{\textcolor{black}{0}}};
\node(n) at ([xshift=0.5cm]$(m)$) {{\textcolor{black}{0}}};
\fill  (12,2) circle(2pt);
\node at (13,-0.55) {$\textcolor{black}{G_{i+1}}$};
\end{tikzpicture}}
\caption{\Cref{ex: trivial group zariski generalised}: Comparing Zariski closures to a generalised $\s$\=/algebraic group}
\label{fig: zariskis vs generalised}
\end{figure}
    \end{example}

\begin{example}\label[example]{ex: zariski closures and gen group}
Let $\s=\{\s_1,\s_2\}$.
Consider the $\s$\=/closed subgroup $G$ of $\mathbb{G}_m$ over a $\s$\=/field $k$ defined by the $\s$\=/Hopf ideal $\I(G)=[\s_1^2\s_2(y)\s_2^4(y)-1]\subseteq k\{y,y^{-1}\}=k\{\mathbb{G}_m\}$.
    In \Cref{ex: zariski closures example}, we saw that 
    \begin{align*}
             \I(G[i])&=(0)\subseteq k[\mathbb{G}_m[i]] &&\text{for $0\leq i\leq 3$}\\
             \I(G[i])&=\bigl(\{\tau(\s_1^2\s_2(y)\s_2^4(y)-1)\ |\ \tau\in T_\s[i-4]\}\bigr)\subseteq k[\mathbb{G}_m[i]] &&\text{for $i\geq 4$}
         \end{align*}
         are the defining ideals for the Zariski closures $(G[i])_{i\in\mathbb{N}}$ of $G$ in $\mathbb{G}_m$.
    Again, there are many options for a generalised $\s$\=/algebraic group $(G_i)_{i\in\mathbb{N}}$ with respect to $\mathbb{G}_m$ such that $\bigcup_{i\in\mathbb{N}}\I(G_i)=\I(G)$. For example, we can choose the defining ideals for our generalised $\s$\=/algebraic group to be `one step behind' the defining ideals for the Zariski closures. That is, consider the generalised $\s$\=/algebraic group $(G_i)_{i\in\mathbb{N}}$ with respect to $\mathbb{G}_m$ defined by the Hopf ideals
    \begin{align*}
             \I(G_i)&=(0)\subseteq k[\mathbb{G}_m[i]] &&\text{for $0\leq i\leq 4$}\\
             \I(G_i)&=\bigl(\{\tau(\s_1^2\s_2(y)\s_2^4(y)-1)\ |\ \tau\in T_\s[i-5]\}\bigr)\subseteq k[\mathbb{G}_m[i]] &&\text{for $i\geq 5$}.
          \end{align*} 
         Notice that for each $i\in\mathbb{N}$, $\I(G_{i+1})=\I(G[i])k[\mathbb{G}_m[i+1]]$. 
    \end{example}

\subsection{Building a Natural Structure}\label{sec: building natural structure}

Given a generalised $\s$\=/algebraic group $(G_i)_{i\in\mathbb{N}}$ with respect to an algebraic group $\G$ over $k$, the union
    $\bigcup_{i\in\mathbb{N}}\mathbb{I}(G_i)=\mathbb{I}(G)$
defines a $\s$\=/closed subgroup $G$ of $\G$ (\Cref{lem: union of generalised group forms alg subgroup}). One can construct the Zariski closures of this group $G$ (see \Cref{def: zariski closures}), and see how they relate to the generalised $\s$\=/algebraic group $(G_i)_{i\in\mathbb{N}}$. Recall that the generalised groups seemed to be `behind' the Zariski closures, here we will introduce a way to formalise this.

\begin{definition}\label[definition]{def: zariski closures of gen grp}
    Let $(G_i)_{i\in\mathbb{N}}$ be a generalised $\s$\=/algebraic group with respect to an algebraic group $\G$ over $k$. For each $i\in\mathbb{N}$, let \begin{align*}
    \I(G[i])=\left(\bigcup_{j\in\mathbb{N}}\mathbb{I}(G_j)\right)\cap k[\G[i]],
\end{align*} and call $(G[i])_{i\in\mathbb{N}}$ the \textbf{Zariski closures} of $(G_i)_{i\in\mathbb{N}}$ with respect to $\G$.
\end{definition}

For any generalised $\s$\=/algebraic group $(G_i)_{i\in\mathbb{N}}$ such that $\bigcup_{i\in\mathbb{N}}\mathbb{I}(G_i)=\mathbb{I}(G)$, the $i$-th order Zariski closure $G[i]$ of $(G_i)_{i\in\mathbb{N}}$ is the $i$-th order Zariski closure of $G$. Notice that for each $i\in\mathbb{N}$, $\I(G_i)\subseteq k[\G[i]]$ by \Cref{def: gen diff alg grp def ideal theoretic}, and $\I(G_i)\subseteq \I(G)$, so \begin{align}\label{eq: gen ideal contained in zariski}
    \I(G_i)\subseteq\I(G)\cap k[\G[i]]=\I(G[i])
\end{align} and hence $G[i]\leq G_i$.

Since for each $i\in\mathbb{N}$, $\I(G[i])$ is finitely generated as an ideal of $k[\G[i]]$, and $\I(G[i])\subseteq \bigcup_{j\in\mathbb{N}}\mathbb{I}(G_j)$, we see that there must exist some $j\in\mathbb{N}$ such that $\I(G[i])\subseteq \I(G_j)$.

\begin{definition}\label[definition]{def: zariski indicators}
    Let $(G_i)_{i\in\mathbb{N}}$ be a generalised $\s$\=/algebraic group with respect to an algebraic group $\G$ over $k$, and let $(G[i])_{i\in\mathbb{N}}$ be the Zariski closures of $(G_i)_{i\in\mathbb{N}}$ with respect to $\G$. For each $i\in\mathbb{N}$, let \begin{align*}
    g_i=\min\{j\ |\ \I(G[i])\subseteq \I(G_j)\}\in\mathbb{N}
\end{align*}
and call $(g_i)_{i\in\mathbb{N}}$ the \textbf{Zariski indicators} of $(G_i)_{i\in\mathbb{N}}$ with respect to $\G$.
\end{definition}

Notice that for $i\in\mathbb{N}$, $\I(G[i])\subseteq \I(G_{g_i})$, so $\I(G[i])\subseteq \mathbb{I}(G_{g_i})\cap k[\G[i]]$. Conversely, by \Cref{def: zariski closures of gen grp}, $\mathbb{I}(G_{g_i})\cap k[\G[i]]\subseteq \cup_{j\in\mathbb{N}}\I(G_j)\cap k[\G[i]]=\I(G[i])$. That is,  
   $ \I(G[i])=\mathbb{I}(G_{g_i})\cap k[\G[i]]$,
and $g_i$ is the smallest natural number for which this is true.

\begin{lemma}\label[lemma]{lem: gi=0 or i}
Suppose that $(g_i)_{i\in\mathbb{N}}$ are the Zariski indicators of a generalised $\s$\=/algebraic group $(G_i)_{i\in\mathbb{N}}$ with respect to $\G$. 
For every $i\in\mathbb{N}$, $g_i=0$ or $g_i\geq i$.
\end{lemma}
\begin{proof} Let $(G[i])_{i\in\mathbb{N}}$ be the Zariski closures of $(G_i)_{i\in\mathbb{N}}$ with respect to $\G$. Suppose that $g_i\leq i-1$. Then, \begin{align*}
        \I(G[i])\subseteq\I(G_{g_i})\subseteq\I(G_{i-1})\subseteq \I(G[i-1]).
    \end{align*}
    where the second and third inclusions are due to \Cref{def: gen diff alg grp def ideal theoretic,eq: gen ideal contained in zariski} respectively.
    Then for any $f\in \I(G[i])\subseteq \I(G[i-1])$, $\s_j(f)\in\s_j(\I(G[i-1]))\subseteq \I(G[i])$ for any $1\leq j\leq n$ by \Cref{rem: zariski closures have restrictions}. If $\I(G[i])$ contains some non-constant polynomial, then this would suggest $\I(G[i])$ contains elements of infinite order, contradicting that $\I(G[i])\subseteq k[\G[i]]$. Therefore $\I(G[i])=(0)$, that is, $\I(G[i])\subseteq\I(G_0)$ and $g_i=0$.
\end{proof}

Notice that by the proof of \Cref{lem: gi=0 or i}, if $g_m> 0$ for some $m\in\mathbb{N}$, then for all $i\geq m$, $g_i>0$ and hence $g_i\geq i$.
Further, given a $\s$\=/closed subgroup $G$ of some $\G$, if $(g_i)_{i\in\mathbb{N}}$ are the Zariski indicators of the Zariski closures $(G[i])_{i\in\mathbb{N}}$ of $G$ with respect to $\G$, if for some $i\in\mathbb{N}$, $\I(G[i])$ is non-zero, then $g_i=i$.

\begin{example}\label[example]{ex: zariski indicators}
Consider the generalised $\s$\=/algebraic group $(G_i)_{i\in\mathbb{N}}$ and its corresponding Zariski closures $(G[i])_{i\in\mathbb{N}}$ from \Cref{ex: zariski closures and gen group}. For $0\leq i\leq 3$, $\I(G[i])=(0)$, so $g_i=0$, and for all $i\geq 4$, $\I(G[i])\not\subseteq\I(G_i)$, but $\I(G[i])\subseteq \I(G_{i+1})$. That is, for all $i\geq 4$, $g_i=i+1$. 
    \end{example}

Intuitively, the Zariski indicators $(g_i)_{i\in\mathbb{N}}$ provide a way to measure of how different our generalised $\s$\=/algebraic group is to its corresponding Zariski closures.
    Notice that in \Cref{ex: zariski closures and gen group}, the generalised $\s$\=/algebraic group given was intuitively `one step behind' the Zariski closures. Then in \Cref{ex: zariski indicators}, we saw that for all large enough $i\in\mathbb{N}$, $g_i=i+1$ for this particular generalised $\s$\=/algebraic group. We already know that Zariski closures behave `nicely' in the ordinary case (\Cref{the: Zariski closures ordinary case}), so perhaps we can say more about a generalised $\s$\=/algebraic group if it behaves similarly to its Zariski closures. 
    We now prove that if at some critical point our generalised $\s$\=/algebraic group is the same as the Zariski closure, then they are equal for all large enough $i\in\mathbb{N}$. 

    \begin{lemma}\label[lemma]{lem: jm=m means I(G[i])=I(Gi)}
    Let $(G_i)_{i\in\mathbb{N}}$ be a generalised $\s$\=/algebraic group with respect to an algebraic group $\G$ over $k$. Let $(G[i])_{i\in\mathbb{N}}$ be the Zariski closures of $(G_i)_{i\in\mathbb{N}}$ in $\G$ and let $(g_i)_{i\in\mathbb{N}}$ be the Zariski indicators of $(G_i)_{i\in\mathbb{N}}$ with respect to $\G$. Suppose that there exists an $m\in\mathbb{N}$ such that  \begin{align*}
        \I(G[i+1])=(\I(G[i]),\s_1(\I(G[i])),\dots,\s_n(\I(G[i])))
    \end{align*}
    for all $i\geq m$. If $g_m-m=0$, then $\I(G[i])=\I(G_i)$ for all $i\geq m$.
\end{lemma}
\begin{proof}
Firstly notice that if $g_i=i$, then $\I(G[i])=\I(G_i)$. So this proof reduces to showing that if $g_m=m$, then for all $i\geq m$, $g_i=i$. This can be proven by induction on the value of $i\geq m$. Clearly the base case holds, so suppose that for some $i\geq m$, $g_i=i$. Then\begin{align*}
        \I(G[i+1])&=(\I(G[i]),\s_1(\I(G[i])),\dots,\s_n(\I(G[i]))) &&\text{as $i\geq m$}\\
        &=(\I(G_i),\s_1(\I(G_i)),\dots,\s_n(\I(G_i))) &&\text{by inductive hypothesis}\\
        &\subseteq \I(G_{i+1})&&\text{by \Cref{rem: one inclusion for gen grp}}
    \end{align*}
    so $g_{i+1}\leq i+1$, hence $g_{i+1}=i+1$ as required.
\end{proof}

\subsection{The Projections of a Generalised Difference Algebraic Group}\label{sec: forming projections}

Recall that the key difference between the Zariski closures $(G[i])_{i\in\mathbb{N}}$ of a $\s$\=/closed subgroup $G$ of $\G$ and a generalised $\s$\=/algebraic group $(G_i)_{i\in\mathbb{N}}$ with respect to $\G$ is that for $i\in\mathbb{N}$, $\pi_{i+1}\bigl( G[i+1]\bigr)= G[i]$, whereas $\pi_{i+1}\bigl(G_{i+1}\bigr)\leq G_i$. We will now consider the image of $G_{i+1}$ under the map $\pi_{i+1}$, and see that if we have a generalised $\s$\=/algebraic group that is not equal to its Zariski closures, these images under projection will form a generalised $\s$\=/algebraic group that is `closer' to the Zariski closures than the original generalised $\s$\=/algebraic group.

\begin{definition}\label[definition]{def: projections}
    Let $(G_i)_{i\in\mathbb{N}}$ be a generalised $\s$\=/algebraic group with respect to some algebraic group $\G$ over $k$. For each $i\in\mathbb{N}$, let $F_i=\pi_{i+1}\bigl(G_{i+1}\bigr)$ where $\pi_{i+1}$ is the projection map defined in \cref{eq: def of pi map}. We refer to the sequence $(F_i)_{i\in\mathbb{N}}$ as the \textbf{projections of $(G_i)_{i\in\mathbb{N}}$ under $\pi$}. 
\end{definition}

That is, for all $i\in\mathbb{N}$, $\pi_{i+1}\colon G_{i+1}\rightarrow F_i$ is a quotient map.

\begin{lemma}\label[lemma]{lem: projections also generalised grps}
Let $(G_i)_{i\in\mathbb{N}}$ be a generalised $\s$\=/algebraic group with respect to an algebraic group $\G$ over $k$. The projections $(F_i)_{i\in\mathbb{N}}$ of $(G_i)_{i\in\mathbb{N}}$ under $\pi$ form a generalised $\s$\=/algebraic group with respect to $\G$.
\end{lemma}
    \begin{proof}
    Notice that for any $i\in\mathbb{N}$, $F_i$ is the closed subgroup of $\G[i]$ defined by \begin{align}\label{eq: def ideal for F}
        \mathbb{I}(F_i)=\mathbb{I}(G_{i+1})\cap k[\G[i]].
    \end{align}
    The fact that the sequence $(F_i)_{i\in\mathbb{N}}$ satisfies \Cref{def: gen diff alg grp def ideal theoretic} then follows from the fact that the sequence $(G_i)_{i\in\mathbb{N}}$ satisfies \Cref{def: gen diff alg grp def ideal theoretic}.
    \end{proof}

\begin{remark}\label[remark]{rem: zariski closures grp and proj same}
Given a generalised $\s$\=/algebraic group $(G_i)_{i\in\mathbb{N}}$ with respect to $\G$ and the projections $(F_i)_{i\in\mathbb{N}}$ of $(G_i)_{i\in\mathbb{N}}$ under $\pi$, notice that
   $\bigcup_{i\in\mathbb{N}}\I(G_i)=\bigcup_{i\in\mathbb{N}}\I(F_i)$. Therefore, by \Cref{def: zariski closures of gen grp}, the Zariski closures of $(G_i)_{i\in\mathbb{N}}$ with respect to $\G$ are the same as the Zariski closures of $(F_i)_{i\in\mathbb{N}}$ with respect to $\G$.
\end{remark}

\begin{lemma}\label[lemma]{lem: zariskis equal ji>i means tji leq ji-1}
 Let $(G_i)_{i\in\mathbb{N}}$ be a generalised $\s$\=/algebraic group with respect to an algebraic group $\G$ over $k$ and let $(F_i)_{i\in\mathbb{N}}$ be the projections of $(G_i)_{i\in\mathbb{N}}$ under $\pi$. Let $(g_i)_{i\in\mathbb{N}}$ and $(f_i)_{i\in\mathbb{N}}$ be the Zariski indicators of $(G_i)_{i\in\mathbb{N}}$ and $(F_i)_{i\in\mathbb{N}}$ respectively (\Cref{def: zariski indicators}). If for some $i\in\mathbb{N}$, $g_i>i$, then $f_i\leq g_i-1$.
\end{lemma}
\begin{proof}
Let $(G[i])_{i\in\mathbb{N}}$ be the Zariski closures of $(G_i)_{i\in\mathbb{N}}$ (and hence of $(F_i)_{i\in\mathbb{N}}$ by \Cref{rem: zariski closures grp and proj same}) with respect to $\G$.
Recall that for each $i\in\mathbb{N}$, $g_i$ and $f_i$ are the smallest natural numbers such that $\I(G[i])\subseteq \I(G_{g_i})$ and $\I(G[i])\subseteq \I(F_{f_i})$ respectively.
 Suppose that $g_i>i$ for some $i\in\mathbb{N}$. Then \begin{align*}
        \I(G[i])&=\I(G_{g_i})\cap k[\G[i]]\\
        &=\I(G_{g_i})\cap k[\G[g_i-1]]\cap k[\G[i]] &&\text{as $g_i>i$}\\
        &=\I(F_{g_i-1})\cap k[\G[i]] &&\text{by (\ref{eq: def ideal for F})}
    \end{align*}
    so $\I(G[i])\subseteq \I(F_{g_i-1})$, that is, $f_i\leq g_i-1$ as required.
\end{proof}

This lemma gives a hint why the projections of a generalised difference algebraic group are a useful construction, they are closer to the shared Zariski closures than the original group.

\begin{example}\label[example]{ex: projections of my group}
    Consider the generalised $\s$\=/algebraic group $(G_i)_{i\in\mathbb{N}}$ and its corresponding Zariski closures $(G[i])_{i\in\mathbb{N}}$ from \Cref{ex: zariski closures and gen group}. In \Cref{ex: zariski indicators}, we found that the Zariski indicators $(g_i)_{i\in\mathbb{N}}$ of $(G_i)_{i\in\mathbb{N}}$ with respect to $\mathbb{G}_m$ are $g_i=0$ for $0\leq i\leq3$, and $g_i=i+1$ for $i\geq 4$. The projections $(F_i)_{i\in\mathbb{N}}$ of $(G_i)_{i\in\mathbb{N}}$ under $\pi$ are defined by \begin{align*}
        \I(F_i)=\I(G_{i+1})\cap k[\G[i]]=\I(G[i])k[\G[i+1]]\cap k[\G[i]]=\I(G[i])
    \end{align*} for each $i\in\mathbb{N}$. So we do see that in fact the projections $(F_i)_{i\in\mathbb{N}}$ of $(G_i)_{i\in\mathbb{N}}$ are `closer' to $(G[i])_{i\in\mathbb{N}}$ than the generalised $\s$\=/algebraic group $(G_i)_{i\in\mathbb{N}}$. Formally, the Zariski indicators $(f_i)_{i\in\mathbb{N}}$ of $(F_i)_{i\in\mathbb{N}}$ with respect to $\mathbb{G}_m$ are $f_i=i$ for all $i\geq 3$.
\end{example}

\subsection{The Extensions of a Generalised Difference Algebraic Group}\label{sec: forming extensions}

We now introduce a construction which will allow us to be able to consider, for example, $G_i$ or ${^{\s_j}}G_i$ as a subgroup of $\G[i+1]$ for $i\in\mathbb{N}$ and $1\leq j\leq n$. This construction will be useful in the proof of the ideal generation property for generalised difference algebraic groups.

Firstly, for a subset $A\subseteq T_\s$, we put $\G_A=\prod_{\tau\in A}{^\tau}\G$. Therefore, for any $A\subseteq B\subseteq T_\s$, there is a quotient map $\G_B\rightarrow\G_A$, and $k[\G_A]\subseteq k[\G_B]$. Notice that
    $\G_{T_\s[i]}=\G[i]$ and $\G_{\s_j(T_\s[i])}={^{\s_j}}\G[i]$
for any $i\in\mathbb{N}$, $1\leq j\leq n$. This notation allows us to introduce a new definition.

\begin{definition}\label[definition]{def: extension}
Let $\G$ be an algebraic group over $k$.
Suppose that we have subsets $A\subseteq B\subseteq T_\s$.
     Given a closed subgroup $\H\leq \G_A$ defined by the ideal $\I(\H)\subseteq k[\G_A]$, let 
        $\widehat{\H}^{\G_B}$
    be the closed subgroup of $\G_B$ defined by $\I(\H)k[\G_B]$, the ideal of $k[\G_B]$ generated by $\I(\H)$. We call $\widehat{\H}^{\G_B}$ the \textbf{extension of $G$ to $\G_B$}.
\end{definition}

Essentially, given a closed subgroup $\H\leq \G_A$ we find its extension to $\G_B$ by `filling in' any gaps with $^\tau \G$ for the appropriate $\tau\in B\backslash A$.
Some of the most crucial constructions of these types will be the extension of $G_i$ to $\G[i+1]$, and the extension of ${^{\s_j}}G_i$ to $\G[i+1]$ for $i\in\mathbb{N}$ and $1\leq j\leq n$. See \Cref{fig comp: Gi s1 Gi and extension s1 Gi}  for visualisations of these constructions in the case where we have two endomorphisms.

\begin{figure}[!htbp]
\centering
\begin{tikzpicture}
\draw[step=0.5cm,lightgray,very thin] (-0.45,-0.45) grid (2.45,2.45);
\draw[darkgray,->] (0,0) -- (2.25,0) node[anchor=north west] {$\s_1$};
\draw[darkgray,->] (0,0) -- (0,2.25) node[anchor=south east] {$\s_2$};
\fill[red] (0,0) circle(2pt);
\fill[red] (0.5,0) circle(2pt) (0,0.5) circle(2pt);
\fill[red] [radius=2pt] (1,0) circle[] (0.5,0.5) circle[] (0,1) circle[];
\node at (1.25,-1) {{$G_i$}};
\draw[step=0.5cm,lightgray,very thin] (3.55,-0.45) grid (6.45,2.45);
\draw[darkgray,->] (4,0) -- (6.25,0) node[anchor=north west] {$\s_1$};
\draw[darkgray,->] (4,0) -- (4,2.25) node[anchor=south east] {$\s_2$};
\fill[red] (4,0) circle(2pt);
\fill[red] (4.5,0) circle(2pt) (4,0.5) circle(2pt);
\fill[red] [radius=2pt] (5,0) circle[] (4.5,0.5) circle[] (4,1) circle[];
\fill [radius=2pt] (5.5,0) circle[] (5,0.5) circle[] (4.5,1) circle[] (4,1.5) circle[];
\node at (5.5,1.35) {\footnotesize copies of ${^\tau}\G$ where };
\node at (5.85,1.1) {\footnotesize  $\ord(\tau)=i+1$};
\node at (3.6,0.25) {\footnotesize \textcolor{red}{$G_i$}};
\node at (5.4,-0.9) {{$\widehat{G_i}^{\G[i+1]}$}};
\draw[step=0.5cm,lightgray,very thin] (7.55,-0.45) grid (10.45,2.45);
\draw[darkgray,->] (8,0) -- (10.25,0) node[anchor=north west] {$\s_1$};
\draw[darkgray,->] (8,0) -- (8,2.25) node[anchor=south east] {$\s_2$};
\fill[red] (8.5,0) circle(2pt);
\fill[red] [radius=2pt] (9,0) circle[] (8.5,0.5) circle[];
\fill[red] [radius=2pt] (9.5,0) circle[] (9,0.5) circle[] (8.5,1) circle[];
\node at (9.25,-1) {{${^{\s_1}}G_i$}};
\draw[step=0.5cm,lightgray,very thin] (11.55,-0.45) grid (14.45,2.45);
\draw[darkgray,->] (12,0) -- (14.25,0) node[anchor=north west] {$\s_1$};
\draw[darkgray,->] (12,0) -- (12,2.25) node[anchor=south east] {$\s_2$};
\fill[red] (12.5,0) circle(2pt);
\fill[red] [radius=2pt] (13,0) circle[] (12.5,0.5) circle[];
\fill[red] [radius=2pt] (13.5,0) circle[] (13,0.5) circle[] (12.5,1) circle[];
\fill (12,0) circle(2pt) (12,0.5) circle(2pt) (12,1) circle(2pt) (12,1.5) circle(2pt); 
\node at (13.25,0.75) {\footnotesize \textcolor{red}{${^{\s_1}}G_i$}};
\node at (12.7,1.9) {\footnotesize copies of ${^\tau}\G$};
\node at (12.95,1.65) {\footnotesize where $\s_1\nmid \tau$};
\node at (13.4,-0.9) {{$\widehat{{^{\s_1}}G_i}^{\G[i+1]}$}};
\end{tikzpicture}
\caption{The Extensions of $G_i$ and ${^{\s_1}}G_i$ to $\G[i+1]$ when $n=2$}
\label{fig comp: Gi s1 Gi and extension s1 Gi}
\end{figure}

These crucial constructions can be written explicitly in the ordinary case,
let $(G_i)_{i\in\mathbb{N}}$ be an ordinary generalised $\s$\=/algebraic group with respect to $\G$. Then we have that \begin{align*}
    \widehat{G_i}^{\G[i+1]}=G_i\times {^{\s^{i+1}}}\G\quad\text{and}\quad \widehat{{^\s}G_i}^{\G[i+1]}=\G\times{^\s}G_i
\end{align*}
for every $i\in\mathbb{N}$. The reason these constructions are difficult to write explicitly as products in the partial case is due to the order of the products. 

For this reason, in order to work with the extensions of a generalised difference algebraic group, we consider some natural isomorphisms which mix up the order of the products. Consider the inverse morphisms of algebraic groups \begin{align*}
    \phi_{(A,B)}\colon \G_B\rightarrow \G_A\times \G_{B\backslash A} \quad\text{and}\quad \phi_{(A,B)}^{-1}\colon \G_A\times \G_{B\backslash A}\rightarrow \G_B
\end{align*}
for some sets $A\subseteq B\subseteq T_\s$. Then for a closed subgroup $\H\leq \G_A$, we have \begin{align}\label{equation: mixing up order of subgroup}
    \phi_{(A,B)}\left(\widehat{\H}^{\G_B}\right)=\H\times \G_{B\backslash A}\quad\text{and}\quad \phi_{(A,B)}^{-1}(\H\times \G_{B\backslash A})=\widehat{\H}^{\G_B}.
\end{align}

We will use these maps changing the order of products to see how extensions commute with projection maps and base change.

\begin{lemma}\label[lemma]{lem: projections commute with extension}
    Let $\G$ be an algebraic group over $k$, and suppose that we have subsets $A, B\subseteq C\subseteq T_\s$. (Assume $\emptyset\subsetneq A\cap B$).
    Consider the projection map $\mu\colon \G_C\rightarrow \G_B$. 
    For a closed subgroup $\H\leq\G_A$, \begin{align*}
        \mu\left(\widehat{\H}^{\G_C}\right)=\widehat{\mu_1(\H)}^{\G_{B}}
    \end{align*}
    where $\mu_1\colon \G_A\rightarrow\G_{A\cap B}$ is the projection from $\G_A$ to $\G_{A\cap B}$. 
\end{lemma}
\begin{proof}
We can write the map $\mu$ as $\mu=\phi_{(A\cap B, B)}^{-1}\circ (\mu_1,\mu_2)\circ \phi_{(A,C)}$, where 
\begin{align*}
    \mu_1\colon \G_A\rightarrow \G_{A\cap B},\quad \mu_2\colon \G_{C\backslash A}\rightarrow \G_{B\backslash (A\cap B)}.
\end{align*}
Further, by properties of projection maps, $(\mu_1,\mu_2)(\H\times \G_{C\backslash A})=\mu_1(\H)\times \G_{B\backslash (A\cap B)}$
for any closed subgroup $\H\leq \G_A$. The result now follows from \cref{equation: mixing up order of subgroup}.
\end{proof}

\begin{corollary}\label[corollary]{cor: kernel of extensions}
    With the notation of \Cref{lem: projections commute with extension}, \begin{align*}
\ker\left(\mu\big|_{\widehat{\H}^{\G_C}}\right)\leq\widehat{\ker(\mu_1|_{\H})}^{\G_C}.
    \end{align*}
\end{corollary}

Let's see some key examples of the applications of \Cref{lem: projections commute with extension}.

\begin{corollary}\label[corollary]{lem: extension commutes with}
    Let $(G_i)_{i\in\mathbb{N}}$ be a generalised $\s$\=/algebraic group with respect to an algebraic group $\G$ over $k$. Then  \begin{align*}
            (a) \ \pi_{i+1}\left(\widehat{G_i}^{\G[i+1]}\right)=G_i\quad\text{and}\quad (b)\ (\s_j)_{i+1}\left( \widehat{^{\s_j}G_i}^{\G[i+1]}\right)={^{\s_j}}G_i
        \end{align*}
    for each $i\in\mathbb{N}$, $1\leq j\leq n$ and further \begin{align*}
        (c)\ \pi_{i+1}\left(\widehat{{^{\s_j}}G_i}^{\G[i+1]}\right)=\widehat{{^{\s_j}}{\left(\pi_i\bigl(G_i\bigr)\right)}}^{\G[i]} \quad \text{and} \quad (d)\ (\s_j)_{i+1}\left(\widehat{G_i}^{\G[i+1]}\right)=\widehat{(\s_j)_i\bigl(G_i\bigr)}^{^{\s_j}\G[i]}
    \end{align*}
    for each $i\geq 1$, $1\leq j\leq n$, where $\pi_{i+1}$ and $(\s_j)_{i+1}$ are the projections maps \cref{eq: def of pi map,eq: def of sigma map}.
\end{corollary}
\begin{proof}
     These identities follow from \Cref{lem: projections commute with extension}, by making the appropriate choices for the subsets $A$, $B$ and $C$. Looking at \Cref{fig comp: Gi s1 Gi and extension s1 Gi}, we can visualise why the identities $(a)$ and $(b)$ hold in the $n=2$ case.
\end{proof}

This tells us that our maps $\pi_i$ and $(\s_j)_i$ for $i\geq 1$ and $1\leq j\leq n$ commute with taking extensions. We also see that base change commutes with extension.

\begin{lemma}\label[lemma]{lem: base change commutes with extension}

Let $\G$ be an algebraic group over $k$, and suppose that $A\subseteq B\subseteq T_\s$. 
    For any any $\tau\in T_\s$,  \begin{align*}
        {^\tau}\left(\widehat{\H}^{\G_B}\right)&=\widehat{{^\tau}\H}^{^\tau{\G_B}}
    \end{align*}
    for any closed subgroup $\H\leq\G_A$. That is, taking the extension commutes with base change.
\end{lemma}
\begin{proof}
Let $\tau\in T_\s$. 
Notice that for a subset $C\subseteq T_\s$, ${^{\tau}}\G_C=\G_{\tau(C)}$, and hence ${^\tau}\phi_{(A,B)}=\phi_{(\tau(A),\tau(B))}$ and ${^\tau}\phi_{(A,B)}^{-1}=\phi_{(\tau(A),\tau(B))}^{-1}$. 
    Therefore
    \begin{align*}
        {^{\tau}}\left(\widehat{\H}^{\G_B}\right)={^{\tau}}\left(\phi_{(A,B)}^{-1}\left(\H\times \G_{B\backslash A}\right)\right)
        =\left({^{\tau}}\phi_{(A,B)}^{-1}\right)\left({^{\tau}}\H\times {^{\tau}}\G_{B\backslash A}\right)
        =\widehat{{^{\tau}}\H}^{{^{\tau}}\G_{B}}
    \end{align*}
    as required.
\end{proof}

\section{Properties of Kernels}\label{sec: kernels}

For a generalised $\s$\=/algebraic group $(G_i)_{i\in\mathbb{N}}$ with $n$ endomorphisms, the kernels of the maps defined in \cref{eq: def of pi map} will form a generalised $\s'$ algebraic group with $n-1$ endomorphisms, and this fact will allow us to apply induction for the proofs of the desired finiteness properties. For any algebraic group $\G$ over a field, use $1_{\G}$ to denote the trivial closed subgroup of $\G$.

\begin{definition}\label[definition]{def: kernels}
    Let $(G_i)_{i\in\mathbb{N}}$ be a generalised $\s$\=/algebraic group with respect to an algebraic group $\G$ over $k$. Let $H_0=G_0$ and $H_i=\ker(\pi_i|_{G_i})$ (see \cref{eq: def of pi map}) for each $i\geq 1$. We call $(H_i)_{i\in\mathbb{N}}$ the \textbf{kernels of $\pi$ restricted to $(G_i)_{i\in\mathbb{N}}$}. 
\end{definition}

\begin{lemma}\label[lemma]{cor: restriction sj to H[i] well defined}
Let $(G_i)_{i\in\mathbb{N}}$ be a generalised $\s$\=/algebraic group with respect to an algebraic group $\G$ over $k$ and let $(H_i)_{i\in\mathbb{N}}$ be the kernels of $\pi$ restricted to $(G_i)_{i\in\mathbb{N}}$. 
    Then there is induced restriction \begin{align*}
            (\s_j)_i\colon H_i\rightarrow{^{\s_j}}H_{i-1}
    \end{align*}
    of the projection map $(\s_j)_i\colon \G[i]\rightarrow {^{\s_j}}\G[i-1]$ (\cref{eq: def of sigma map}) for every $i\geq$, $1\leq j\leq n$.
\end{lemma}
\begin{proof}
This is the partial analogue to a result in \cite[Section 3, page 526]{wibmer2022finiteness},  and follows from the fact that the diagram in \Cref{fig: comm diagram partial- for g} commutes. 
\end{proof}

\begin{corollary}\label[corollary]{cor: kernels generalised grp}
    Let $(G_i)_{i\in\mathbb{N}}$ be a generalised $\s$\=/algebraic group with respect to an algebraic group $\G$ over $k$. The kernels $(H_i)_{i\in\mathbb{N}}$ of $\pi$ restricted to $(G_i)_{i\in\mathbb{N}}$ form a generalised $\s$\=/algebraic group with respect to $\G$. 
\end{corollary}
\begin{proof}
    It is clear that for each $i\in\mathbb{N}$, $H_i$ is a closed subgroup of $G_i$ and hence of $\G[i]$. Further, for each $i\geq 1$, $\pi_i\bigl(H_i\bigr)=1_{\G[i-1]}\leq H_{i-1}$, and for each $1\leq j\leq n$, $(\s_j)_i\bigl(H_i)\leq {^{\s_j}}H_{i-1}$ by \Cref{cor: restriction sj to H[i] well defined}, so $(H_i)_{i\in\mathbb{N}}$ is a generalised $\s$\=/algebraic group (\Cref{def: gen diff alg grp}). 
\end{proof}

\subsection{Kernels of Ordinary Generalised Difference Algebraic Groups}

To avoid confusion, when we have an ordinary $\s$\=/field $k$, write $(\s)_i\colon\G[i]\rightarrow \G[i-1]$ for the map \cref{eq: def of sigma map}.

\begin{lemma}\label[lemma]{lem: H[i+1] isomorphic to H[i] (ordinary)}
Let $\G$ be an algebraic group over an ordinary $\s$\=/field $k$.
Let $(G_i)_{i\in\mathbb{N}}$ be a generalised $\s$\=/algebraic group with respect to $\G$, and let $(H_i)_{i\in\mathbb{N}}$ be the kernels of $\pi$ restricted to $(G_i)_{i\in\mathbb{N}}$. 
    For each $i\geq 1$, $(\s)_i\colon H_i\rightarrow {^\s}H_{i-1}$ is a closed embedding, and for large $i\in\mathbb{N}$, $(\s)_i\colon H_i\rightarrow {^\s}H_{i-1}$ is an isomorphism.
\end{lemma}
\begin{proof}
This is a generalisation of result in \cite[Section 3, page 526]{wibmer2022finiteness}, and is proven using the same methods.
For each $i\in\mathbb{N}$, for any $k$\=/algebra $R$, 
    $H_i(R)=\bigl\{(r_0,r_1,\dots,r_i)\in G_i(R)\ |\ r_0=\cdots=r_{i-1}=1\bigr\}$, where we are using $1$ to denote the unit of $\G(R)$. The restriction $((\s)_i)_R\colon H_i(R)\rightarrow {^\s}H_{i-1}(R)$,  \begin{align*}
    (\underbrace{1,\cdots,1}_{i \text{ many}},r_i)\mapsto (\underbrace{1,\cdots,1}_{i-1 \text{ many}},r_i)
\end{align*}
is clearly injective, hence $(\s)_i\colon H_i\rightarrow {^\s}H_{i-1}$ is a closed embedding. Firstly suppose that $\s\colon k\rightarrow k$ is a bijection. Then for each $i\in\mathbb{N}$, we can consider the morphism \begin{align}\label{eq: base change projection}
    {^{\s^{-i}}}(\s)_i\colon {^{\s^{-i}}}H_i\rightarrow {^{\s^{-i}}}{^\s}H_{i-1}
\end{align} obtained by base change via $\s^{-i}\colon k\rightarrow k$ (\Cref{def: group base change}), which is still a closed embedding. By \Cref{rem: base change comp}, ${^{\s^{-i}}}({^\s}H_{i-1})={^{\s^{-(i-1)}}}H_{i-1}$, and hence we have a descending chain \begin{align*}
        H_0\hookleftarrow {^{\s^{-1}}}H_1 \hookleftarrow {^{\s^{-2}}}H_2 \hookleftarrow \cdots
    \end{align*}
    of closed subgroups, which must stabilise.
    That is, for large enough $i\in\mathbb{N}$, 
        the morphism (\ref{eq: base change projection}) is an isomorphism and hence $(\s)_i\colon H_i\rightarrow {^\s}H_{i-1}$
    is an isomorphism. 
    If $\s\colon k\rightarrow k$ is not a bijection, we can pass via base change to a field where it is. For example, we could consider the inversive closure of $k$, an inversive $\s$\=/overfield of $k$, whose existence is proven in  \cite[Proposition 2.1.7, page 109]{levin2008difference}. Since for large $i\in\mathbb{N}$, $(\s)_i\colon H_i\rightarrow {^\s}H_{i-1}$ will be an isomorphism after base change, it must already be an isomorphism before base change. 
\end{proof}

\subsection{Kernels of Partial Generalised Difference Algebraic Groups}\label{sec: kernels being n-1 groups}

Let $\s=\basicset$, and assume in this section that $n\geq 2$. Let $\G$ be an algebraic group over $k$. Let $(G_i)_{i\in\mathbb{N}}$ be a generalised $\s$\=/algebraic group with respect to $\G$, and let $(H_i)_{i\in\mathbb{N}}$ be the kernels of $\pi$ restricted to $(G_i)_{i\in\mathbb{N}}$. See \Cref{fig ker: G[i] H[i] rho[i]H[i]} for a visualisation of $G_i$ and $H_i$ in the $n=2$ case. For the kernels $H_i$, the only non-trivial information is contained in the ${^\tau}\G$ entries where $\ord(\tau)=i$. Therefore, we introduce a new construction, \begin{align*}
    \G(i)=\prod_{\tau\in T_\s(i)}{^\tau}\G
\end{align*}
for each $i\in\mathbb{N}$. That is, where $\G[i]$ is the product of ${^\tau}\G$ for $\tau\in T_\s[i]$, $\G(i)$ is the product of ${^\tau}\G$ for $\tau\in T_\s(i)$ (see \cref{eq: Ts[i] and Ts(i)} for the definitions of $T_\s[i]$ and $T_\s(i)$, and \cref{eq: G[i]} for the definition of $\G[i]$).

For each $i\in\mathbb{N}$, consider the morphisms of algebraic groups
\begin{align}\label{eq: def of rho}
    \rho_i\colon \G[i]&\rightarrow\G(i),\quad  (x_\tau)_{\tau\in T_\s[i]}\mapsto (x_\tau)_{\tau\in T_\s(i)} &&\text{and}\\ \label{eq: def of rho star}
    \rho^*_i\colon \G(i) &\rightarrow \G[i], \quad \hspace{0.1mm}
    (x_\tau)_{\tau\in T_\s(i)}\mapsto \bigl(1_{\G[i-1]},(x_\tau)_{\tau\in T_\s(i)}\bigr).
\end{align}
That is, $\rho_i$ drops the entries of $^\tau\G$ where $\ord(\tau)<i$, and $\rho^*_i$ 
fills entries in ${^\tau}\G$ where $\ord(\tau)<i$ with $1_{({^\tau}\G)}$.

\begin{figure}[!htbp]
\centering
\begin{tikzpicture}
\draw[step=0.5cm,lightgray,very thin] (-0.25,-0.25) grid (1.95,1.95);
\draw[darkgray,->] (0,0) -- (1.85,0) node[anchor=north west] {$\s_1$};
\draw[darkgray,->] (0,0) -- (0,1.85) node[anchor=south east] {$\s_2$};
\fill[red] (0,0) circle(2pt);
\fill[red] (0.5,0) circle(2pt) (0,0.5) circle(2pt);
\fill[red] [radius=2pt] (1,0) circle[] (0.5,0.5) circle[] (0,1) circle[];
\fill[red] [radius=2pt] (1.5,0) circle[] (1,0.5) circle[] (0.5,1) circle[] (0,1.5) circle[];
\node at (0.75,-0.55) {{$G_i$}};
\draw[step=0.5cm,lightgray,very thin] (3.75,-0.25) grid (5.95,1.95);
\draw[darkgray,->] (4,0) -- (5.85,0) node[anchor=north west] {$\s_1$};
\draw[darkgray,->] (4,0) -- (4,1.85) node[anchor=south east] {$\s_2$};
\node(a) at (4,0) {{\textcolor{red}{1}}};
\node(b) at ([xshift=0.5cm]$(a)$) {{\textcolor{red}{1}}};
\node(c) at ([xshift=0.5cm]$(b)$) {{\textcolor{red}{1}}};
\node(d) at ([yshift=0.5cm]$(a)$) {{\textcolor{red}{1}}};
\node(e) at ([xshift=0.5cm]$(d)$) {{\textcolor{red}{1}}};
\node(f) at ([yshift=0.5cm]$(d)$) {{\textcolor{red}{1}}};
\fill[red] [radius=2pt] (5.5,0) circle[] (5,0.5) circle[] (4.5,1) circle[] (4,1.5) circle[];
\node at (4.75,-0.55) {{$H_i$}};
\draw[step=0.5cm,lightgray,very thin] (7.75,-0.25) grid (9.95,1.95);
\draw[darkgray,->] (8,0) -- (9.85,0) node[anchor=north west] {$\s_1$};
\draw[darkgray,->] (8,0) -- (8,1.85) node[anchor=south east] {$\s_2$};
\node(a) at (8,0) {{\textcolor{red!30}{1}}};
\node(b) at ([xshift=0.5cm]$(a)$) {{\textcolor{red!30}{1}}};
\node(c) at ([xshift=0.5cm]$(b)$) {{\textcolor{red!30}{1}}};
\node(d) at ([yshift=0.5cm]$(a)$) {{\textcolor{red!30}{1}}};
\node(e) at ([xshift=0.5cm]$(d)$) {{\textcolor{red!30}{1}}};
\node(f) at ([yshift=0.5cm]$(d)$) {{\textcolor{red!30}{1}}};
\fill[red] [radius=2pt] (9.5,0) circle[] (9,0.5) circle[] (8.5,1) circle[] (8,1.5) circle[];
\node at (8.75,-0.55) {{$\rho_i\bigl(H_i\bigr)$}};
\end{tikzpicture}
\caption{$G_i$, $H_i$ and $\rho_i\bigl(H_i\bigr)$ when $n=2$}
\label{fig ker: G[i] H[i] rho[i]H[i]}
\end{figure}

\begin{lemma}\label[lemma]{lem: rho restricted to H is isom}
Let $(G_i)_{i\in\mathbb{N}}$ be a generalised $\s$\=/algebraic group with respect to an algebraic group $\G$ over $k$, and let $(H_i)_{i\in\mathbb{N}}$ be the kernels of $\pi$ restricted to $(G_i)_{i\in\mathbb{N}}$.
    The restriction $\rho_i\colon H_i\rightarrow\rho_i\bigl(H_i\bigr)$ is an isomorphism, with inverse map $\rho^*_i\colon \rho_i\bigl(H_i\bigr)\rightarrow H_i$ for each $i\in\mathbb{N}$.
\end{lemma}
\begin{proof}
    This holds as the projections of $H_i$ onto $^\tau\G$ where $\ord(\tau)<i$ are all $1_{({^\tau}\G)}$. See \Cref{fig ker: G[i] H[i] rho[i]H[i]} to visualise this in the $n=2$ case.
\end{proof}

Looking at \Cref{fig ker: G[i] H[i] rho[i]H[i]} for the case where $n=2$, we see that $\rho_i\bigl(H_i\bigr)$ has $i+1$ many dots for each $i$, so it looks like perhaps it could be some kind of ordinary generalised difference algebraic group. It looks like building difference structure on $\G$ in some alternative way could allow us to realise $\rho_i\bigl(H_i\bigr)$ as a generalised $\s'$-algebraic group for some set $\s'$ of endomorphisms with $|\s'|=n-1$.

In order to do this, firstly assume that $\s_n\colon  k\rightarrow k$ is a bijection. Let $\s'=\{\s_1',\dots,\s_{n-1}'\}$, where $\s_j'=\s_j{\s_n}^{-1}$ for each $1\leq j\leq n-1$. As $k$ is a $\s$\=/field, it is also a $\s'$-field. Recall that $\G$ is an algebraic group over the $\s$\=/field $k$, and let $\G'$ denote $\G$ considered as an algebraic group over the $\s'$-field $k$. We can then build a $\s'$-algebraic group $[\s']_k\G'$ using \Cref{prop: alg grp is diff alg gr}. Again we will just write $\G'$ to mean $[\s']_k\G'$. For each $i\in\mathbb{N}$, one can construct  
$T_{\s'}[i]=\{(\s_1\s_n^{-1})^{k_1}\cdots(\s_{n-1}\s_n^{-1})^{k_{n-1}}\ |\ k_1+\cdots +k_{n-1}\leq i\}$ (adapting \cref{eq: Ts[i] and Ts(i)}) and hence the algebraic group $\G'[i]$ (adapting \cref{eq: G[i]}). 

\begin{lemma}\label[lemma]{lem : G'[i] twist of G(i)}
    Let $\G$ be an algebraic group over $k$ and suppose that $\s_n\colon  k\rightarrow k$ is a bijection. Let $\s'=\{\s_1',\dots,\s_{n-1}'\}$, where $\s_j'=\s_j{\s_n}^{-1}$ for each $1\leq j\leq n-1$, and let $\G'$ be $\G$ considered as an algebraic group over the $\s'$-field $k$.
    Then for each $i\in\mathbb{N}$,\begin{align*}
        \G'[i]={^{{\s_n}^{-i}}}\G(i)
    \end{align*}
    where ${^{{\s_n}^{-i}}}\G(i)$ is the algebraic group obtained from $\G(i)$ by base change via $\s_n^{-i}\colon k\rightarrow k$ (\Cref{def: group base change}).
\end{lemma}
\begin{proof}
    Using combinatorics we can see that for each $i\in\mathbb{N}$, $T_{\s'}[i]=\s_n^{-i}(T_\s(i))$, and hence
\begin{align*}
    \G'[i]=\prod_{\tau\in T_{\s'}[i]}{^\tau}\G={^{{\s_n}^{-i}}}\left(\prod\limits_{\tau\in T_\s(i)}{^\tau}\G \right)={^{{\s_n}^{-i}}}\G(i)
\end{align*}
for each $i\in\mathbb{N}$.
\end{proof}

\begin{lemma}\label[lemma]{lem: form of pi' and sj'}
    Let $\G$ be an algebraic group over $k$ and suppose that $\s_n\colon  k\rightarrow k$ is a bijection. Let $\s'=\{\s_1',\dots,\s_{n-1}'\}$, where $\s_j'=\s_j{\s_n}^{-1}$ for each $1\leq j\leq n-1$, and let $\G'$ be $\G$ considered as an algebraic group over the $\s'$-field $k$.
    The projection maps \begin{align*}
         \pi'_i\colon \G'[i]\rightarrow\G'[i-1]\quad\text{and}\quad(\s_j')_i\colon \G'[i]\rightarrow{^{\s_j'}}(\G'[i-1]),
     \end{align*}
    which are the maps $\pi_i$ and $(\s_j)_i$ (\cref{eq: def of pi map,eq: def of sigma map}) adapted to $\G'$, can be written as \begin{align*}
    \pi'_i={^{{\s_n}^{-i}}}({^{\s_n}}\rho_{i-1}\circ (\s_n)_i\circ\rho^*_i)
\quad \text{and} \quad
    (\s_j')_i={^{{\s_n}^{-i}}}({^{\s_j}}\rho_{i-1}\circ (\s_j)_i\circ \rho^*_i)
\end{align*}for each $i\geq 1$ and $1\leq j\leq n-1$, where $\rho_i$ and $\rho^*_i$ are the maps defined in \cref{eq: def of rho,eq: def of rho star}.
\end{lemma}
\begin{proof}
These equalities hold because while the $\rho^*$ map introduces lots of $1_{({^\tau}\G)}$'s, they are later forgotten by the $\rho$ map.
See \Cref{fig ker: pi'[i] broken down} for a visualisation of the decomposition of $\pi'_i$ in the two endomorphisms case.
\end{proof}
    \begin{figure}[!htbp]
\centering
\begin{tikzpicture}
\draw[step=0.5cm,lightgray,very thin] (-0.25,-1.95) grid (1.95,0.25);
\draw[darkgray,->] (0,0) -- (1.85,0) node[anchor=south west] {$\s_1$};
\draw[darkgray,->] (0,0) -- (0,-1.85) node[anchor=south east] {$\s_2\ $};
\fill (1.5,-1.5) circle(2pt);
\fill (0,0) circle(2pt);
\fill (0.5,-0.5) circle(2pt);
\fill (1,-1) circle(2pt);
\node at (0.75,-2.25) {{$\G'[i]$}};
\node at (2.85,-0.5) {$\xrightarrow{{^{{\s_2}^{-i}}}\rho^*_i}$};
\draw[step=0.5cm,lightgray,very thin] (3.75,-1.95) grid (5.95,0.25);
\draw[darkgray,->] (4,0) -- (5.85,0) node[anchor=south west] {$\s_1$};
\draw[darkgray,->] (4,0) -- (4,-1.85) node[anchor=south east] {$\s_2\ $};
\fill (5.5,-1.5) circle(2pt);
\fill (4,0) circle(2pt);
\fill (4.5,-0.5) circle(2pt);
\fill (5,-1) circle(2pt);
\node(a) at (4,-1.5) {{$1$}};
\node(b) at ([xshift=0.5cm]$(a)$) {{{$1$}}};
\node(c) at ([xshift=0.5cm]$(b)$) {{{$1$}}};
\node(d) at ([yshift=0.5cm]$(a)$) {{{$1$}}};
\node(e) at ([xshift=0.5cm]$(d)$) {{{$1$}}};
\node(f) at ([yshift=0.5cm]$(d)$) {{{$1$}}};
\node at (6.85,-0.5) {$\xrightarrow{{^{{\s_2}^{-i}}}(\s_2)_i}$};
\draw[step=0.5cm,lightgray,very thin] (7.75,-1.95) grid (9.95,0.25);
\draw[darkgray,->] (8,0) -- (9.85,0) node[anchor=south west] {$\s_1$};
\draw[darkgray,->] (8,0) -- (8,-1.85) node[anchor=south east] {$\s_2\ $};
\fill (8,0) circle(2pt);
\fill (8.5,-0.5) circle(2pt);
\fill (9,-1) circle(2pt);
\fill[black!30] (9.5,-1.5) circle(2pt);
\node(a) at (8,-1.5) {\textcolor{black!30}{$1$}};
\node(b) at ([xshift=0.5cm]$(a)$) {\textcolor{black!30}{$1$}};
\node(c) at ([xshift=0.5cm]$(b)$) {\textcolor{black!30}{$1$}};
\node(d) at ([yshift=0.5cm]$(a)$) {{{$1$}}};
\node(e) at ([xshift=0.5cm]$(d)$) {{{$1$}}};
\node(f) at ([yshift=0.5cm]$(d)$) {{{$1$}}};
\node at (11.1,-0.5) {{$\xrightarrow{{^{{\s_2}^{-(i-1)}}}\rho_{i-1}}$}};
%
\draw[step=0.5cm,lightgray,very thin] (12.25,-1.95) grid (14.45,0.25);
\draw[darkgray,->] (12.5,0) -- (14.35,0) node[anchor=south west] {$\s_1$};
\draw[darkgray,->] (12.5,0) -- (12.5,-1.85) node[anchor=south east] {$\s_2\ $};
\fill (12.5,0) circle(2pt);
\fill (13,-0.5) circle(2pt);
\fill (13.5,-1) circle(2pt);
\fill[black!30] (14,-1.5) circle(2pt);
\node(a) at (12.5,-1.5) {\textcolor{black!30}{$1$}};
\node(b) at ([xshift=0.5cm]$(a)$) {\textcolor{black!30}{$1$}};
\node(c) at ([xshift=0.5cm]$(b)$) {\textcolor{black!30}{$1$}};
\node(d) at ([yshift=0.5cm]$(a)$) {\textcolor{black!30}{$1$}};
\node(e) at ([xshift=0.5cm]$(d)$) {\textcolor{black!30}{$1$}};
\node(f) at ([yshift=0.5cm]$(d)$) {\textcolor{black!30}{$1$}};
\node at (13.25,-2.25) {{$\G'[i-1]$}};
%
\end{tikzpicture}
\caption{Decomposition of $\pi'_i\colon \G'[i]\rightarrow \G'[i-1]$ when $n=2$}
\label{fig ker: pi'[i] broken down}
\end{figure}

We will now see the main motivation for introducing generalised $\s$\=/algebraic groups, which is that the kernels of the $\pi$ maps form a generalised $\s'$\=/algebraic group where $|\s'|= n-1$.

\begin{proposition}\label[proposition]{prop: twisted kernels n-1 groups}
Let $k$ be a $\s$\=/field, and assume that $\s_n\colon k\rightarrow k$ is a bijection. Let $(G_i)_{i\in\mathbb{N}}$ be a generalised $\s$\=/algebraic group with respect to an algebraic group $\G$ over $k$. Let $(H_i)_{i\in\mathbb{N}}$ be the kernels of $\pi$ restricted to $(G_i)_{i\in\mathbb{N}}$. Let $\s'=\{\s_1',\dots,\s_{n-1}'\}$, where $\s_j'=\s_j{\s_n}^{-1}$ for each $1\leq j\leq n-1$, and let $\G'$ be $\G$ considered as an algebraic group over the $\s'$-field $k$. For every $i\in\mathbb{N}$, let
$H'_i={^{{\s_n}^{-i}}}(\rho_i\bigl(H_i\bigr))$.
    Then $(H'_i)_{i\in\mathbb{N}}$ is a generalised $\s'$\=/algebraic group with respect to $\G'$.
\end{proposition}
\begin{proof}
We know that for each $i\in\mathbb{N}$, $H_i\leq\G[i]$. Therefore, $\rho_i\bigl(H_i\bigr)\leq \G(i)$ and we see that \begin{align*}
    H'_i={^{{\s_n}^{-i}}}\bigl(\rho_i\bigl(H_i\bigr)\bigr)\leq {^{{\s_n}^{-i}}}\G(i)=\G'[i]
\end{align*}
by \Cref{lem : G'[i] twist of G(i)}. See \Cref{fig ker: going from H[i] to H'[i]} for a visualisation of this in the two endomorphisms case.

\begin{figure}[!htbp]
\centering
\begin{tikzpicture}
\draw[step=0.5cm,lightgray,very thin] (-0.45,-0.45) grid (1.95,1.95);
\draw[darkgray,->] (0,0) -- (1.85,0) node[anchor=north west] {$\s_1$};
\draw[darkgray,->] (0,0) -- (0,1.85) node[anchor=south east] {$\s_2$};
\fill[red] [radius=2pt] (1.5,0) circle[] (1,0.5) circle[] (0.5,1) circle[] (0,1.5) circle[];
\node(a) at (0,0) {\textcolor{red}{$1$}};
\node(b) at ([xshift=0.5cm]$(a)$) {{\textcolor{red}{1}}};
\node(c) at ([xshift=0.5cm]$(b)$) {{\textcolor{red}{1}}};
\node(d) at ([yshift=0.5cm]$(a)$) {{\textcolor{red}{1}}};
\node(e) at ([xshift=0.5cm]$(d)$) {{\textcolor{red}{1}}};
\node(f) at ([yshift=0.5cm]$(d)$) {{\textcolor{red}{1}}};
\node at (0.75,-0.75) {{$H_i$}};
\draw[step=0.5cm,lightgray,very thin] (3.55,-0.45) grid (5.95,1.95);
\draw[darkgray,->] (4,0) -- (5.85,0) node[anchor=north west] {$\s_1$};
\draw[darkgray,->] (4,0) -- (4,1.85) node[anchor=south east] {$\s_2$};
\fill[red] [radius=2pt] (5.5,0) circle[] (5,0.5) circle[] (4.5,1) circle[] (4,1.5) circle[];
\node(a) at (4,0) {\textcolor{red!30}{$1$}};
\node(b) at ([xshift=0.5cm]$(a)$) {{\textcolor{red!30}{1}}};
\node(c) at ([xshift=0.5cm]$(b)$) {{\textcolor{red!30}{1}}};
\node(d) at ([yshift=0.5cm]$(a)$) {{\textcolor{red!30}{1}}};
\node(e) at ([xshift=0.5cm]$(d)$) {{\textcolor{red!30}{1}}};
\node(f) at ([yshift=0.5cm]$(d)$) {{\textcolor{red!30}{1}}};
\node at (4.75,-0.75) {{$\rho_i\bigl(H_i\bigr)$}};
\draw[step=0.5cm,lightgray,very thin] (7.55,-0.45) grid (9.95,1.95);
\draw[darkgray,->] (8,1.5) -- (9.85,1.5) node[anchor=north west] {$\s_1$};
\draw[darkgray,->] (8,1.5) -- (8,-0.35) node[anchor=south east] {$\s_2$};
\fill[red] [radius=2pt] (9.5,0) circle[] (9,0.5) circle[] (8.5,1) circle[];
\fill[red] (8,1.5) circle(2pt); 
\node at (8.75,-0.75) {{$H'_i={^{{\s_2}^{-i}}}(\rho_i\bigl(H_i\bigr))$}};
\draw[step=0.5cm,lightgray,very thin] (11.55,-0.45) grid (13.95,1.95);
\draw[darkgray,->] (12,1.5) -- (13.85,1.5) node[anchor=north west] {$\s_1$};
\draw[darkgray,->] (12,1.5) -- (12,-0.35) node[anchor=south east] {$\s_2$};
\fill [radius=2pt] (13.5,0) circle[] (13,0.5) circle[] (12.5,1) circle[];
\fill (12,1.5) circle(2pt); 
\node at (12.75,-0.75) {{$\G'[i]$}};
\end{tikzpicture}
\caption{$H_i$, $\rho_i\bigl(H_i\bigr)$, $H'_i$ and $\G'[i]$ when $n=2$ }
\label{fig ker: going from H[i] to H'[i]}
\end{figure}

Therefore, to see that $(H'_i)_{i\in\mathbb{N}}$ is a generalised $\s'$\=/algebraic group with respect to $\G'$, it just remains to show that for each $i\geq 1$, there are induced restrictions $\pi'_i\colon H'_i\rightarrow H'_{i-1}$ and $(\s_j')_i\colon H'_i\rightarrow{^{\s_j'}}H'_{i-1}$ for each $1\leq j\leq n-1$. This follows due to \Cref{lem: rho restricted to H is isom,lem: form of pi' and sj',cor: restriction sj to H[i] well defined} and the definition of $H'_i$.
\end{proof}

Using the constructions and notation from the previous proposition, in the case where $\s_n\colon k\rightarrow k$ is a bijection, notice that for each $i\in\mathbb{N}$, $H'_i$ is defined as a $\s'$\=/closed subgroup of $\G'[i]$ by the Hopf ideal \begin{align}\label{eq: defining ideal twisted kernels}
    \I(H'_i)=\s_n^{-i}\bigl(\I\bigl(\rho_i\bigl(H_i\bigr)\bigr)\bigr)=\s_n^{-i}\bigl(\I(H_i)\cap k[\G(i)]\bigr).
\end{align}

\begin{example}\label[example]{}
Let $k$ be a $\s$\=/field such that $\s=\{\s_1,\s_2\}$. 
Consider the $\s$\=/closed subgroup $G$ of $\mathbb{G}_m$ introduced in \Cref{ex:introducing example}, and the Zariski closures $(G[i])_{i\in\mathbb{N}}$ of $G$ with respect to $\mathbb{G}_m$, whose defining ideals are found in \Cref{ex: zariski closures example}.
         We know that $(G[i])_{i\in\mathbb{N}}$ is a generalised $\s$\=/algebraic group with respect to $\mathbb{G}_m$. The kernels $(H_i)_{i\in\mathbb{N}}$ of $\pi$ restricted to $(G[i])_{i\in\mathbb{N}}$ are defined by $\I(H_i)=(\I(G[i]),\{\tau (y)-1\ |\ \tau\in T_\s[i-1]\})$ for each $i\geq 1$, and $\I(H_0)=\I(G[0])=(0)$. Notice that \begin{align*}
             \I(H_i)&=\bigl(\{\tau( y)-1\ |\ \tau\in T_\s[i-1]\}\bigr)\subseteq k[\mathbb{G}_m[i]] &&\text{for $1\leq i\leq 3$}\\
             \I(H_i)&=\bigl(\{\tau (y)-1\ |\ \tau\in T_\s[i-1]\}, \{\tau\s_2^4(y)-1\ |\ \tau\in T_\s(i-4)\}\bigr)\subseteq k[\mathbb{G}_m[i]] &&\text{for $i\geq 4$}.
         \end{align*} 
         Now suppose that $\s_2\colon k\rightarrow k$ is a bijection. We can apply the constructions from 
    \Cref{prop: twisted kernels n-1 groups} to find the generalised $\s'$\=/algebraic group $(H_i')_{i\in\mathbb{N}}$ corresponding to these kernels $(H_i)_{i\in\mathbb{N}}$, where $\s'=\{\s_1\s_2^{-1}\}$. We can use (\ref{eq: defining ideal twisted kernels}) to see that \begin{align*}
        \I(H_i')&=(0)\subseteq k[\mathbb{G}_m'[i]] &&\text{for $0\leq i\leq 3$}\\
             \I(H_i')&=\s_2^{-i}\bigl(\{(\tau\s_2^4)(y)-1\ |\ \tau\in T_\s(i-4)\}\bigr)
    =\bigl(\{\tau'(y)-1\ |\ \tau'\in T_{\s'}[i-4]\}\bigr) &&\text{for $i\geq 4$},
    \end{align*} using the fact that $T_{\s'}[i]=\s_2^{-i}T_\s(i)$ for each $i\in\mathbb{N}$. The chain of Hopf ideals $(\I(H_i'))_{i\in\mathbb{N}}$ satisfies the required properties to define a generalised $\s'$\=/algebraic group $(H_i')_{i\in\mathbb{N}}$. Notice that if we let $H'$ be the $\s'$\=/closed subgroup of $\G'$ defined by $\s'$\=/Hopf ideal $\I(H')=\bigcup_{i\in\mathbb{N}}\I(H'_i)$, the $i$-th order Zariski closure of $H'$ with respect to $\mathbb{G}_m'$ would be defined by $(\{\tau'(y)-1\ |\ \tau'\in T_{\s'}[i]\})$. So we see that our adjusted kernels $(H_i')_{i\in\mathbb{N}}$ constructed using \Cref{prop: twisted kernels n-1 groups} are not the Zariski closures of this $\s'$\=/algebraic group, but they do form a generalised $\s'$\=/algebraic group. This provides an example for the motivation for introducing the more general concept of a generalised difference algebraic group.
\end{example}

\Cref{prop: twisted kernels n-1 groups} will be useful when using induction to prove properties of generalised difference algebraic groups, as when inducting on the number of endomorphisms $n$, we will be able to apply the inductive hypothesis to the $(H'_i)_{i\in\mathbb{N}}$.

\section{Ideal Generation Property}\label{sec: ideal gen prop}
\begin{definition}\label[definition]{def: ideal gen prop}
    Let $\G$ be an algebraic group over $k$.
We say that a generalised $\s$\=/algebraic group $(G_i)_{i\in\mathbb{N}}$ with respect to $\G$ has the \textbf{ideal generation property} if \begin{align*}
    \I(G_{i+1})=(\I(G_i), \s_1(\I(G_i)),\dots, \s_n(\I(G_i)))
\end{align*}
for large enough $i\in \mathbb{N}$.
\end{definition}

We already know by \Cref{rem: one inclusion for gen grp} that the $\supseteq$ inclusion in \Cref{def: ideal gen prop} holds for any generalised $\s$\=/algebraic group.
Notice that $(G_i)_{i\in\mathbb{N}}$ having the ideal generation property is equivalent to the equality \begin{align}\label{eq: ideal gen in groups}
    G_{i+1}=\widehat{G_i}^{\G[i+1]}\cap\widehat{{^{\s_1}}G_i}^{\G[i+1]}\cap\cdots\cap \widehat{{^{\s_n}}G_i}^{\G[i+1]}
\end{align}
holding for large enough $i\in\mathbb{N}$. Again by \Cref{rem: one inclusion for gen grp}, the $\leq$ inclusion of (\ref{eq: ideal gen in groups}) holds for any generalised $\s$\=/algebraic group and any $i\in\mathbb{N}$.

We will use results from \Cref{sec: kernels} to prove that the kernels $(H_i)_{i\in\mathbb{N}}$ (\Cref{def: kernels}) of $\pi$ (\cref{eq: def of pi map}) restricted to some generalised $\s$\=/algebraic group $(G_i)_{i\in\mathbb{N}}$ with respect to an algebraic group $\G$ satisfy the ideal generation property. That is, we will see that \begin{align*}
    \I(H_{i+1})=(\I(H_i),\s_1(\I(H_i)),\dots,\s_n(\I(H_i)))
\end{align*}
for large enough $i\in\mathbb{N}$. Notice that $\s=\basicset$ are the original endomorphisms, not the twisted ones that we form in \Cref{prop: twisted kernels n-1 groups}. 
Firstly, we prove that if the ideal generation property holds for the kernels of any generalised $\s$\=/algebraic group, then it holds for any generalised $\s$\=/algebraic group.

\begin{proposition}\label[proposition]{prop: ideal gen holds for kernels then holds for groups}
Let $\G$ be an algebraic group over a $\s$\=/field $k$. Suppose that for any generalised $\s$\=/algebraic group $(G_i)_{i\in\mathbb{N}}$ with respect to $\G$, the kernels $(H_i)_{i\in\mathbb{N}}$ of $\pi$ restricted to $(G_i)_{i\in\mathbb{N}}$ satisfy the ideal generation property.
Then any generalised $\s$\=/algebraic group $(G_i)_{i\in\mathbb{N}}$ with respect to $\G$ satisfies the ideal generation property.
\end{proposition}
\begin{proof}
For a generalised $\s$\=/algebraic group $(G_i)_{i\in\mathbb{N}}$ with respect to $\G$, denote by $(H_i)_{i\in\mathbb{N}}$ the kernels of $\pi$ restricted to $(G_i)_{i\in\mathbb{N}}$ (\Cref{def: kernels}), and for any $i\in\mathbb{N}$, let \begin{align*}
    \H_{i+1}=\widehat{G_i}^{\G[i+1]}\cap\widehat{{^{\s_1}}G_i}^{\G[i+1]}\cap\cdots\cap \widehat{{^{\s_n}}G_i}^{\G[i+1]}.
\end{align*}  Firstly, we prove that for large enough $i\in\mathbb{N}$, \begin{align}\label{eq: ker of intersection}
    \ker\left(\pi_{i+1}\Big|_{\H_{i+1}}\right)=H_{i+1}
\end{align} 
holds.
Since $G_{i+1}\leq \H_{i+1}$ (by \Cref{rem: one inclusion for gen grp}), certainly $H_{i+1}\leq \ker(\pi_{i+1}|_{\H_{i+1}})$. 
Making appropriate choices of $A,B\subseteq C\subseteq T_\s$ for \Cref{cor: kernel of extensions}, for each $i\geq 1$, and each $1\leq j\leq n$, 
\begin{align*}
\ker\left(\pi_{i+1}\Big|_{\widehat{G_i}^{\G[i+1]}}\right)\leq\widehat{1_{\G[i]}}^{\G[i+1]}\leq\widehat{H_i}^{\G[i+1]}\ \text{and}\ \ker\left(\pi_{i+1}\Big|_{\widehat{{^{\s_j}}G_i}^{\G[i+1]}}\right)\leq\widehat{{^{\s_j}}H_i}^{\G[i+1]}.
\end{align*}
 
         Therefore, \begin{align*}
            \ker\left(\pi_{i+1}\Big|_{\H_{i+1}}\right)&\leq\ker\left(\pi_{i+1}\Big|_{\widehat{G_i}^{\G[i+1]}}\right)\cap\ker\left(\pi_{i+1}\Big|_{\widehat{{^{\s_1}}G_i}^{\G[i+1]}}\right) \cap\cdots\cap \ker\left(\pi_{i+1}\Big|_{\widehat{{^{\s_n}}G_i}^{\G[i+1]}}\right)\\
            &\leq \widehat{H_i}^{\G[i+1]}\cap \widehat{{^{\s_1}}H_i}^{\G[i+1]}\cap\cdots\cap \widehat{{^{\s_n}}H_i}^{\G[i+1]}\\
            &=H_{i+1}
        \end{align*} for large enough $i\in\mathbb{N}$, as the sequence $(H_i)_{i\in\mathbb{N}}$ has the ideal generation property.

Secondly, we will prove that under the assumptions of the proposition, Zariski closures have the ideal generation property.
For a $\s$\=/closed subgroup $G$ of $\G$, construct the Zariski closures $(G[i])_{i\in\mathbb{N}}$ of $G$ with respect to $\G$ (\Cref{def: zariski closures}), the kernels $(H_i)_{i\in\mathbb{N}}$ of $\pi$ restricted to $(G[i])_{i\in\mathbb{N}}$, and the algebraic group \begin{align*}
    \H_{i+1}=\widehat{G[i]}^{\G[i+1]}\cap\widehat{{^{\s_1}}G[i]}^{\G[i+1]}\cap\cdots \cap \widehat{{^{\s_n}G[i]}}^{\G[i+1]}
\end{align*}
for each $i\in\mathbb{N}$. We need to prove that $\H_{i+1}=G_{i+1}$ for large enough $i\in\mathbb{N}$, recall that $G_{i+1}\leq \H_{i+1}$ for any $i\in\mathbb{N}$ by \Cref{rem: one inclusion for gen grp}. Notice that for every $i\in\mathbb{N}$,
    $\pi_{i+1}(\H_{i+1})\leq \pi_{i+1}\left(\widehat{G[i]}^{\G[i+1]}\right)= G[i]$ by \Cref{lem: extension commutes with}(a), and hence the diagram 
\begin{align*}
        \begin{tikzpicture}  
\node (x)   {$G[i+1]$};
  \node (z) at ([xshift=4cm]$(x)$) {$\H_{i+1}$};  
   \node (s) at ([yshift=-1cm]$(x)!0.5!(z)$) {$G[i]$}; 
  \draw[right hook->] (x)-- (z) node[midway,above, color=black] {};  
  \draw[line] (x)-- (s) node[pos=0.75,left, color=black] {$\pi_{i+1}\ \ $};
  \draw[line] (z)--(s) node[pos=0.75,right, color=black] {$\ \ \pi_{i+1}$};
\end{tikzpicture} 
    \end{align*}
commutes. For each $i\in\mathbb{N}$, as $\pi_{i+1}\colon G[i+1]\rightarrow G[i]$ is a quotient map (\Cref{rem: zariski closures have restrictions}), the commutative diagram tells us that $\pi_{i+1}\colon \H_{i+1}\rightarrow G[i]$ is also a quotient map.
For large enough $i\in\mathbb{N}$, we know that $\ker(\pi_{i+1}|_{\H_{i+1}})= H_{i+1}$ (\ref{eq: ker of intersection}), hence $\pi_{i+1}\colon G[i+1]\rightarrow G[i]$ and $\pi_{i+1}\colon \H_{i+1}\rightarrow G[i]$ are both quotient maps with the same kernel, therefore $G[i+1]=\H_{i+1}$ as required. That is, the Zariski closures of any $\s$\=/closed subgroup $G$ of $\G$ satisfy the ideal generation property.

We are now ready to prove the theorem by induction on how `different' a generalised $\s$\=/algebraic group is from its Zariski closures. Let $(G_i)_{i\in\mathbb{N}}$ be a generalised $\s$\=/algebraic group with respect to $\G$. The Zariski closures $(G[i])_{i\in\mathbb{N}}$ of $(G_i)_{i\in\mathbb{N}}$ with respect to $\G$ (\Cref{def: zariski closures of gen grp}) satisfy the ideal generation property, so let $m\in\mathbb{N}$ be the smallest natural number such that  
    $\I(G[i+1])=(\I(G[i]),\s_1(\I(G[i])),\dots,\s_n(\I(G[i])))$
for all $i\geq m$. Let $(g_i)_{i\in\mathbb{N}}$ be the Zariski indicators of $(G_i)_{i\in\mathbb{N}}$ with respect to $\G$ (\Cref{def: zariski indicators}). 

If $g_m=0$, then $\I(G_i)=\I(G[i])=(0)$ for all $i\in\mathbb{N}$, and so $(G_i)_{i\in\mathbb{N}}$ implicitly has the ideal generation property. Hence, we can suppose $g_m> 0$, and so by \Cref{lem: gi=0 or i}, $g_m\geq m$. We will prove that generalised $\s$\=/algebraic groups satisfy the ideal generation property by induction on the value of $g_m-m\geq 0$. 

If $(G_i)_{i\in\mathbb{N}}$ is a generalised $\s$\=/algebraic group such that $g_m=m$, for every $i\geq m$, $\I(G[i])=\I(G_i)$ by \Cref{lem: jm=m means I(G[i])=I(Gi)} and hence $(G_i)_{i\in\mathbb{N}}$ has the ideal generation property. Suppose that the ideal generation property holds for any generalised $\s$\=/algebraic group $(G_i)_{i\in\mathbb{N}}$ where $g_m-m<p$ for some $p\geq 1$.

Let $(G_i)_{i\in\mathbb{N}}$ be a generalised $\s$\=/algebraic group with respect to $\G$ such that $g_m-m=p\geq 1$. Let $(F_i)_{i\in\mathbb{N}}$ be the projections of $(G_i)_{i\in\mathbb{N}}$ under $\pi$ (\Cref{def: projections}). That is, $F_i=\pi_{i+1}\bigl(G_{i+1}\bigr)$. \Cref{lem: projections also generalised grps} tells us that $(F_i)_{i\in\mathbb{N}}$ is a generalised $\s$\=/algebraic group with respect to $\G$, and by \Cref{lem: zariskis equal ji>i means tji leq ji-1}, $f_m<g_m$ and we can apply the inductive hypothesis to $(F_i)_{i\in\mathbb{N}}$. That is, the sequence $(F_i)_{i\in\mathbb{N}}$ has the ideal generation property.
Again, for every $i\in\mathbb{N}$, let $\H_{i+1}$ be the algebraic group \begin{align*}
    \H_{i+1}=\widehat{G_i}^{\G[i+1]}\cap\widehat{{^{\s_1}}G_i}^{\G[i+1]}\cap\cdots\cap \widehat{{^{\s_n}}G_i}^{\G[i+1]},
\end{align*} and recall that $G_{i+1}\leq \H_{i+1}$ (\Cref{rem: one inclusion for gen grp}). We will prove that $\H_{i+1}=G_{i+1}$ for large $i\in\mathbb{N}$. Notice that 
\begin{align}\label{eq: Gi leq Fi-1}
    G_i\leq \widehat{F_{i-1}}^{\G[i]}
\end{align}
for $i\geq 1$, as the extension of $F_{i-1}$ to $\G[i]$ is defined as a subgroup of $\G[i]$ by $\I(F_{i-1})k[\G[i]]$ (by \Cref{def: extension}), and $\I(F_{i-1})\subseteq \I(G_i)$ (\cref{eq: def ideal for F}), so $\I(F_{i-1})k[\G[i]]\subseteq \I(G_i)$. Further, \begin{align*}
    \pi_{i+1}(\H_{i+1})&={G_i}\cap \widehat{{^{\s_1}}\left(\pi_i\bigl(G_i\bigr)\right)}^{\G[i]}\cap\cdots\cap \widehat{{^{\s_n}}\left(\pi_i\bigl(G_i\bigr)\right)}^{\G[i]} &&\text{by \Cref{lem: extension commutes with} (a) and (c)}\\
    &\leq \widehat{F_{i-1}}^{\G[i]}\cap \widehat{{^{\s_1}}F_{i-1}}^{\G[i]}\cap\cdots\cap \widehat{{^{\s_n}}F_{i-1}}^{\G[i]}&&\text{by (\ref{eq: Gi leq Fi-1})}\\
    &=F_{i} &&\text{for $i\gg 0$}
\end{align*}
for large $i\in\mathbb{N}$ as $(F_i)_{i\in\mathbb{N}}$ has the ideal generation property. 
That is, the diagram 
\begin{align*}
        \begin{tikzpicture}  
\node (x)   {$G_{i+1}$};
  \node (z) at ([xshift=4cm]$(x)$) {$\H_{i+1}$};  
   \node (s) at ([yshift=-1cm]$(x)!0.5!(z)$) {$F_i$}; 
  \draw[right hook->] (x)-- (z) node[midway,above, color=black] {};  
  \draw[line] (x)-- (s) node[midway,left, color=black] {$\pi_{i+1}\ \ $};
  \draw[line] (z)--(s) node[midway,right, color=black] {$\ \ \pi_{i+1}$};
\end{tikzpicture} 
    \end{align*}
commutes for large enough $i\in\mathbb{N}$. For large enough $i\in\mathbb{N}$,
since $\pi_{i+1}\colon G_{i+1}\rightarrow F_i$ is a quotient map, $\pi_{i+1}\colon  \H_{i+1}\rightarrow F_i$ is also a quotient map.
 Further, $\ker\left(\pi_{i+1}\big|_{\H_{i+1}}\right)=H_{i+1}$ (\cref{eq: ker of intersection}), and hence $\pi_{i+1}\colon G_{i+1}\rightarrow F_i$ and $\pi_{i+1}\colon  \H_{i+1}\rightarrow F_i$ are both quotient maps with the same kernel, so $G_{i+1}=\H_{i+1}$. That is, the sequence $(G_i)_{i\in\mathbb{N}}$ satisfies the ideal generation property as required.
 \end{proof}

\subsection{Ideal Generation Property in the Ordinary Case}

We will now show that the ideal generation property holds for any ordinary generalised difference algebraic group. This is a generalisation of the result for the Zariski closures of an ordinary difference algebraic group \cite[Corollary 4.2, page 533]{wibmer2022finiteness}.

\begin{theorem}\label[theorem]{the: ordinary ideal gen}
Let $\G$ be an algebraic group over an ordinary $\s$\=/field $k$.
    Let $(G_i)_{i\in\mathbb{N}}$ be a generalised $\s$\=/algebraic group with respect to $\G$. Then  
        $\I(G_{i+1})=(\I(G_i),\s(\I(G_i)))$
    for large enough $i\in\mathbb{N}$.
\end{theorem}
\begin{proof}
 Let $(G_i)_{i\in\mathbb{N}}$ be a generalised ordinary $\s$\=/algebraic group with respect to $\G$, and let $(H_i)_{i\in\mathbb{N}}$ be the kernels of $\pi$ with respect to $(G_i)_{i\in\mathbb{N}}$.
By \Cref{lem: H[i+1] isomorphic to H[i] (ordinary)}, for large enough $i\in\mathbb{N}$, $(\s)_{i+1}\colon H_{i+1}\rightarrow {^\s}H_i$ is an isomorphism, and hence $H_{i+1}=1\times{^\s}H_i$.
Notice that this implies that for large enough $i\in\mathbb{N}$, \begin{align*}
    H_{i+1}=\widehat{H_i}^{\G[i+1]}\cap \widehat{{^\s}H_i}^{\G[i+1]}
\end{align*}
and hence 
    $\I(H_{i+1})=(\I(H_i),\s(\I(H_i)))$.
The theorem now follows from \Cref{prop: ideal gen holds for kernels then holds for groups}.
\end{proof}

\subsection{Ideal Generation Property in the Partial Case}

We aim to extend the ideal generation property to the $n$-endomorphisms case, see \Cref{fig ideal: diagram of intersections g[i+1]} for a visualisation of this in the two endomorphisms case.


\begin{figure}[!htbp]
\centering
\begin{tikzpicture}
\draw[step=0.5cm,lightgray,very thin] (-0.25,-0.25) grid (1.95,1.95);
\draw[darkgray,->] (0,0) -- (1.85,0) node[anchor=north west] {$\s_1$};
\draw[darkgray,->] (0,0) -- (0,1.85) node[anchor=south east] {$\s_2$};
\fill[red] (0,0) circle(2pt);
\fill[red] (0.5,0) circle(2pt) (0,0.5) circle(2pt);
\fill[red] [radius=2pt] (1,0) circle[] (0.5,0.5) circle[] (0,1) circle[];
\fill[red] [radius=2pt] (1.5,0) circle[] (1,0.5) circle[] (0.5,1) circle[] (0,1.5) circle[];
\node at (0.75,-0.65) {{$G_{i+1}$}};
\node at (2.75,-0.65) {$=$};
\draw[step=0.5cm,lightgray,very thin] (3.25,-0.25) grid (5.45,1.95);
\draw[darkgray,->] (3.5,0) -- (5.35,0) node[anchor=north west] {$\s_1$};
\draw[darkgray,->] (3.5,0) -- (3.5,1.85) node[anchor=south east] {$\s_2$};
\fill[red] (3.5,0) circle(2pt);
\fill[red] (4,0) circle(2pt) (3.5,0.5) circle(2pt);
\fill[red] [radius=2pt] (4.5,0) circle[] (4,0.5) circle[] (3.5,1) circle[];
\fill [radius=2pt] (5,0) circle[] (4.5,0.5) circle[] (4,1) circle[] (3.5,1.5) circle[];
\node at (4.5,-0.65) {{$\widehat{G_i}^{\G[i+1]}$}};
\node at (6.25,-0.65) {$\cap$};
\draw[step=0.5cm,lightgray,very thin] (6.75,-0.25) grid (8.95,1.95);
\draw[darkgray,->] (7,0) -- (8.85,0) node[anchor=north west] {$\s_1$};
\draw[darkgray,->] (7,0) -- (7,1.85) node[anchor=south east] {$\s_2$};
\fill[red] (7.5,0) circle(2pt);
\fill[red] [radius=2pt] (8,0) circle[] (7.5,0.5) circle[];
\fill[red] [radius=2pt] (8.5,0) circle[] (8,0.5) circle[] (7.5,1) circle[];
\fill (7,0) circle(2pt) (7,0.5) circle(2pt) (7,1) circle(2pt) (7,1.5) circle(2pt); 
\node at (8,-0.65) {{$\widehat{{^{\s_1}}G_i}^{\G[i+1]}$}};
\node at (9.75,-0.65) {$\cap$};
\draw[step=0.5cm,lightgray,very thin] (10.25,-0.25) grid (12.45,1.95);
\draw[darkgray,->] (10.5,0) -- (12.35,0) node[anchor=north west] {$\s_1$};
\draw[darkgray,->] (10.5,0) -- (10.5,1.85) node[anchor=south east] {$\s_2$};
\fill[red] (10.5,0.5) circle(2pt);
\fill[red] [radius=2pt] (11,0.5) circle[] (10.5,1) circle[];
\fill[red] [radius=2pt] (11.5,0.5) circle[] (11,1) circle[] (10.5,1.5) circle[];
\fill (10.5,0) circle(2pt) (11,0) circle(2pt) (11.5,0) circle(2pt) (12,0) circle(2pt);
\node at (11.5,-0.65) {{$\widehat{{^{\s_2}}G_{i}}^{\G[i+1]}$}};
\end{tikzpicture}
\caption{Decomposition of $G_{i+1}$ when $n=2$}
\label{fig ideal: diagram of intersections g[i+1]}
\end{figure}

\begin{theorem}\label[theorem]{the: partial ideal generation property}
    Let $\G$ be an algebraic group over a $\s$\=/field $k$. Let $(G_i)_{i\in\mathbb{N}}$ be a generalised $\s$\=/algebraic group with respect to $\G$. Then
    \begin{align*}
        \I(G_{i+1})=(\I(G_i),\s_1(\I(G_i)),\dots,\s_n(\I(G_i)))
    \end{align*}
    for large enough $i\in\mathbb{N}$.
\end{theorem}
\begin{proof}
    We will prove this by induction on the number of endomorphisms $n$. This is proven in the ordinary case in \Cref{the: ordinary ideal gen}.
Now suppose that $n\geq2$, and assume that for any generalised $\s$\=/algebraic group $(G_i)_{i\in\mathbb{N}}$ with respect to $\G$ such that $|\s|=n-1$, $(G_i)_{i\in\mathbb{N}}$ has the ideal generation property.

      Let $\s=\basicset$,  and let $\G$ be an algebraic group over a $\s$\=/field $k$. Let $(G_i)_{i\in\mathbb{N}}$ be a generalised $\s$\=/algebraic group with respect to $\G$, and let $(H_i)_{i\in\mathbb{N}}$ be the kernels of $\pi$ restricted to $(G_i)_{i\in\mathbb{N}}$ (\Cref{def: kernels}). 
    We will prove that the sequence $(H_i)_{i\in\mathbb{N}}$ has the ideal generation property. Let us firstly suppose that $\s_n\colon k\rightarrow k$ is a bijection.

Let $\s'=\{\s_1',\dots,\s_{n-1}'\}$, where $\s_j'=\s_j{\s_n}^{-1}$ for each $1\leq j\leq n-1$, and let $\G'$ denote $\G$ considered as an algebraic group over the $\s'$-field $k$.
Using \Cref{prop: twisted kernels n-1 groups}, we can construct the sequence $(H'_i)_{i\in\mathbb{N}}$, which is a generalised $\s'$\=/algebraic group with respect to $\G'$. Since $|\s'|=n-1$, by the inductive hypothesis, the sequence $(H'_i)_{i\in\mathbb{N}}$ has the ideal generation property.
That is, for large enough $i\in\mathbb{N}$,
\begin{align}\label{eq: twisted kernels ideal gen}
    H'_{i+1}&=\widehat{H'_i}^{\G'[i+1]}\cap\widehat{{^{\s_1'}}H'_i}^{\G'[i+1]} \cap\cdots \cap \widehat{{^{\s_{n-1}'}}H'_i}^{\G'[i+1]} .
\end{align}
See \Cref{fig ideal: diagram of intersections h'[i+1]} for a visualisation of how this decomposition of $H'_{i+1}$ looks for the two endomorphisms case.

    \begin{figure}[!htbp]
\centering
\begin{tikzpicture}
\draw[step=0.5cm,lightgray,very thin] (-0.25,-1.95) grid (1.95,0.25);
\draw[darkgray,->] (0,0) -- (1.85,0) node[anchor=south west] {$\s_1$};
\draw[darkgray,->] (0,0) -- (0,-1.85) node[anchor=south east] {$\s_2\ $};
\fill[red] (1.5,-1.5) circle(2pt);
\fill[red] (0,0) circle(2pt);
\fill[red] (0.5,-0.5) circle(2pt);
\fill[red] (1,-1) circle(2pt);
\node at (0.75,-2.3) {{$H'_{i+1}$}};
\node at (2.5,-2.3) {$=$};
\draw[step=0.5cm,lightgray,very thin] (3.25,-1.95) grid (5.45,0.25);
\draw[darkgray,->] (3.5,0) -- (5.35,0) node[anchor=south west] {$\s_1$};
\draw[darkgray,->] (3.5,0) -- (3.5,-1.85) node[anchor=south east] {$\s_2\ $};
\fill (5,-1.5) circle(2pt);
\fill[red] (3.5,0) circle(2pt);
\fill[red] (4,-0.5) circle(2pt);
\fill[red] (4.5,-1) circle(2pt);
\node at (4.5,-2.3) {{$\widehat{{H'_i}}^{\G'[i+1]}$}};
\node at (6.15,-2.3) {$\cap$};
\draw[step=0.5cm,lightgray,very thin] (6.75,-1.95) grid (8.95,0.25);
\draw[darkgray,->] (7,0) -- (8.85,0) node[anchor=south west] {$\s_1$};
\draw[darkgray,->] (7,0) -- (7,-1.85) node[anchor=south east] {$\s_2\ $};
\fill[red] (8.5,-1.5) circle(2pt);
\fill (7,0) circle(2pt);
\fill[red] (7.5,-0.5) circle(2pt);
\fill[red] (8,-1) circle(2pt);
\node at (8,-2.3) {{$\widehat{{^{\s_1'}}{H'_i}}^{\G'[i+1]}$}};
\end{tikzpicture}
\caption{Decomposition of $H'_{i+1}$ when $n=2$}
\label{fig ideal: diagram of intersections h'[i+1]}
\end{figure}

Let 
    $\K_{i+1}=\widehat{H_i}^{\G[i+1]}\cap \widehat{{^{\s_1}}H_i}^{\G[i+1]}\cap\cdots\cap \widehat{{^{\s_n}}H_i}^{\G[i+1]}$
 for all $i\in\mathbb{N}$. Proving that $(H_i)_{i\in\mathbb{N}}$ has the ideal generation property is equivalent to proving that $\K_{i+1}=H_{i+1}$ for large enough $i\in\mathbb{N}$.
See \Cref{fig ideal: diagram of intersections h[i+1]} for a visualisation of this decomposition of $H_{i+1}$ for the two endomorphisms case. By \Cref{rem: one inclusion for gen grp}, we have $H_{i+1}\leq \K_{i+1}$ for any $i\in\mathbb{N}$.

\begin{figure}[!htbp]
\centering
\begin{tikzpicture}
\draw[step=0.5cm,lightgray,very thin] (-0.25,-0.25) grid (1.95,1.95);
\draw[darkgray,->] (0,0) -- (1.85,0) node[anchor=north west] {$\s_1$};
\draw[darkgray,->] (0,0) -- (0,1.85) node[anchor=south east] {$\s_2$};
\node (a) at (0,0) {\textcolor{red}{$1$}};
\node (b) at ([xshift=0.5cm]$(a)$) {\textcolor{red}{$1$}};
\node (c) at ([xshift=0.5cm]$(b)$) {\textcolor{red}{$1$}};
\node (d) at ([yshift=0.5cm]$(a)$) {\textcolor{red}{$1$}};
\node (e) at ([xshift=0.5cm]$(d)$) {\textcolor{red}{$1$}};
\node (f) at ([yshift=0.5cm]$(d)$) {\textcolor{red}{$1$}};
\fill[red] [radius=2pt] (1.5,0) circle[] (1,0.5) circle[] (0.5,1) circle[] (0,1.5) circle[];
\node at (1,-0.8) {{$H_{i+1}$}};
\node at (2.75,-0.8) {$=$};
\draw[step=0.5cm,lightgray,very thin] (3.25,-0.25) grid (5.45,1.95);
\draw[darkgray,->] (3.5,0) -- (5.35,0) node[anchor=north west] {$\s_1$};
\draw[darkgray,->] (3.5,0) -- (3.5,1.85) node[anchor=south east] {$\s_2$};
\node (a) at (3.5,0) {\textcolor{red}{$1$}};
\node (b) at ([xshift=0.5cm]$(a)$) {\textcolor{red}{$1$}};
\node (d) at ([yshift=0.5cm]$(a)$) {\textcolor{red}{$1$}};
\fill[red] [radius=2pt] (4.5,0) circle[] (4,0.5) circle[] (3.5,1) circle[];
\fill [radius=2pt] (5,0) circle[] (4.5,0.5) circle[] (4,1) circle[] (3.5,1.5) circle[];
\node at (4.5,-0.8) {{$\widehat{H_i}^{\G[i+1]}$}};
\node at (6.25,-0.8) {$\cap$};
\draw[step=0.5cm,lightgray,very thin] (6.75,-0.25) grid (8.95,1.95);
\draw[darkgray,->] (7,0) -- (8.85,0) node[anchor=north west] {$\s_1$};
\draw[darkgray,->] (7,0) -- (7,1.85) node[anchor=south east] {$\s_2$};
\node (a) at (7.5,0) {\textcolor{red}{$1$}};
\node (b) at ([xshift=0.5cm]$(a)$) {\textcolor{red}{$1$}};
\node (d) at ([yshift=0.5cm]$(a)$) {\textcolor{red}{$1$}};
\fill[red] [radius=2pt] (8.5,0) circle[] (8,0.5) circle[] (7.5,1) circle[];
\fill (7,0) circle(2pt) (7,0.5) circle(2pt) (7,1) circle(2pt) (7,1.5) circle(2pt); 
\node at (8,-0.8) {{$\widehat{{^{\s_1}}H_i}^{\G[i+1]}$}};
\node at (9.75,-0.8) {$\cap$};
\draw[step=0.5cm,lightgray,very thin] (10.25,-0.25) grid (12.45,1.95);
\draw[darkgray,->] (10.5,0) -- (12.35,0) node[anchor=north west] {$\s_1$};
\draw[darkgray,->] (10.5,0) -- (10.5,1.85) node[anchor=south east] {$\s_2$};
\node (a) at (10.5,0.5) {\textcolor{red}{$1$}};
\node (b) at ([xshift=0.5cm]$(a)$) {\textcolor{red}{$1$}};
\node (d) at ([yshift=0.5cm]$(a)$) {\textcolor{red}{$1$}};
\fill[red] [radius=2pt] (11.5,0.5) circle[] (11,1) circle[] (10.5,1.5) circle[];
\fill (10.5,0) circle(2pt) (11,0) circle(2pt) (11.5,0) circle(2pt) (12,0) circle(2pt);
\node at (11.5,-0.8) {{$\widehat{{^{\s_2}}H_i}^{\G[i+1]}$}};
\end{tikzpicture}
\caption{Decomposition of $H_{i+1}=\K_{i+1}$ when $n=2$}
\label{fig ideal: diagram of intersections h[i+1]}
\end{figure}

Let $i\geq 1$. Since $H_i$ is the kernel of $\pi_i$ restricted to $G_i$, for any $\tau\in T_\s[i-1]$, the projection of $H_i$ onto the $^\tau\G$ component is $1_{({^\tau} \G)}$. By \Cref{lem: extension commutes with}(c), ${^{\s_j}}H_i$ is the kernel of  ${^{\s_j}}\pi_i$ restricted to ${^{\s_j}}G_i$, so if for some $1\leq j\leq n$, $\tau\in \s_j(T_\s[i-1])$, then the projection of ${^{\s_j}}H_i$ onto the $^\tau \G$ component is also $1_{({^\tau} \G)}$. For any $\tau\in T_\s[i]$, either $\tau\in T_\s[i-1]$ or $\tau\in \s_j(T_\s[i-1])$ for some $1\leq j\leq n$, so for all $\tau\in T_\s[i]$, the projection of $\K_{i+1}$ onto ${^\tau}\G$ is $1_{({^\tau} \G)}$, and $\pi_{i+1}(\K_{i+1})=\prod_{\tau\in T_\s[i]}1_{({^\tau} \G)}=1_{\G[i]}$.

For $i\in\mathbb{N}$ and $1\leq j\leq n$,
\begin{align*}
     \rho_{i+1}\left(\widehat{{^{\s_j}}H_i}^{\G[i+1]}\right)=\widehat{{^{\s_j}}\left(\rho_i\bigl(H_i\bigr)\right)}^{\G(i+1)}
     = \widehat{{^{\s_j}}{^{{\s_n}^i}H'_i}}^{^{{\s_n}^{i+1}}\G'[i+1]}
     ={^{{\s_n}^{i+1}}}\left(\widehat{{^{\s_j\s_n^{-1}}}H'_i}^{\G'[i+1]} \right)
 \end{align*}
 where we have used \Cref{lem: projections commute with extension} for the first equality, \Cref{lem : G'[i] twist of G(i)} and \Cref{prop: twisted kernels n-1 groups} for the second, and \Cref{lem: base change commutes with extension} for the third. Further, by definition of the $\rho_{i+1}$ map (\cref{eq: def of rho}) and the extension of $H_i\leq G[i]$ to $\G[i+1]$ (\Cref{def: extension}), $\rho_{i+1}\left(\widehat{H_i}^{\G[i+1]}\right)=\G(i+1)$. Therefore \begin{align*}
    \rho_{i+1}\bigl(\K_{i+1}\bigr)
    &\leq\G(i+1)\cap \rho_{i+1}\left(\widehat{{^{\s_1}}H_i}^{\G[i+1]}\right)\cap\cdots\cap \rho_{i+1}\left(\widehat{{^{\s_n}}H_i}^{\G[i+1]}\right)\\
    &={^{{\s_n}^{i+1}}}\left(\widehat{{^{\s_1'}}H'_i}^{\G'[i+1]} \cap\cdots\cap\widehat{{^{\s_{n-1}'}}H'_i}^{\G'[i+1]} \cap\widehat{H'_i}^{\G'[i+1]} \right) \\
    &={^{{\s_n}^{i+1}}}H'_{i+1} &&\text{by (\ref{eq: twisted kernels ideal gen})}\\
    &=\rho_{i+1}\bigl(H_{i+1}\bigr) 
\end{align*}
for large enough $i\in\mathbb{N}$. Therefore, $\rho_{i+1}\bigl(\K_{i+1}\bigr)\leq \rho_{i+1}\bigl(H_{i+1}\bigr)$, and since \begin{align*}
    \pi_{i+1}(\K_{i+1})=1_{\G[i]}=\pi_{i+1}\bigl(H_{i+1}\bigr),
\end{align*} we have $\K_{i+1}\leq H_{i+1}$, hence $\K_{i+1}=H_{i+1}$.
That is, in the case where $\s_n\colon k\rightarrow k$ is a bijection, the sequence $(H_i)_{i\in\mathbb{N}}$ has the ideal generation property.

If we are working over a field where $\s_n\colon k\rightarrow k$ isn't a bijection, we can pass via base change to a field where it is a bijection. For example, we could again take the inversive closure of $k$ as described in the proof of \Cref{lem: H[i+1] isomorphic to H[i] (ordinary)}. After base change, the equality of groups will hold for large enough $i\in\mathbb{N}$. If the groups are equal after base change, they must be equal before base change, that is, for any $\s$\=/field $k$, the sequence $(H_i)_{i\in\mathbb{N}}$ has the ideal generation property.



We have shown that given any generalised $\s$\=/algebraic group $(G_i)_{i\in\mathbb{N}}$, that the kernels $(H_i)_{i\in\mathbb{N}}$ of $\pi$ restricted to $(G_i)_{i\in\mathbb{N}}$ have the ideal generation property. Then by \Cref{prop: ideal gen holds for kernels then holds for groups}, any generalised $\s$\=/algebraic group $(G_i)_{i\in\mathbb{N}}$ with respect to $\G$ satisfies the ideal generation property.
\end{proof}

\begin{example}\label[example]{ex: ideal gen prop ex}
    Consider the generalised $\s$\=/algebraic group $(G_i)_{i\in\mathbb{N}}$ with respect to $\mathbb{G}_m$ and its Zariski closures $(G[i])_{i\in\mathbb{N}}$ from \Cref{ex: zariski closures and gen group}.
    Notice that $\I(G_{i+1})=(\I(G_i),\s_1(\I(G_i)),\dots,\s_n(\I(G_i)))$ for all $i\geq 5$, and that $\I(G[i+1])=(\I(G[i]),\s_1(\I(G[i])),\dots,\s_n(\I(G[i])))$ for all $i\geq 4$.
\end{example}

\begin{corollary}\label[corollary]{cor: defining ideal fin generated}
    Let $\G$ be an algebraic group over $k$. For any $\s$\=/closed subgroup $G\leq \G$, the defining ideal $\I(G)\subseteq k\{\G\}$ of $G$ is finitely $\s$\=/generated. 
\end{corollary}
\begin{proof}
    Consider the Zariski closures $(G[i])_{i\in\mathbb{N}}$ of $G$ with respect to $\G$. As $(G[i])_{i\in\mathbb{N}}$ form a generalised $\s$\=/algebraic group with respect to $\G$, by \Cref{the: partial ideal generation property} there exists some $m\in\mathbb{N}$ such that for all $i\geq m$, we have
        $\I(G[i+1])=(\I(G[i]),\s_1(\I(G[i])),\dots,\s_n(\I(G[i])))$ . Notice that $\I(G)=[\I(G[m])]$, and since $\I(G[m])$ is a finitely generated ideal, $\I(G)$ is finitely $\s$\=/generated as required.
\end{proof}

\begin{corollary}\label[corollary]{cor: hopf ideal fin gen}
    Let $A$ be a finitely $\s$\=/generated $k$\=/$\s$\=/Hopf algebra. Then any $\s$\=/Hopf ideal $I\subseteq A$ is finitely $\s$\=/generated. 
\end{corollary}
\begin{proof}
    We know by \Cref{prop: one to one corresp groups and Hopf algs} that $A$ is the representing algebra for some $\s$\=/algebraic group $G$, which can be embedded into $\G=\operatorname{GL}_s$ for some $s\in\mathbb{N}$ by \Cref{prop: embedding into gln}. That is, $A\cong k\{\G\}/J$ for some $\s$\=/Hopf ideal $J\subseteq k\{\G\}$ by \Cref{prop: one to one corresp subgroups and Hopf ideals}. 
    Let $I\subseteq A$ be a $\s$\=/Hopf ideal of $A$. We can consider the preimage $I'$ of $I$ under the projection $k\{\G\}\rightarrow k\{\G\}/J$, which is a $\s$\=/Hopf ideal of $k\{\G\}$. That is, $I'$ defines a $\s$\=/closed subgroup of $\G$, and hence is finitely $\s$\=/generated by \Cref{cor: defining ideal fin generated}. Since $I$ is the image of $I'$ in $A$, $I$ is also finitely $\s$\=/generated as required.
\end{proof}

\subsection{Finite \s\=/Generation of Difference Hopf Subalgebras}

We will now see how we can use \Cref{cor: hopf ideal fin gen} to prove that any $k$\=/$\s$\=/Hopf subalgebra of a finitely $\s$\=/generated $k$\=/$\s$\=/Hopf algebra is finitely $\s$\=/generated.

\begin{lemma}\label[lemma]{lem: finite set contained}
    Let $A$ be a $k$\=/$\s$\=/Hopf algebra. Then every finite subset of $A$ is contained in a finitely $\s$\=/generated $k$\=/$\s$\=/Hopf subalgebra of $A$.
\end{lemma}
\begin{proof}
Let $F$ be a finite subset of $A$. We know from \cite[Section 3.3, page 24]{waterhouse2012introduction}, that $F$ is contained in a finitely generated $k$\=/Hopf subalgebra $B$ of $A$. Then $k\{B\}$ is a finitely $\s$\=/generated $k$\=/$\s$\=/subalgebra of $A$ containing $F$. Since the Hopf and difference structures on $A$ are compatible, $k\{B\}$ is in fact a $k$\=/$\s$\=/Hopf subalgebra of $A$ as required.
\end{proof}

Given a $k$\=/Hopf algebra $A$, recall the definition of the \textbf{augmentation ideal}, $\mathfrak{m}_A$, the kernel of the counit $\epsilon_A\colon A\rightarrow k$ \cite[Section 9.2, page 154]{montgomery1993hopf}. This is a Hopf ideal of $A$. If $A$ is a $k$\=/$\s$\=/Hopf algebra, then $\mathfrak{m}_A$ is a $\s$\=/Hopf ideal of $A$.

\begin{theorem}\label[theorem]{the: subalgebra fin gen}
    Let $A$ be a finitely $\s$\=/generated $k$\=/$\s$\=/Hopf algebra and let $B$ be a $k$\=/$\s$\=/Hopf subalgebra of $A$. Then $B$ is finitely $\s$\=/generated over $k$.
\end{theorem}
\begin{proof}
Since $\mathfrak{m}_B$ is a $\s$\=/Hopf ideal of $B$, the ideal $\mathfrak{m}_BA$ of $A$ generated by $\mathfrak{m}_B$ is a $\s$\=/Hopf ideal of $A$. By \Cref{cor: hopf ideal fin gen}, there is some finite set $F\subseteq \mathfrak{m}_BA$ that $\s$\=/generates $\mathfrak{m}_BA$ as a $\s$\=/ideal of $A$. We can assume that $F\subseteq \mathfrak{m}_B\subseteq B$. By \Cref{lem: finite set contained}, there is a finitely $\s$\=/generated $k$\=/$\s$\=/Hopf subalgebra $C$ of $B$ that contains $F$. Since $F\subseteq \mathfrak{m}_C$, and $\mathfrak{m}_C\subseteq \mathfrak{m}_B$, as ideals of $A$, $\mathfrak{m}_BA=\mathfrak{m}_CA$. Corollary 3.10 in \cite[Section 3, page 9]{takeuchi1972correspondence}, tells us that the mapping taking a Hopf subalgebra $C$ of $A$ to the Hopf ideal $\mathfrak{m}_CA$ of $A$ is injective. Therefore, $B=C$, and hence $B$ is a finitely $\s$\=/generated $k$\=/$\s$\=/Hopf subalgebra of $A$ as required.
\end{proof}

\section{Quotients of Difference Algebraic Groups}\label{sec: quotients}

Now we will adapt arguments from \cite[Section 3, pages 489-490]{wibmer2021almost}, to generalise the existence of the quotient of a difference algebraic group by any normal difference closed subgroup to the partial case. 

This section uses concepts from \cite[Section 3]{montgomery1993hopf}, and results on augmentation ideals from \cite[Section 4]{takeuchi1972correspondence}, .

\subsection{Normal Subgroups and Normal Hopf Ideals}

\begin{definition}
    Let $G$ be a $\s$\=/algebraic group over $k$. A $\s$\=/closed subgroup $N$ of $G$ is called \textbf{normal} if for every $k$\=/$\s$\=/algebra $R$, $N(R)$ is a normal subgroup of $G(R)$.
\end{definition}

For an element $a$ of a $k$\=/$\s$\=/Hopf algebra $A$, we write $\Delta(a)=\sum a_1\otimes a_2$. Due to coassociaitivity, we can write $(\Delta\otimes \id)\circ \Delta=(\id\otimes \Delta)\circ \Delta=\sum a_1\otimes a_2\otimes a_3$ without confusion.

Consider, for a $\s$\=/algebraic group $G$, the adjoint mapping $\psi\colon  G\times G\rightarrow G$, where for each $k$\=/$\s$\=/algebra $R$, $\psi_R\colon G(R)\times G(R)\rightarrow G(R)$ is of the form $(g,h)\mapsto ghg^{-1}$. This has dual map $\psi^*\colon k\{G\}\rightarrow k\{G\}\otimes_k k\{G\}$, such that $a\mapsto\sum (a_1S(a_3)\otimes a_2)$, where $S\colon A\rightarrow A$ is the counit of the $k$\=/$\s$\=/Hopf algebra $A$. 

A $\s$\=/closed subgroup $N$ of $G$ is normal if and only if $\psi(G\times N)\leq N$. This happens if and only if $\psi^*(k\{G\}/I)\subseteq(k\{G\}\otimes_k k\{G\})/(k\{G\}\otimes_k I)$, where $I$ is the defining ideal of $N$ in $G$. This holds if and only if $\psi^*(I)\subseteq k\{G\}\otimes_k I$.
This motivates the definition for a normal difference Hopf ideal, which is the difference analogue to a normal Hopf ideal as defined in \cite[Section 3.4, page 35]{montgomery1993hopf}.
\begin{definition}
    A $\s$\=/Hopf ideal $I$ of a $k$\=/$\s$\=/Hopf algebra $A$ is \textbf{normal} if $\psi^*(a)=\sum (a_1S(a_3)\otimes a_2)\in A\otimes_k I$ for any $a\in I$.
\end{definition}

Therefore, a $\s$\=/closed subgroup $N\leq G$ is normal if and only if its defining ideal $I\subseteq k\{G\}$ is a normal $\s$\=/Hopf ideal.

\subsection{Quotients}

Given a morphism $\phi\colon G\rightarrow H$ of $\s$\=/algebraic groups over $k$, the \textbf{kernel} $\ker(\phi)$ of $\phi$ is the functor from $k$\=/$\s$\=/$\operatorname{Alg}$ to $\operatorname{Sets}$ that takes a $k$\=/$\s$\=/algebra $R$ to $\ker(\phi)(R)=\ker(\phi_R)\subseteq G(R)$. The kernel $\ker(\phi)$ of $\phi$ is in fact a $\s$\=/closed subgroup of $G$, defined by the $\s$\=/Hopf ideal $\phi^*\bigl(\mathfrak{m}_{k\{H\}}\bigr)k\{G\}$ of $k\{G\}$, that is, the $\s$\=/ideal of $k\{G\}$ that is generated by the image of the augmentation ideal $\mathfrak{m}_{k\{H\}}$ of $k\{H\}$ under the dual morphism to $\phi$ \cite[Section 2.1, page 14]{waterhouse2012introduction}.

\begin{theorem}\label[theorem]{the: existence of quotients}
    Let $G$ be a $\s$\=/algebraic group over $k$, and let $N$ be a normal $\s$\=/closed subgroup of $G$. Then there exists a $\s$\=/algebraic group $G/N$ and a morphism $\pi\colon G\rightarrow G/N$ of $\s$\=/algebraic groups with $N\leq \ker(\pi)$ such that for any morphism $\phi\colon G\rightarrow H$ of $\s$\=/algebraic groups with $N\leq \ker(\phi)$, there exists a unique morphism of $\s$\=/algebraic groups $\psi\colon G/N\rightarrow H$ such that $\phi'\circ \pi=\phi$.  
\end{theorem}
\begin{proof}
    We know that $I=\I(N)$ is a normal $\s$\=/Hopf ideal of $k\{G\}$, so it is well defined to consider \begin{align*}
        k\{G\}(I)=\{f\in k\{G\}\ |\ \Delta(f)-f\otimes 1\in k\{G\}\otimes_k I\},
    \end{align*}
    which, by \cite[Lemma 4.4, page 11]{takeuchi1972correspondence},
    is a $k$\=/Hopf subalgebra of $k\{G\}$. Since the comultiplication $\Delta\colon k\{G\}\rightarrow k\{G\}\otimes_k k\{G\}$ is a morphism of $k$\=/$\s$\=/algebras, 
     $k\{G\}(I)$ is in fact a $k$\=/$\s$\=/Hopf subalgebra of $k\{G\}$. By \cite[Section 4, page 12]{takeuchi1972correspondence}, $k\{G\}(I)$ is the greatest $\s$\=/Hopf\=/subalgebra of $k\{G\}$ such that $\mathfrak{m}_{k\{G\}(I)}\subseteq I$, and in fact, $I=\mathfrak{m}_{k\{G\}(I)}k\{G\}$, the $\s$\=/Hopf ideal of $k\{G\}$ generated by the augmentation ideal $\mathfrak{m}_{k\{G\}(I)}$ of $k\{G\}(I)$. 
     By \Cref{the: subalgebra fin gen}, $k\{G\}(I)$ is finitely $\s$\=/generated and hence by  \Cref{prop: one to one corresp groups and Hopf algs}, is the representing algebra for some $\s$\=/algebraic group over $k$, which we call $G/N$. 
     
     Let $\pi\colon G\rightarrow G/N$ be the morphism of $\s$\=/algebraic groups over $k$ that is dual to the inclusion $\pi^*\colon k\{G/N\}\hookrightarrow k\{G\}$ of $k$\=/$\s$\=/Hopf algebras. Now, $\ker(\pi)$ is the $\s$\=/closed subgroup of $G$ defined by the $\s$\=/Hopf ideal $\mathfrak{m}_{k\{G/N\}}k\{G\}$ of $k\{G\}$. Since $k\{G/N\}=k\{G\}(I)$, we see that $\ker(\pi)$ is defined by $\mathfrak{m}_{k\{G\}(I)}k\{G\}=I$, that is, $\ker(\pi)=N$.
     Therefore, we have found a $\s$\=/algebraic group $G/N$ and a morphism $\pi\colon G\rightarrow G/N$ with $N\leq \ker(\pi)$, let us now show this satisfies the required property.

     Suppose that $\phi\colon  G\rightarrow H$ is a morphism of $\s$\=/algebraic groups with $N\leq \ker(\phi)$. As $\phi^*\colon k\{H\}\rightarrow k\{G\}$ is a morphism of $k$\=/$\s$\=/Hopf algebras, it respects the counit and hence $\phi^*\bigl(\mathfrak{m}_{k\{H\}}\bigr)k\{G\}=\mathfrak{m}_{\phi^*(k\{H\})}k\{G\}$. We know that $\I(\ker(\phi))=\phi^*\bigl(\mathfrak{m}_{k\{H\}}\bigr)k\{G\}\subseteq I$.  Since $k\{G/N\}$ is the largest Hopf subalgebra of $k\{G\}$ with $\mathfrak{m}_{k\{G/N\}}k\{G\}\subseteq I$, we see that $\phi^*(k\{H\})\subseteq k\{G/N\}$. That is, we can restrict the codomain $\phi'^*\colon k\{H\}\rightarrow k\{G/N\}$. Since $\pi^*\colon k\{G/N\}\rightarrow k\{G\}$ is the inclusion, $\pi^*\circ\phi'^*=\phi^*$, and further, $\phi'^*$ is the unique morphism of $k$\=/$\s$\=/Hopf algebras with this property. Hence, letting $\phi'\colon G/N\rightarrow H$ be the dual morphism to $\phi'^*$, $\phi'$ is the unique morphism of $\s$\=/algebraic groups such that $\phi'\circ\pi=\phi$ as required.
\end{proof}

For a $\s$\=/algebraic group $G$ over $k$ and a normal $\s$\=/closed subgroup $N$ of $G$, a morphism $\pi\colon G\rightarrow G/N$ satisfying the universal property in \Cref{the: existence of quotients} is called a \textbf{quotient} of $G \operatorname{mod} N$.

\begin{lemma}\label[lemma]{lem: quotient iff injective and kernel}
Let $G$ be a $\s$\=/algebraic group over $k$ and let $N$ be a normal $\s$\=/closed subgroup of $G$.
    A morphism of $\s$\=/algebraic groups $\pi'\colon G\rightarrow (G/N)'$ is a quotient of $G \operatorname{mod} N$ if and only if $\ker(\pi')=N$ and $\pi'^*\colon k\{(G/N)'\}\rightarrow k\{G\}$ is injective. 
\end{lemma}
\begin{proof}
Let $G/N$ and $\pi\colon G\rightarrow G/N$ be the $\s$\=/algebraic group and the morphism of $\s$\=/algebraic groups defined in the proof of \Cref{the: existence of quotients}. Then $\ker(\pi)=N$ and $\pi^*\colon k\{G/N\}\rightarrow k\{G\}$ is injective.

Now, suppose that we have $\s$\=/algebraic group $(G/N)'$ and a morphism $\pi'\colon G\rightarrow (G/N)'$ such that $\pi'$ is a quotient of $G \operatorname{mod} N$. Then there is a unique morphism $\phi\colon (G/N)'\rightarrow G/N$ such that $\phi\circ\pi'=\pi$, and a unique morphism $\psi\colon G/N\rightarrow (G/N)'$ such that $\psi\circ\pi=\pi'$. As $(\psi\circ\phi)\circ\pi'=\pi'$, by uniqueness, $\psi\circ\phi=\id_{(G/N)'}$, and similarly we see that $\phi\circ\psi=\id_{G/N}$. That is, $\phi$ is an isomorphism of $\s$\=/algebraic groups. Therefore, the dual morphism $\phi^*\colon k\{G/N\}\rightarrow k\{(G/N)'\}$ is an isomorphism of $k$\=/$\s$\=/Hopf algebras such that $\pi'^*\circ\phi^*=\pi^*$. Since $\pi^*$ is injective, $\pi'^*$ is also injective. 
As $\phi^*$ is an isomorphism of $k$\=/$\s$\=/Hopf algebras, $\phi^*(\mathfrak{m}_{k\{G/N\}})=\mathfrak{m}_{k\{(G/N)'\}}$. Using this along with the fact that $\pi'^*\circ\phi^*=\pi^*$, we see that \begin{align*}
    \I(\ker(\pi'))=\pi'^*(\mathfrak{m}_{k\{(G/N)'\}})k\{G\}=\pi^*(\mathfrak{m}_{k\{G/N\}})k\{G\}=\I(N),
\end{align*} and therefore $\ker(\pi')=N$ as required. 

    Conversely, suppose that $\ker(\pi')=N$ and that $\pi'^*$ is injective. 
    Then \begin{align*}
        \I(N)=\I(\ker(\pi'))=\pi'^*(\mathfrak{m}_{k\{(G/N)'\}})k\{G\}.
    \end{align*} As $k\{G/N\}$ 
 is the unique $k$\=/$\s$\=/Hopf subalgebra of $k\{G\}$ such that $\mathfrak{m}_{k\{G/N\}}k\{G\}=\I(N)$ \cite[Section 4, page 12]{takeuchi1972correspondence},  $\pi'^*(k\{(G/N)'\})=k\{G/N\}$, hence $(G/N)'$ and $G/N$ are isomorphic as $\s$\=/algebraic groups. It is then straightforward to check that $\pi'\colon G\rightarrow (G/N)'$ is a quotient of $G \operatorname{mod} N$.
\end{proof}

Notice that in the proof of \Cref{lem: quotient iff injective and kernel} we have proven that given a $\s$\=/algebraic group $G$ and a normal $\s$\=/closed subgroup $N$ of $G$, the quotient of $G\operatorname{mod} N$ is unique up to isomorphism.

\section{Dimension Polynomials}\label{sec: dim polys}

Recall that the \textbf{dimension} $\dim(\H)$ of an algebraic group $\H$ over $k$ is equal to the Krull dimension $\dim(k[\H])$ of the coordinate ring $k[\H]$ \cite[Chapter II Section 3.2, page 86]{hartshorne1977algebraic}. 
If  $A$ is a finitely generated $k$\=/algebra and $I\subseteq A$ is an ideal, then $\dim(A/I)=\dim(A)-\hgt(I)$ \cite[Corollary 13.4, page 286]{eisenbud1995commutative}, where the \textbf{height} $\hgt(\mathfrak{p})$ of a prime ideal $\mathfrak{p}$ is the supremum of the lengths of descending chains of prime ideals contained in $\mathfrak{p}$, and for a general ideal $I$, the height $\hgt(I)$ is the minimum of the heights of the prime ideals containing $I$.

The dimension polynomials whose existence we will prove are of a certain form; they are numerical.
A polynomial $f\in\mathbb{Q}[t]$ is called \textbf{numerical} if for large enough $i\in\mathbb{N}$, $f(i)\in \mathbb{Z}$ \cite[Section 2.1, page 45]{kondratieva1999differential}.

\begin{definition}
    Let $\G$ be an algebraic group over $k$.
We say that there exists a \textbf{dimension polynomial} for a generalised $\s$\=/algebraic group $(G_i)_{i\in \mathbb{N}}$ if there exists a numerical polynomial $\phi(t)\in \mathbb{Q}[t]$ such that for large enough $i\in \mathbb{N}$, $\phi(i)=\dim(G_i)$. We call $\phi(t)\in\mathbb{Q}[t] $ the dimension polynomial for $(G_i)_{i\in \mathbb{N}}$. The degree of the dimension polynomial is the degree of $\phi(t)\in\mathbb{Q}[t]$.
\end{definition}

We will prove that there exists a dimension polynomial for any generalised $\s$\=/algebraic group. The next lemma will be a useful tool in this proof.

\begin{lemma}\label[lemma]{lem: gi+1-gi implies dimension poly}
    Let $(G_i)_{i\in\mathbb{N}}$ be a generalised $\s$\=/algebraic group with respect to an algebraic group $\G$ over $k$. Suppose that there exists a numerical polynomial $\phi(t)\in \mathbb{Q}[t]$ of degree at most $n-1$ such that \begin{align*}
        \dim(G_{i+1})-\dim(G_i)=\phi(i)
    \end{align*}
    for large enough $i\in \mathbb{N}$. Then there exists a dimension polynomial of degree at most $n$ for $(G_i)_{i\in\mathbb{N}}$.
\end{lemma}
\begin{proof}
Let $m\in\mathbb{N}$ be such that  $\dim(G_{i+1})-\dim(G_i)=\phi(i)$ for all $i\geq m$, where $\phi(i)$ has degree $r\leq n-1$.  Then we have that 
        $\dim(G_{i+1})
        =\dim(G_m)+\sum_{j=m}^i\phi(j)$ for $i\geq m$.
    By \cite[Proposition 2.1.6, page 49]{kondratieva1999differential}, there exists a numerical polynomial $\psi(t)\in\mathbb{Q}[t]$ of degree $r+1$ such that \begin{align*}
        \psi(i)&=\phi(m+1)+\phi(m+2)+\cdots+\phi(m+(i-m))       =\dim(G_{i+1})-\dim(G_{m+1})
\end{align*}
for any $i\geq m$. Then, let $\psi'(t)=\psi(t-1)+\dim(G_{m+1})$ and notice that $\psi'(t)\in\mathbb{Q}[t]$ is a numerical polynomial of degree $r+1$ such that for $i\geq m+1$, $\psi'(i)=\dim(G_i)$. That is, there exists a dimension polynomial of degree at most $n$ for $(G_i)_{i\in\mathbb{N}}$.
\end{proof}

To prove the existence of the dimension polynomial, we will follow a similar method to the proof of the ideal generation property. Firstly, we prove that if there is a dimension polynomial for the kernels of any generalised $\s$\=/algebraic group, then there is a dimension polynomial for any generalised $\s$\=/algebraic group.

\begin{proposition}\label[proposition]{prop: kernels imply groups dim poly}
    Let $\G$ be an algebraic group over $k$. Suppose that for any generalised $\s$\=/algebraic group $(G_i)_{i\in\mathbb{N}}$ with respect to $\G$, there exists a dimension polynomial of degree at most $n-1$ for the kernels $(H_i)_{i\in\mathbb{N}}$ of $\pi$ restricted to $(G_i)_{i\in\mathbb{N}}$ (\Cref{def: kernels}).
Then there exists a dimension polynomial of degree at most $n$ for any generalised $\s$\=/algebraic group $(G_i)_{i\in\mathbb{N}}$ with respect to $\G$.
\end{proposition}
\begin{proof}

Let $G$ be a $\s$\=/closed subgroup of $\G$. We will prove that there is a dimension polynomial for the Zariski closures (\Cref{def: zariski closures}) $(G[i])_{i\in\mathbb{N}}$ of $G$ with respect to $\G$ .   
    Let $(H_i)_{i\in\mathbb{N}}$ be the kernels of $\pi$ restricted to $(G[i])_{i\in\mathbb{N}}$. By assumption, there exists a dimension polynomial $\phi(t)\in\mathbb{Q}[t]$ of degree at most $n-1$ for $(H_i)_{i\in\mathbb{N}}$. 
    For every $i\geq 1$, since $\pi_i\colon  G[i]\rightarrow G[i-1]$ is a quotient map, by \cite[Section 5c, page 104]{milne2017algebraic}, we have that 
        $\dim(G[i])-\dim(G[i-1])=\dim(H_i)$.
    That is, there exists a numerical polynomial $\phi(t)\in\mathbb{Q}[t]$ of degree at most $n-1$ such that
        $\dim(G[i])-\dim(G[i-1])=\phi(i)$
    for large enough $i\in\mathbb{N}$. Therefore, by \Cref{lem: gi+1-gi implies dimension poly}, there exists a dimension polynomial of degree at most $n$ for $(G[i])_{i\in\mathbb{N}}$. 

    Now, let $(G_i)_{i\in\mathbb{N}}$ be a generalised $\s$\=/algebraic group with respect to $\G$, and let $(G[i])_{i\in\mathbb{N}}$ be the Zariski closures of $(G_i)_{i\in\mathbb{N}}$ with respect to $\G$ (\Cref{def: zariski closures of gen grp}). We have already shown that there exists a dimension polynomial of degree at most $n$ for the Zariski closures $(G[i])_{i\in\mathbb{N}}$. Let $m\in\mathbb{N}$ be the smallest natural number such that 
    $\I(G[i+1])=(\I(G[i]),\s_1(\I(G[i])),\dots,\s_n(\I(G[i])))$
for all $i\geq m$. Such an $m$ exists by \Cref{the: partial ideal generation property}. Let $(g_i)_{i\in\mathbb{N}}$ be the Zariski indicators of $(G_i)_{i\in\mathbb{N}}$ with respect to $\G$ (\Cref{def: zariski indicators}). 
If $g_m=0$, then $\I(G_i)=\I(G[i])=(0)$ for all $i\in\mathbb{N}$, and so $(G_i)_{i\in\mathbb{N}}$ has the same dimension polynomial as its Zariski closures. Hence, we can suppose $g_m> 0$, and so by \Cref{lem: gi=0 or i}, $g_m\geq m$. We will prove that there exists a dimension polynomial for generalised $\s$\=/algebraic groups by induction on the value of $g_m-m\geq 0$.

If $(G_i)_{i\in\mathbb{N}}$ is such that $g_m-m=0$, for every $i\geq m$, $\I(G_i)=\I(G[i])$ by \Cref{lem: jm=m means I(G[i])=I(Gi)}, and hence the dimension polynomial for the sequence $(G[i])_{i\in\mathbb{N}}$ is also a dimension polynomial for the sequence $(G_i)_{i\in\mathbb{N}}$. Suppose that there exists a dimension polynomial of degree at most $n$ for any generalised $\s$\=/algebraic group $(G_i)_{i\in\mathbb{N}}$ with respect to $\G$ such that $g_m-m<p$ for some $p\geq 1$.

Let $(G_i)_{i\in\mathbb{N}}$ be a generalised $\s$\=/algebraic group with respect to $\G$ such that $g_m-m=p\geq 1$ and let $(F_i)_{i\in\mathbb{N}}$ be the projections of $(G_i)_{i\in\mathbb{N}}$ under $\pi$ (\Cref{def: projections}). By \Cref{lem: projections also generalised grps}, $(F_i)_{i\in\mathbb{N}}$ is a generalised $\s$\=/algebraic group with respect to $\G$, and by \Cref{lem: zariskis equal ji>i means tji leq ji-1}, we can apply the inductive hypothesis to $(F_i)_{i\in\mathbb{N}}$. That is, there exists a dimension polynomial $\phi(t)\in\mathbb{Q}[t]$ of degree at most $n$ for $(F_i)_{i\in\mathbb{N}}$.
By assumption, there exists a dimension polynomial $\psi(t)\in\mathbb{Q}[t]$ of degree at most $n-1$ for the kernels $(H_i)_{i\in\mathbb{N}}$.
Since for every $i\in\mathbb{N}$, $\pi_{i+1}\colon G_{i+1}\rightarrow F_i$ is a quotient map with kernel $H_{i+1}$, 
    $\dim(G_{i+1})=\dim(F_i) +\dim(H_{i+1})$
for all $i\in\mathbb{N}$.
That is, 
    $\dim(G_{i+1})=\phi(i) +\psi(i+1)$
for large enough $i\in\mathbb{N}$.
Letting $\Phi(t)=\phi(t-1)+\psi(t)$, we see that $\Phi(t)$ is a dimension polynomial for $(G_i)_{i\in\mathbb{N}}$ of degree at most $n$.
\end{proof}

\subsection{Dimension Polynomials in the Ordinary Case}

\begin{theorem}\label[theorem]{the: ordinary dimension polynomial}
    Let $\G$ be an algebraic group over an ordinary $\s$\=/field $k$. There exists a dimension polynomial of degree at most 1 for any generalised $\s$\=/algebraic group $(G_i)_{i\in\mathbb{N}}$ with respect to $\G$.
\end{theorem}
\begin{proof}
Let $(G_i)_{i\in\mathbb{N}}$ be a generalised ordinary $\s$\=/algebraic group with respect to $\G$, and let $(H_i)_{i\in\mathbb{N}}$ be the kernels of $\pi$ with respect to $(G_i)_{i\in\mathbb{N}}$. By \Cref{lem: H[i+1] isomorphic to H[i] (ordinary)}, for large enough $i\in\mathbb{N}$, $(\s)_i\colon H_i\rightarrow{^\s}H_{i-1}$ is an isomorphism. That is, for large enough $i\in\mathbb{N}$, $\dim(H_i)=\dim({^\s}H_{i-1})$, and as base change does not change dimension \cite[Section 3, page 95]{hartshorne1977algebraic}, $\dim(H_i)=\dim(H_{i-1})$.  That is, there exists some constant $d\in\mathbb{N}$, such that for all large enough $i\in\mathbb{N}$, $\dim(H_i)=d$. Therefore, there exists a dimension polynomial of degree at most $0$ for $(H_i)_{i\in\mathbb{N}}$.
    By \Cref{prop: kernels imply groups dim poly}, there exists a dimension polynomial of degree at most $1$ for any (ordinary) generalised $\s$\=/algebraic group $(G_i)_{i\in\mathbb{N}}$ with respect to $\G$.
\end{proof}

\subsection{Dimension Polynomials in the Partial Case}

\begin{theorem}\label[theorem]{the: partial dim poly}
    Let $\G$ be an algebraic group over a $\s$\=/field $k$. Let $(G_i)_{i\in\mathbb{N}}$ be a generalised $\s$\=/algebraic group with respect to $\G$. Then there exists a numerical polynomial $\Phi(t)\in\mathbb{Q}[t]$ of degree at most $n$ such that for large enough $i\in\mathbb{N}$, $\dim(G_i)=\Phi(i)$. 
\end{theorem}
\begin{proof}
    We will prove this by induction on the number of endomorphisms $n$, following similar steps to the proof of \Cref{the: partial ideal generation property}. This statement is true in the ordinary case by \Cref{the: ordinary dimension polynomial}. Let $\s=\basicset$ with $n\geq2$, and assume that for any generalised $\s$\=/algebraic group $(G_i)_{i\in\mathbb{N}}$ with respect to $\G$ such that $|\s|=n-1$, there exists a dimension polynomial of degree at most $n-1$ for $(G_i)_{i\in\mathbb{N}}$.

     Let $(G_i)_{i\in\mathbb{N}}$ be a generalised $\s$\=/algebraic group with respect to $\G$. 
    Firstly assume that $\s_n\colon k\rightarrow k$ is a bijection. Then we can form the generalised $\s'$\=/algebraic group $(H'_i)_{i\in\mathbb{N}}$ as per \Cref{prop: twisted kernels n-1 groups}, and apply the inductive hypothesis to say there exists a numerical polynomial $\phi(t)\in \mathbb{Q}[t]$ of degree at most $n-1$ such that for large enough $i\in\mathbb{N}$, $\phi(i)=\dim(H'_i)$. We can see that $\dim(H_i)=\dim(H'_i)$, by \Cref{lem: rho restricted to H is isom}, \Cref{prop: twisted kernels n-1 groups} and the fact that base change doesn't change dimension. So $\phi(t)$ is a dimension polynomial of degree at most $n-1$ for $(H_i)_{i\in\mathbb{N}}$. We can follow a similar argument as in the proof of \Cref{the: partial ideal generation property} to say that since such a dimension polynomial exists when $\s_n\colon k\rightarrow k$ is a bijection, by properties of base change it must exist over any $\s$\=/field $k$.
    Therefore, by \Cref{prop: kernels imply groups dim poly}, there exists a dimension polynomial of degree $n$ for any generalised $\s$\=/algebraic group $(G_i)_{i\in\mathbb{N}}$ with respect to $\G$.
\end{proof}

\begin{corollary}\label[corollary]{cor: partial dim poly zariskis}
    Let $G$ be a $\s$\=/closed subgroup of an algebraic group $\G$ over $k$, and let $(G[i])_{i\in\mathbb{N}}$ be the Zariski closures of $G$ with respect to $\G$ (\Cref{def: zariski closures}). Then there exists a numerical polynomial $\Phi(t)\in\mathbb{Q}[t]$ of degree at most $n$ such that for large enough $i\in\mathbb{N}$, $\dim(G[i])=\Phi[i]$. 
 \end{corollary}

 Using combinatorics, we can see that $|T_\s[i]|=\binom{i+n}{n}$. This allows us to do the computations in the following examples.

 \begin{example}\label[example]{ex: dim of G[i]}
    Suppose that $\G$ is an algebraic group over a $\s$\=/field $k$. Recall that $(\G[i])_{i\in\mathbb{N}}$ as defined in (\cref{eq: G[i]}) are the Zariski closures of $\G$ with respect to itself. Further, \begin{align*}
        \dim(\G[i])=\dim(\G)|T_\s[i]|=\dim(\G)\binom{i+n}{n} 
    \end{align*}
    for each $i\in\mathbb{N}$. That is, $\phi(t)=\dim(\G)\binom{t+n}{n}$ is a dimension polynomial for the Zariski closures of $\G$ with respect to itself.
\end{example}

 \begin{example}\label[example]{ex: dim polys for gen and zar}
    Consider the generalised $\s$\=/algebraic group $(G_i)_{i\in\mathbb{N}}$ with respect to $\mathbb{G}_m$ and its Zariski closures $(G[i])_{i\in\mathbb{N}}$ from \Cref{ex: zariski closures and gen group}.
    Firstly, recall that $\dim(\mathbb{G}_m)=1$, and hence for each $i\in\mathbb{N}$, $\dim(\mathbb{G}_m[i])=|T_\s[i]|=\binom{i+2}{2}$ by \Cref{ex: dim of G[i]}. 

    For the Zariski closures $(G[i])_{i\in\mathbb{N}}$, $\dim(G[i])=\dim(\mathbb{G}_m[i])$, for all $i\geq 4$, 
    \begin{align*}
        \dim(G[i])=\dim(\mathbb{G}_m[i])-\hgt(\I(G[i]))
        =|T_\s[i]|-|T_\s[i-4]|
        =4i-2,
    \end{align*}
    and hence $\phi(t)=4t-2$ is the dimension polynomial for the Zariski closures $(G[i])_{i\in\mathbb{N}}$ of $(G_i)_{i\in\mathbb{N}}$ in $\mathbb{G}_m$.
    
    Recall that $\I(G_i)=\I(G[i-1])k[\mathbb{G}_m[i]]$ for each $i\geq 1$, and hence $\hgt(\I(G_i))=\hgt(\I(G[i-1]))$ \cite[Exercise 10.2, page 242]{eisenbud1995commutative}. Therefore, for $i\geq 5$, we see that \begin{align*}
        \dim(G_i)=\dim(\mathbb{G}_m[i])-\hgt(\I(G[i-1]))
        =|T_\s[i]|-|T_\s[i-5]|
        =5i-5,
    \end{align*}
and hence $\psi(t)=5t-5$ is the dimension polynomial for the generalised $\s$\=/algebraic group $(G_i)_{i\in\mathbb{N}}$ with respect to $\mathbb{G}_m$.
\end{example}

 \subsection{Dimension Polynomial Invariants}

 We have seen (in \Cref{cor: partial dim poly zariskis}) that we can find a dimension polynomial $\phi(t)\in\mathbb{Q}[t]$ for the Zariski closures of a $\s$\=/closed subgroup $G$ of some algebraic group $\G$. 
We know from \cite[Section 1.4, page 48]{levin2008difference}, that if $\phi(t)$ is of degree $d$, then $\phi$ can be written in the form \begin{align*}
         \phi(t)=\sum_{j=0}^dc_j\binom{t+j}{j}
     \end{align*}
     where $c_0,\dots,c_d\in\mathbb{Z}$ are uniquely defined by $\phi(t)$, and $c_d\neq 0$. We will prove that $\phi$ holds some invariants of the $\s$\=/algebraic group $G$. Firstly, we make a useful remark.

     \begin{remark}\label[remark]{rem: zariski closures as image}
         Suppose that $(G[i])_{i\in\mathbb{N}}$ are the Zariski closures of a $\s$\=/closed subgroup $G$ of an algebraic group $\G$ (\Cref{def: zariski closures}). Rather than put $k[G[i]]=k[\G[i]]/(\I(G)\cap k[\G[i]])$, notice that we can consider $k[G[i]]$ to be the image of $k[\G[i]]$ in $k\{G\}$. Then, \Cref{ex: tau A and A[i],example: sk A} and \Cref{rem: tau constructions agree} tell us a lot about the structure of the $k[\G[i]]$ and $k\{\G\}$ for an algebraic group $\G$. 
From this, we can induce some properties of the representing algebras $k[G[i]]$ and $k\{G\}$. Firstly, if $k[G[0]]=k[A]$ for some $A\subseteq k\{G\}$, then for all $i\in\mathbb{N}$, $k[G[i]]=k[\tau(A)\ |\ \tau\in T_\s[i]]$, and if $B\subseteq k[G[m]]$ for some $m\in\mathbb{N}$, then for any $i\in\mathbb{N}$, $\tau\in T_\s[i]$, $\tau(B)\in k[G[m+i]]$. Finally, we see that $k\{G\}=\cup_{i\in\mathbb{N}}k[G[i]]$ and $k\{G\}=k\{A\}$.
     \end{remark}

 \begin{lemma}\label[lemma]{lem: dim poly invariants}
   Let $G$ be a $\s$\=/closed subgroup of an algebraic group $\G$ over $k$. By \Cref{cor: partial dim poly zariskis}, we can find a dimension polynomial for the Zariski closures of $G$ in $\G$. The degree and leading coefficient of this polynomial depend only on $G$ and not the choice of $\G$ or $\s$\=/closed embedding $G\hookrightarrow \G$.
 \end{lemma}
 \begin{proof}
This is the partial analogue to the ordinary case in \cite[Section 3, page 529]{wibmer2022finiteness}. 
Suppose that we have two $\s$\=/closed embeddings $G\hookrightarrow \G$ and $G\hookrightarrow \G'$, we can consider the Zariski closures $(G[i])_{i\in\mathbb{N}}$ and $(G'[i])_{i\in\mathbb{N}}$ of $G$ with respect to $\G$ and $\G'$ respectively, and by \Cref{the: partial dim poly}, we can find dimension polynomials $\phi(t)\in\mathbb{Q}[t]$ and $\phi'(t)\in\mathbb{Q}[t]$ for $(G[i])_{i\in\mathbb{N}}$ and $(G'[i])_{i\in\mathbb{N}}$ respectively.

We can find a finite subset $A\subseteq k\{G\}$ that generates $k[G[0]]$ as a $k$\=/algebra, and as $k\{G\}=\cup_{i\in\mathbb{N}}k[G'[i]]$, there exists some $m\in\mathbb{N}$ such that $A\subseteq k[G'[m]]$. Then, 
    $k[G[i]]=k[\tau(A)\ |\ \tau\in T_\s[i]]\subseteq k[G'[m+i]]$
for all $i\in\mathbb{N}$ by \Cref{rem: zariski closures as image}. That is, $\phi(i)\leq \phi'(i+m)$ for large $i\in\mathbb{N}$, and letting $i$ tend to infinity tells us $\deg(\phi)\leq \deg(\phi')$. Following the same argument, we can see that $\deg(\phi')\leq \deg(\phi)$, hence $\deg(\phi)= \deg(\phi')$. Repeating these arguments again using $\deg(\phi)= \deg(\phi')$ tells us that $\phi$ and $\phi'$ have the same leading coefficient.
 \end{proof}

As the degree and leading coefficient of the dimension polynomial for the Zariski closures of a $\s$\=/algebraic group $G$ (with respect to some embedding) are invariants of $G$, it is well defined to adapt the concepts of difference dimension, difference type and typical difference dimension from difference modules and difference field extensions \cite{levin2008difference} to partial difference algebraic groups.

 \begin{definition}\label[definition]{def: diff type diff dim etc}
    Let $G$ be a $\s$\=/closed subgroup of an algebraic group $\G$ over $k$, and let $\phi(t)\in\mathbb{Q}[t]$ be the dimension polynomial of the Zariski closures of $G$ with respect to $\G$, that is, $\phi(t)=\sum_{j=0}^dc_j\binom{t+j}{j}$ for some $d\leq n$, $c_0,\dots,c_d\in\mathbb{Z}$ and $c_d\neq 0$. The integer $d$ is called the \textbf{difference type} ($\s$\=/type) of $G$, and $c_d$ is the \textbf{typical difference dimension} (typical $\s$\=/dimension) of $G$. Finally, we define the \textbf{difference dimension} ($\s$\=/dimension) of $G$ to be $c_d$ if $d=n$, and $0$ if $d<n$. 
\end{definition}

By \Cref{lem: dim poly invariants}, the $\s$\=/type, typical $\s$\=/dimension and $\s$\=/dimension of a $\s$\=/closed subgroup $G$ of an algebraic group $\G$ depend only on $G$ and not the choice of $\G$ or $\s$\=/closed embedding $G\hookrightarrow \G$.

\begin{example}\label[example]{ex: invariants for algebraic group}
    Let $\G$ be an algebraic group over a $\s$\=/field $k$. In \Cref{ex: dim of G[i]}, we found that the dimension polynomial for the Zariski closures $(\G[i])_{i\in\mathbb{N}}$ of $\G$ with respect to itself is $\phi(t)=\dim(\G)\binom{t+n}{n}$. Therefore, we see that $G$ has $\s$\=/dimension $\dim(\G)$, $\s$\=/type $n$, and typical $\s$\=/dimension $\dim(\G)$.
\end{example}

\begin{example}\label[example]{ex: invariants for our G}
    Consider the $\s$\=/closed subgroup $G$ of $\mathbb{G}_m$ introduced in \Cref{ex:introducing example}. In \Cref{ex: dim polys for gen and zar}, we found that the dimension polynomial for the Zariski closures $(G[i])_{i\in\mathbb{N}}$ of $G$ in $\mathbb{G}_m$ is $\phi(t)=4t-2$. Therefore, we see that $G$ has $\s$\=/dimension $0$, $\s$\=/type $1$, and typical $\s$\=/dimension $4$.
\end{example}

\section{Comparison with Classical Results}\label{sec: relating polys to levin}

We will now see how the difference dimension, difference type and typical difference dimension defined for a difference algebraic group in \Cref{def: diff type diff dim etc} relate to the analogous definitions for difference modules and difference field extensions in \cite[Sections 3 and 4]{levin2008difference}. As in this reference, by definition, all basic endomorphisms on $\s$\=/rings are injective, we restrict to the case where our endomorphisms are injective. 

\subsection{Difference Modules}

Given a $\s$\=/ring $R$, we define a \textbf{difference operator} ($\s$\=/operator) over $R$ to be a sum of the form $\sum_{\tau\in T_\s}a_\tau\tau$, where $a_\tau\in R$ for all $\tau\in T_\s$ and all but finitely many of the $a_\tau$ are zero. The set $\D$ of all $\s$\=/operators over $R$ has a natural ring structure, and is called the 
\textbf{ring of $\s$\=/operators} over $R$. 
Further, we can assign a natural order to a $\s$\=/operator $A=\sum_{\tau\in T_\s}a_\tau\tau$, that is, $\ord(A)=\max\{\ord(\tau)\ |\ a_\tau\neq 0\}$. This allows us to define the \textbf{standard filtration} of $\D$, where $\D_i=0$ for $i<0$ and $\D_i=\{A\in\D\ |\ \ord(A)\leq i\}$ for all $i\in\mathbb{N}$. Notice that $\D_0=R$.

We call a $\D$\=/module a \textbf{difference $R$\=/module} ($\s$\=/$R$\=/module). A $\s$\=/$R$\=/module is \textbf{finitely $\s$\=/generated} if it is a finitely generated $\D$\=/module, and is a free $\s$\=/$R$\=/module with $\s$\=/basis $B$ if it is a free $\D$\=/module with $\D$\=/basis $B$. We can also consider a $\s$\=/$R$\=/module $M$ to be an $R$\=/module $M$ with pairwise commuting additive maps $\s_j\colon M\rightarrow M$ for each $1\leq j\leq n$ such that for all $r\in R$, $x\in M$, $\s_j(rx)=\s_j(r)\s_j(x)$. 

A \textbf{filtration} of a $\s$\=/$R$\=/module $M$ is an ascending chain $(M_i)_{i\in\mathbb{N}}$ of $R$\=/submodules of $M$ such that $\bigcup_{i\in\mathbb{Z}}M_i=M$, $\D_rM_i\subseteq M_{r+i}$ for all $r,i\in\mathbb{Z}$, and $M_i=0$ for small enough $i\in\mathbb{Z}$, where $(\D_i)_{i\in\mathbb{Z}}$ is the standard filtration of the ring $\D$ of $\s$\=/operators over $R$ \cite{levin2008difference}. A filtration $(M_i)_{i\in\mathbb{N}}$ of a $\s$\=/$R$\=/module $M$ is called \textbf{excellent} if for all $i\in\mathbb{Z}$, $M_i$ is a finitely generated $R$\=/module, and if there exists some $m\in\mathbb{Z}$ such that $M_i=\D_{i-m}M_{m}$ for all $i\geq m$. Notice that if a $\s$\=/$R$\=/module $M$ admits an excellent filtration, it is finitely $\s$\=/generated (as it is $\s$\=/generated by $M_{m}$, which is finitely generated).

We recall \cite[Theorem 3.2.3, page 161]{levin2008difference}, stating the existence of dimension polynomials for excellent filtrations of difference modules. In this reference, the theorem is proven for $\s$\=/$R$\=/modules where $R$ is an Artinian $\s$\=/ring, but here we will only state the result for $\s$\=/$k$\=/modules for a $\s$\=/field $k$. Therefore, while the dimension polynomials in this reference give the length of each $R$\=/module $M_i$ in a filtration $(M_i)_{i\in\mathbb{Z}}$, here we can instead consider the dimension of each $k$\=/module $M_i$ as a $k$\=/vector space.

\begin{theorem}\label[theorem]{the: levin dim poly modules}
    Let $k$ be a $\s$\=/field and let $M$ be a finitely $\s$\=/generated $\s$\=/$k$\=/module. Given an excellent filtration $(M_i)_{i\in\mathbb{Z}}$ of $M$, there exists some numerical polynomial $\phi(t)\in\mathbb{Q}[t]$ of degree at most $n$, such that for large enough $i\in\mathbb{Z}$, $\phi(i)=\dim_k(M_i)$.
\end{theorem}

We know that for an excellent filtration $(M_i)_{i\in\mathbb{N}}$ of a finitely $\s$\=/generated $\s$\=/$k$\=/module $M$, its dimension polynomial is of the form  $\psi(t)=\sum_{i=0}^dc_i\binom{t+i}{i}$, where $c_0,\dots,c_d\in\mathbb{Z}$, $c_d\neq 0$, and $d\leq n$. Further, by \cite[Theorem 3.2.9, page 164]{levin2008difference}, the degree $d$ of the polynomial and the integer $c_d$ are invariant of the choice of excellent filtration, they only rely on the $\s$\=/$k$\=/module $M$.
The integers $d$ and $c_d$ are called the \textbf{difference type} and \textbf{typical difference dimension} of the $\s$\=/$k$\=/module $M$ respectively. Further, we define the \textbf{difference dimension} of $M$ to be $c_d$ if $d=n$, and $0$ if $d<n$.

We will now see how given a finitely $\s$\=/generated $\s$\=/$k$\=/module $M$, we can find a corresponding $\s$\=/closed subgroup $G$ of the additive group $\mathbb{G}_a^s$ over $k$ for some $s\geq 1$. Then the difference dimension, difference type and typical difference dimension of $M$ and $G$ coincide.

\subsection{Relating Difference Modules to Difference Subgroups}

In this section, we will be considering the correspondence between finitely $\s$\=/generated $\s$\=/$k$\=/modules and $\s$\=/closed subgroups of the additive group $\mathbb{G}_a^s$ over $k$, for some $s\geq1$. We know from \Cref{ex: tau A and A[i]} that $k\{\mathbb{G}_a^s\}=k\{y_1,\dots,y_s\}$, and that $k\{\mathbb{G}_a^s\}$ is a $k$\=/$\s$\=/Hopf algebra, where the Hopf structure maps are the unique $k$\=/$\s$\=/algebra morphisms such that $\Delta(y_j)=y_j\otimes 1+1\otimes y_j$, $\epsilon(y_j)=0$, and $S(y_j)=-y_j$
for $1\leq j\leq s$. This structure is inherited from that of the algebraic group $\mathbb{G}_a$ as stated in \cite[Section 1.4, page 9]{waterhouse2012introduction}. 
Notice that the $\s$\=/ideal $[F]$ generated by a set $F\subseteq k\{y_1,\dots,y_s\}$ of homogeneous linear $\s$\=/polynomials is a $\s$\=/Hopf ideal of $k\{\mathbb{G}_a^s\}$.
For brevity, let $y=\{y_1,\dots,y_s\}$, $T_\s(y)=\{\tau (y_j)\ |\ 1\leq j\leq s,\ \tau\in T_\s\}$ and $T_\s[i](y)=\{\tau (y_j)\ |\ 1\leq j\leq s,\ \tau\in T_\s[i]\}$.

Let $M$ be a finitely $\s$\=/generated $\s$\=/$k$\=/module. Then $M$ is isomorphic to $N/K$, where $N$ the free $\s$\=/$k$\=/module with $\s$\=/basis $y=\{y_1,\dots,y_s\}$ for some $s\geq1$ and $K$ is a $\s$\=/$k$\=/submodule of $N$ that is $\s$\=/generated by some set $F\subseteq k\{y\}$ of homogeneous linear $\s$\=/polynomials.  
We know that $M$ can be considered as a $k$\=/module, generated by the images of the elements of $T_\s(y)$ in $M$. For each $i\in\mathbb{N}$, let $M_i$ be the $k$\=/submodule of $M$ generated by the images of $T_\s[i](y)$ in $M$, and for each $i<0$, let $M_i=0$. Notice that $(M_i)_{i\in\mathbb{Z}}$ is an excellent filtration of $M$, which we call the \textbf{standard excellent filtration} of $M$. Therefore, by \Cref{the: levin dim poly modules}, there exists a numerical polynomial $\psi(t)\in \mathbb{Q}[t]$ of degree at most $n$ such that for large enough $i\in\mathbb{Z}$, $\psi(i)=\dim_k(M_i)$. Notice that if for each $i\in\mathbb{N}$, we let $N_i$ be the free $k$\=/module with basis $T_\s[i](y)$, and let $K_i=N_i\cap H$, then $M_i$ is isomorphic to $N_i/K_i$ as a $k$\=/module.

Given a finitely $\s$\=/generated $\s$\=/$k$\=/module $M$ and the constructions above, by \Cref{prop: one to one corresp subgroups and Hopf ideals}, we can consider the corresponding $\s$\=/closed subgroup $G$ of $\mathbb{G}_a^s$ defined by the $\s$\=/Hopf ideal $I=[F]=(K)$, and we can define the Zariski closures $(G[i])_{i\in\mathbb{N}}$ of $G$ with respect to $\mathbb{G}_a^s$ (\Cref{def: zariski closures}). 
That is, $k\{G\}=k\{y\}/I$, and for each $i\in\mathbb{N}$, $k[G[i]]$ is the $k$\=/Hopf\=/subalgebra of $k\{G\}$ that is generated by the images of the elements of $T_\s[i](y)$ in $k\{G\}$ (see remark \Cref{rem: zariski closures as image}).

\begin{theorem}\label[theorem]{the: relating modules and groups}
Let $k$ be a $\s$\=/field.
    Let $M$ be a finitely $\s$\=/generated $\s$\=/$k$\=/module (with $s\geq 1$ generators), and let $G$ be the corresponding $\s$\=/closed subgroup of $\mathbb{G}_a^s$ as described above. 
    Then the dimension polynomials for the standard excellent filtration of $M$ and the Zariski closures of $G$ with respect to $\mathbb{G}_a^s$ whose existence are proven in \Cref{the: levin dim poly modules} and \Cref{the: partial dim poly} respectively are equal.
    Therefore, $M$ and $G$ have the same $\s$\=/dimension, $\s$\=/type and typical $\s$\=/dimension.
\end{theorem}
\begin{proof}
Using the constructions above, for each $i\in\mathbb{N}$, $\dim_k(M_i)=\dim_k(N_i/H_i)=\dim_k(N_i)-\dim_k(K_i)$ by standard results on the dimension of vector spaces \cite[Section 6, page 77]{atiyah1969introduction}. Further, for each $i\in\mathbb{N}$, $\dim_k(N_i)=|T_\s[i](y)|$, and $\dim_k(K_i)$ is the number of $k$\=/linearly independent generators of $K_i$.
For each $i\in\mathbb{N}$, $\dim(G[i])$ is the Krull dimension of the $k$\=/algebra $k[\mathbb{G}_a^s[i]]/\I(G[i])$, where $\I(G[i])=\I(G)\cap k[\mathbb{G}_a^s[i]]$. Then $\dim(k[G[i]])=\dim(k[\mathbb{G}_a^s[i]])-\hgt(\I(G[i]))$.
Since for all $i\in\mathbb{N}$, $\I(G[i])=(K_i)$, the height of $\I(G[i])$ is the number of linearly independent linear relations that generate $K_i$. Hence $\hgt(\I(G[i]))=\dim_k(K_i)$, and since $\dim(k[\mathbb{G}_a^s[i]])=|T_\s[i](y)|$, for all $i\in\mathbb{N}$, $\dim(G[i])=\dim_k(M_i)$.

Therefore, the dimension polynomials $\phi(t)$ and $\psi(t)$ for the standard excellent filtration $(M_i)_{i\in\mathbb{Z}}$ of $M$ and the Zariski closures $(G[i])_{i\in\mathbb{N}}$ of $G$ with respect to $\G$ respectively are two polynomials such that $\phi(i)=\psi(i)$ for all large enough $i\in\mathbb{N}$, and hence are the same polynomial. That is, $\phi(t)$ and $\psi(t)$ have the same degree and leading coefficient, and hence $M$ and $G$ have the same $\s$\=/dimension, $\s$\=/type and typical $\s$\=/dimension respectively.
\end{proof}

Due to the correspondence between finitely $\s$\=/generated $\s$\=/$k$\=/modules and additive $\s$\=/algebraic groups over $k$ described above, the existence of dimension polynomials for $\s$\=/$k$\=/modules proven in \cite[Theorem 3.2.3, page 161]{levin2008difference} can be considered to be a result on additive $\s$\=/algebraic groups. Therefore, \Cref{the: partial dim poly} in this paper is a generalisation of this result, extending it to hold for any $\s$\=/algebraic group over a $\s$\=/field $k$.

\begin{example}
    Consider \cite[Example 3.2.5, page 162]{levin2008difference}, where $k$ is a $\s$\=/field and $\D$ is the ring of $\s$\=/operators over $k$, which can be considered as a $\s$\=/$k$\=/module with excellent filtration $(\D_i)_{i\in\mathbb{Z}}$, the standard filtration on $\D$. This reference finds that the dimension polynomial for $\D$ is $\phi(t)=\binom{t+n}{n}$. We can see that $\D$ is generated as a $\s$\=/$k$\=/module by just one free $\s$\=/generator, hence $\D$ corresponds to the $\s$\=/algebraic subgroup of $\mathbb{G}_a$ defined by $\s$\=/Hopf ideal $(0)\subseteq k\{y\}$. That is, the $\s$\=/$k$\=/module $\D$ corresponds to the additive group $\mathbb{G}_a$ over $k$, and further, the excellent filtration of $\D$ corresponding to the Zariski closures of $\mathbb{G}_a$ in itself is exactly the standard filtration of $\D$. 
    Noting that $\dim(\mathbb{G}_a)=1$, \Cref{ex: dim of G[i]} tells us that the dimension polynomial for the Zariski closures of $\mathbb{G}_a$ in itself is $\psi(t)=\binom{t+n}{n}$, and $\phi=\psi$ as required. In particular, we see that the $\s$\=/dimension, $\s$\=/type and typical $\s$\=/dimensions of both $\D$ as a $\s$\=/$k$\=/module and $\mathbb{G}_a$ as a $\s$\=/algebraic group are $1$, $n$ and $1$ respectively.
\end{example}

\subsection{Relating Difference Field Extensions to Difference Algebraic Groups}

In this section, we will show that given a difference algebraic group $G$ over a $\s$\=/field $k$ (where $\s_j\colon k\{G\}\rightarrow k\{G\}$ is injective for each $1\leq j\leq n$) such that $k\{G\}$ is an integral domain, we can find a corresponding $\s$\=/field extension $k\langle G\rangle$ of $k$. Then we will prove that $k\langle G\rangle$ and $G$ have the same $\s$\=/dimension, $\s$\=/type and typical $\s$\=/dimension respectively. Firstly we recall some concepts and results from \cite[Section 4]{levin2008difference}.

Let $k$ be a field and let $L$ be an overfield of $k$. Recall that a set $B\subseteq L$ is called a transcendence basis of $L$ over $k$ if $B$ is algebraically independent over $k$ and $L$ is algebraic over $k(B)$ (the smallest subfield of $L$ containing both $k$ and $B$). Any two transcendence bases of $L$ over $k$ have the same cardinality, called the transcendence degree of $L$ over $k$, which we denote $trdeg_kL$.

For a $\s$\=/field $k$ and a $\s$\=/overfield $L$, given a subset $A\subseteq L$, denote by $k\langle A\rangle $ the smallest $\s$\=/field containing both $k$ and $A$. We call $k\langle A\rangle$ the \textbf{difference field extension} ($\s$\=/field extension) of $k$ generated by $A$. We say that a $\s$\=/field extension $L$ of $k$ is \textbf{finitely $\s$\=/generated} if $L=k\langle A\rangle$ for some finite set $A\subseteq L$. We recall \cite[Theorem 4.2.1, page 225]{levin2008difference} proving the existence of dimension polynomials for finitely difference generated difference field extensions.

\begin{theorem}\label[theorem]{the: dim poly field ext}
    Let $k$ be a $\s$\=/field, and let $k\langle A\rangle$ be the $\s$\=/field extension of $k$ $\s$\=/generated by a finite set $A=\{a_1,\dots,a_s\}$ for some $s\geq 1$. Then there exists a numerical polynomial $\phi(t)\in\mathbb{Q}[t]$ of degree at most $n$ such that for large enough $i\in\mathbb{N}$, $\phi(i)=trdeg_kk(\{\tau a_j\ |\ \tau\in T_\s[i],\ 1\leq j\leq s\})$.
\end{theorem}

We know that for a finitely $\s$\=/generated $\s$\=/field extension $L=k\langle A\rangle$ of $k$, its dimension polynomial is of the form  $\psi(t)=\sum_{i=0}^dc_i\binom{t+i}{i}$, where $c_0,\dots,c_d\in\mathbb{Z}$, $c_d\neq 0$, and $d\leq n$. By \cite[Theorem 4.2.1, page 255]{levin2008difference}, the degree $d$ of the polynomial and the integer $c_d$ are invariant of the choice of $\s$\=/generators $A$.
The integers $d$ and $c_d$ are called the \textbf{difference type} and \textbf{typical difference dimension} of the $\s$\=/field extension $L$ over $k$ respectively. Further, the \textbf{difference dimension} of $L$ over $k$ is $c_d$ if $d=n$, and $0$ if $d<n$.

We will now see the partial analogue to \cite[Proposition 3.17, page 531]{wibmer2022finiteness}. Note that if a $k$\=/$\s$\=/algebra $R$ is an integral domain, and $\s_j\colon R\rightarrow R$ is injective for each $1\leq j\leq n$, the field of fractions of $R$ is also a $k$\=/$\s$\=/algebra.

\begin{theorem}
    Let $G$ be a $\s$\=/closed subgroup of an algebraic group $\G$ over a $\s$\=/field $k$ such that for each $1\leq j\leq n$, $\s_j\colon  k\{G\}\rightarrow k\{G\}$ is injective. Suppose that the $k$\=/$\s$\=/Hopf algebra $k\{G\}$ is an integral domain, and therefore its field of fractions $k\langle G\rangle$ is a $\s$\=/field extension of $k$. There exists a finite $\s$\=/generating set $A$ of the $\s$\=/field extension $k\langle G\rangle$ such that the dimension polynomials for the $\s$\=/field extension $k\langle A\rangle$ of $k$ and the Zariski closures $(G[i])_{i\in\mathbb{N}}$ of $G$ with respect to $\G$ whose existence are proven in \Cref{the: dim poly field ext} and \Cref{the: partial dim poly} respectively are equal.
    Therefore, $k\langle G\rangle$ and $G$ have the same $\s$\=/dimension, $\s$\=/type and typical $\s$\=/dimension respectively.
\end{theorem}
\begin{proof}
Let $A\subseteq k\{G\}$ be a finite generating set of the $k$\=/algebra $k[G[0]]$. By \Cref{rem: zariski closures as image}, $k\{G\}$ is generated as a $k$\=/$\s$\=/algebra by $A$, and hence $k\langle G\rangle $ is $\s$\=/generated as a $\s$\=/field extension of $k$ by $A$.
Let $i\in\mathbb{N}$, and put $T_\s[i](A)=\{\tau(a)\ |\ \tau\in T_\s[i], a\in A\}$. Notice (again by \Cref{rem: zariski closures as image}) that $k[G[i]]=k[T_\s[i](A)]$, and hence $k(T_\s[i](A))$ is the field of fractions of $k[G[i]]$. By \cite[Section 13.1, page 286]{eisenbud1995commutative}, the Krull dimension of an integral domain is equal to the transcendence degree of its field of fractions, that is, $\dim(G[i])$ is equal to $trdeg_kk(T_\s[i](A))$ for each $i\in\mathbb{N}$.
    Therefore, the dimension polynomials $\phi(t)$ and $\psi(t)$ for the $\s$\=/field extension $k\langle G\rangle=k\langle A\rangle$ of $k$ and the Zariski closures $(G[i])_{i\in\mathbb{N}}$ of $G$ with respect to $\G$ respectively are polynomials such that $\phi(i)=\psi(i)$ for large enough $i\in\mathbb{N}$, and hence are the same polynomial. That is, $\phi(t)$ and $\psi(t)$ have the same degree and leading coefficient, and hence $k\langle G\rangle$ and $G$ have the same $\s$\=/dimension, $\s$\=/type and typical $\s$\=/dimension respectively.
\end{proof}


\textbf{Acknowledgements}: The author would like to thank her PhD supervisor Dr Michael Wibmer for his guidance and feedback on the content of this paper, and the anonymous referee for their helpful comments and suggestions. 
\printbibliography
\end{document}